\documentclass[11pt]{amsart}
\usepackage[utf8]{inputenc} 	% LaTeX, comprend les accents !
\usepackage[T1]{fontenc}
\usepackage{lmodern}
\usepackage{yhmath}

\usepackage{amsthm,amsmath,amssymb,amsbsy,bbm,mathrsfs,supertabular,
eurosym,graphicx,enumitem,xcolor}
\usepackage{mathtools}
\usepackage{leftidx} % faire des indices/exposant \`a gauche

\newtheoremstyle{BBstyle0}  {}{}{\itshape}{}{\bfseries}{}{6pt}{}
\newtheoremstyle{BBstyle1}  {3pt}{3pt}{\rmfamily}{}{\itshape}{: }{3pt}{}
\newtheoremstyle{BBstyle2}  {3pt}{3pt}{\itshape}{}{\bfseries\large}{}{0pt}{}
\newtheoremstyle{BBstyle3}  {}{}{\itshape}{}{\bfseries}{: }{3pt}{}
\newtheoremstyle{BBstyle4}  {}{}{\rmfamily}{}{\bfseries}{}{6pt}{}
  
\usepackage[normalem]{ulem} %pour hachurer, souligner etc. 
\usepackage[authoryear]{natbib}
\newtheorem{thm}{Theorem}
\newtheorem{lem}{Lemma}
\newtheorem{prop}{Proposition}
\newtheorem{df}{Definition}
\newtheorem{cor}{Corollary}
\newtheorem{ass}{Assumption}

\theoremstyle{definition}
\newtheorem{exa}{Example}

\usepackage[english]{babel}

\usepackage{leftidx} % faire des indices/exposant \`a gauche

\usepackage{hyperref}  		     % doit imp\'erativement \^etre apr\`es "showkeys" !
\newcommand{\pa}[1]{\left({#1}\right)}
\newcommand{\norm}[1]{\left\|{#1}\right\|}
\newcommand{\cro}[1]{\left[{#1}\right]}
\newcommand{\ab}[1]{\left|{#1}\right|}
\newcommand{\ac}[1]{\left\{{#1}\right\}}

\newcommand{\argmin}{\mathop{\rm argmin}}

\newcommand{\Var}{\mathop{\rm Var}\nolimits}

\newcommand{\dfleche}[1]{\,\displaystyle{\mathop{\longrightarrow}_{#1}}\,}
\newcommand{\CP}[1]{\stackrel{\mathrm{P}}{\dfleche{#1}}}

\newcommand{\E}{{\mathbb{E}}}
 
\renewcommand{\L}{{\mathbb{L}}}
\newcommand{\N}{{\mathbb{N}}}
\renewcommand{\P}{{\mathbb{P}}}
 
\newcommand{\R}{{\mathbb{R}}}

\newcommand{\Z}{{\mathbb{Z}}}

\newcommand{\sB}{{\mathscr{B}}}

\newcommand{\sD}{{\mathscr{D}}}

\newcommand{\sF}{{\mathscr{F}}}

\newcommand{\sL}{{\mathscr{L}}} 
\newcommand{\sM}{{\mathscr{M}}}

\newcommand{\sP}{{\mathscr{P}}}

\newcommand{\sT}{{\mathscr{T}}}

\DeclareMathAlphabet{\mathscrbf}{OMS}{mdugm}{b}{n}
\newcommand{\cA}{{\mathcal{A}}}
\newcommand{\cB}{{\mathcal{B}}}

\newcommand{\cE}{{\mathcal{E}}}
\newcommand{\cF}{{\mathcal{F}}}

\newcommand{\cK}{{\mathcal{K}}}
 
\newcommand{\cM}{{\mathcal{M}}}
\newcommand{\cN}{{\mathcal{N}}}

\newcommand{\cR}{{\mathcal{R}}}

\newcommand{\cU}{{\mathcal{U}}}
\newcommand{\cV}{{\mathcal{V}}}

\newcommand{\gj}{{\mathbf{j}}}

\newcommand{\gm}{{\mathbf{m}}}

\newcommand{\gp}{{\mathbf{p}}}

\newcommand{\gx}{{\bs{x}}}

\newcommand{\gJ}{{\mathbf{J}}}

\newcommand{\gL}{{\mathbf{L}}} 
\newcommand{\gM}{{\mathbf{M}}}

\newcommand{\gP}{{\mathbf{P}}}

\newcommand{\gT}{{\mathbf{T}}}

\newcommand{\bs}[1]{\boldsymbol{#1}}

\newcommand{\bsX}{{\bs{X}}}

\newcommand{\gtheta}{{\bs{\theta}}}

\newcommand{\gTheta}{{\Theta}}

\newcommand{\es}{{s}}

\newlist{lista}{enumerate}{2}
\setlist[lista,1]{label=\alph*),ref=\alph*)}

\newlist{listi}{enumerate}{1}
\setlist[listi,1]{label=(\roman*),ref=(\roman*),align=left}
\newlist{lists}{enumerate}{3}
\setlist[lists,1]{label=$\blacktriangleright$,align=right}
\setlist[lists,2]{label=$\bullet$,align=right}
\setlist[lists,3]{label=$\star$,align=right}

\newcommand{\eref}[1]{(\ref{#1})}

\renewcommand{\ge}{\geqslant}
\renewcommand{\le}{\leqslant}
\newcommand{\1}{1\hskip-2.6pt{\rm l}}
\newcommand{\0}{{\bf 0}}

\newcommand{\scal}[2]{\langle #1,#2\rangle}

\newcommand{\etc}[1]{#1_1,\ldots,#1_n}

\newcommand{\dps}[1]{\displaystyle{#1}}

\newcommand{\on}{^{\otimes n}}

\newcommand{\et}{^{\star}}
\renewcommand{\atop}[2]{\substack{#1\\#2}}
\newcommand{\card}[1]{{\ab{#1}}}
\newcommand{\eps}{{\varepsilon}}

\def\vide{{\varnothing}}

\newcommand{\map}[5]{
\begin{array}{l|rcl}
#1: & #2 & \longrightarrow & #3 \\
    & #4 & \longmapsto & #5 
\end{array}
    }

\newcommand{\co}[1]{\leftidx{^\mathsf{c}}{\!{#1}}{}} % complementaire d'un ensemble

\usepackage{booktabs}
\usepackage{footnote}
\usepackage{enumitem}

\def\bst{\bs{\theta}}

\def\ray{{r}}
\def\cc{{c}}
\def\PAP{{\overline Q}}

\begin{document}
\title{From robust tests to Bayes-like posterior distributions}
\author{Yannick BARAUD}
\address{\parbox{\linewidth}{Department of Mathematics,\\
University of Luxembourg\\
Maison du nombre\\
6 avenue de la Fonte\\
L-4364 Esch-sur-Alzette\\
Grand Duchy of Luxembourg}}
\email{yannick.baraud@uni.lu}
\keywords{Bayes procedure -- Gibbs estimator -- Posterior distribution -- Robustness -- Hellinger distance -- Total variation distance.}
\subjclass{Primary 62G05, 62G35, 62F35, 62F15}
\thanks{This project has received funding from the European Union's Horizon 2020 research and innovation programme under grant agreement N\textsuperscript{o} 811017}
\date{\today}

\begin{abstract}
In the Bayes paradigm, given a loss function and an $n$-sample, we present the construction of a new type of posterior distribution, that extends the classical Bayes one. The loss functions we have in mind are either those derived from the total variation and Hellinger distances or some $\L_{j}$-ones for $j>1$. We prove that, with a probability close to one, this new posterior distribution concentrates its mass in a neighbourhood (for the chosen loss function) of the law of the data, provided that this law belongs to the support of the prior or, at least, lies close enough to it. We therefore establish that the new posterior distribution enjoys some robustness properties with respect to a possible misspecification of the prior, or more precisely, its support. 
We also show that the posterior distribution is stable with respect to the equidistribution assumption we started from. Besides, when the model is regular and well-specified and one uses the squared Hellinger loss, we show that our credible regions possess, at least for $n$ sufficiently large, the same ellipsoidal shapes and approximately the same sizes as those we would derive from the classical Bayesian posterior distribution by using the Bernstein--von Mises theorem. Then we use our Bayesian-like approach to solve the following problems.  We first consider the estimation of a location parameter or both the location and scale parameters of a density in a nonparametric framework. Then we tackle the problem of estimating a density, with the squared Hellinger loss, in a high-dimensional parametric model under some sparsity conditions on the parameter. Importantly, the results established in this paper are nonasymptotic and provide, as much as possible, bounds with explicit constants.
\end{abstract}

\maketitle

\section{Introduction}
Observe $n$ i.i.d.\ random variables $\etc{X}$ with values in a measurable space $(E,\cE)$ and assume that their common distribution $P\et$ belongs to a family $\sM$ of candidate probabilities, or at least lies close enough to it in a suitable sense. We consider the problem of estimating $P\et$ from the observation of $\bsX=(\etc{X})$ and we evaluate the performance of an estimator with values in $\sM$ by means of a given loss function $\ell:\sP\times \sM\to \R_{+}$, where $\sP$ denotes a set of probabilities containing $P\et$. 

%The loss functions we have in mind are based on the Hellinger or the total variation distances between distributions or the $\L_{j}$-ones between their densities (with respect to some reference measure). We recall that the total variation distance $\norm{P-Q}$ and the squared Hellinger one $h^{2}(P,Q)$ between two probabilities $P$ and $Q$ on $(E,\cE)$ are respectively given by the formulas 
%%
%%Equation
%\begin{equation}\label{def-VT}
%\norm{P-Q}=\frac{1}{2}\int_{E}\ab{\frac{dP}{d\mu}-\frac{dQ}{d\mu}}d\mu
%\end{equation}
%%
%%
%and 
%%
%\begin{equation}\label{def-h2}
%h^{2}(P,Q)=\frac{1}{2}\int_{E}\pa{\sqrt{\frac{dP}{d\mu}}-\sqrt{\frac{dQ}{d\mu}}}^{2}d\mu,
%\end{equation}
%%
%where $\mu$ dominates both $P$ and $Q$. The result is independent of the choice of $\mu$. 

Our approach to solve this estimation problem has a Bayesian flavour. We endow $\sM$ with a $\sigma$-algebra $\cA$ and a probability measure $\pi$ that plays the same role as the prior in the classical Bayes paradigm. Our aim is to design a posterior distribution $\widehat \pi_{\bsX}$, solely based on $\bsX$ and the choice of $\ell$, that concentrates its mass, with a probability close to one, on an {\em $\ell$-ball,} namely a set of the form
%beg
\begin{equation}
\sB(P\et,r)=\ac{P\in\sM,\; \ell(P\et,P)\le r}\quad \text{with}\quad r>0.
\
\label{eq-boules}
\end{equation}
This means that with a probability close to 1, a point $\widehat P$ which is randomly drawn according to our (random) distribution $\widehat \pi_{\bsX}$ is likely to estimate $P\et$ with an accuracy (with respect to the chosen loss $\ell$) not larger than $r$. Our objective is to design $\widehat \pi_{\bsX}$ in such a way that this concentration property holds for small values of $r$ and under mild assumptions on $P\et$ and $\sM$. 

In the literature, many authors have studied the concentration properties of the classical Bayes posterior distribution on Hellinger balls. We refer to the pioneering papers by van der Vaart and his co-authors --- see for example Ghosal, Ghosh and van der Vaart~\citeyearpar{MR1790007}. They show that the concentration property around $P\et$ holds, as $n$ tends to infinity, provided that the prior $\pi$ puts enough mass on sets of the form $\cK(P\et,\eps)=\{P\in\sM,\; K(P\et,P)<\eps\}$ where $\eps$ is a positive number and $K(P\et,P)$ the Kullback--Leibler divergence between $P\et$ and $P$. This assumption may, however, be quite restrictive even in the favorable situation where $P\et$ belongs to the model $\sM$. Such sets may indeed be empty, and the condition therefore unsatisfied, when the probabilities in $\sM$ are not equivalent. This is for example the case when $\sM$ is the set of all uniform distributions $P_{\theta}$ on $[\theta-1/2,\theta+1/2]$, with $\theta\in\R$, although the problem of estimating $P\et\in\sM$ in this setting is quite easy, even in the Bayesian paradigm. The assumption appears even more restrictive when the probability $P\et$ does not belong to $\sM$, that is when the model is misspecified. 
For example, if the distributions in $\sM$ are all equivalent and $R$ is singular with respect to $\overline P\in\sM$,  $\cK(P\et,\eps)$ is empty for $P\et=(1-10^{-10})\overline P +10^{-10}R$ although $P\et$ and $\overline P\in\sM$ are statistically indistinguishable from any $n$-sample of realistic size. 

Unfortunately, it is in general impossible to get rid of the restrictive conditions we have mentioned above. It is well known that the Bayes posterior distribution can be unstable in case of a misspecification of the model. Examples that illustrate this weakness have been given in Jiang and Tanner~\citeyearpar{MR2458185} and Baraud and Birg\'e~\citeyearpar{BarBir2020} for instance. This instability is due to the fact that the Bayes posterior distribution is based on the log-likelihood function and similar issues are known for the maximum likelihood estimator. 

In order to obtain the concentration and stability properties we look for, we replace the log-likelihood function by a more stable one. Substituting another function to the log-likelihood one is not new in the literature and leads to what is called {\em quasi-posterior distributions}.  The resulting estimators, called {\em quasi-Bayesian estimators} or {\em Laplace type estimators}, have been studied by various statisticians among which Chernozhukov and Hong~\citeyearpar{MR1984779} and Bissiri~{\em et al.}~\citeyearpar{MR3557191} (we also refer to the references therein). These papers, however, do not address the problem of misspecification. In contrast, it is addressed in Jiang and Tanner~\citeyearpar{MR2458185} for performing variable selection in the logistic model. The authors show that the classical Bayesian approach is no longer reliable when the model is slightly misspecified while their Gibbs posterior distribution performs well and offers thus a much safer alternative. The problem of estimating a high-dimensional parameter $\theta\in\R^{d}$ under a sparsity condition was considered in Atchad\'e~\citeyearpar{MR3718168}. His quasi-posterior distribution is obtained by replacing the joint density  of the data by a more suitable one and by using some specific prior that forces sparsity. He proves that the so-defined posterior distribution contracts around the true parameter $\theta\et$ at rate $\sqrt{(s\et\log d)/n}$ (where $s\et$ is the number of nonzero coordinates of $\theta\et$) when both $d$ and $n$ tend to infinity. A common feature of the papers we have cited above lies in their asymptotic nature. This is not the case for Bhattacharya, Pati and Yang~\citeyearpar{MR3909926} who replaced the likelihood function in the expression of the posterior distribution by the {\em fractional likelihood}, that is  a suitable power of the likelihood function. The authors also consider the situation where the model is possibly misspecified but their result involves the $\alpha$-divergence which, as the Kullback one, can be infinite even when the true distribution of the data is close to the model for the total variation distance or the Hellinger one.  

Baraud and Birg\'e~\citeyearpar{BarBir2020} propose a surrogate to the Bayes posterior distribution that is called the {\em $\rho$-posterior distribution} in reference to the theory of $\rho$-estimation that was developed in Baraud {\em et al.}~\citeyearpar{MR3595933} and Baraud and Birg\'e~\citeyearpar{BarBir2018} in order to solve the various problems connected to the maximum likelihood method. The $\rho$-posterior distribution preserves some of the nice features of the classical Bayes one but also possesses the robustness property we are interested in. The authors show that their posterior distribution concentrates on a Hellinger ball around $P\et$ as soon as the prior puts enough mass around a point which is close enough to $P\et$. However their approach applies to specific dominated models $\sM=\{P=p\cdot \mu,\; p\in\cM\}$ only. They assume that the family $\cM$ of densities that defines their model possesses some special combinatorial structure which is either met when $\cM$ is finite or when it satisfies some VC-type condition (see their Section~5). As a consequence, the concentration radius they obtain not only depends on the choice of the prior but also on a complexity term that is linked to this structure. Unlike theirs, our approach makes no such assumptions on $\cM$ and we are therefore able to get rid of this unpleasant complexity term while retaining a similar dependency with respect to the choice of the prior. Baraud and Birg\'e's posterior distribution has also the drawback to involve the supremum over the family $\cM$ of an empirical process. Their posterior distribution is therefore difficult to calculate in practice, unless $\cM$ is finite with a reasonable size. From a more theoretical point of view, it also raises some unpleasant issues with regard to the measurability of this supremum in the typical situation where the family $\cM$ is uncountable, which is the typical case.  Finally, Baraud and Birg\'e's approach restricts to the squared Hellinger loss while ours applies to many others.

Closer to our approach are the aggregation methods and PAC-Bayesian techniques that have been popularized by Olivier Catoni in statistical learning (see Catoni~\citeyearpar{Catoni04}). This approach has mainly been applied for the purpose of empirical risk minimization and statistical learning (see for example Alquier~\citeyearpar{MR2483458}). Our aim is to extend these techniques toward a versatile tool that can solve our Bayes-like estimation problem for various loss functions simultaneously. 

The problem of designing a good estimator of $P\et$ for a given loss function $\ell$ was tackled in the frequentist paradigm in Baraud~\citeyearpar{BY-TEST}. There, the author provides a general framework that enables one to deal with various loss functions of interest, among which the total variation, 1-Wasserstein, Hellinger, and $\L_{j}$-losses among others. His approach relies on the construction of a suitable family of robust tests and lies in the line of the former work of Le Cam~\citeyearpar{MR0334381}, Birg\'e~\citeyearpar{MR722129} and Birg\'e~\citeyearpar{MR2219712}. The aim of the present paper is to transpose this theory from the frequentist to the Bayesian paradigm. If $\ell $ is the Kullback--Leibler divergence, our construction recovers the classical Bayes posterior distribution even though this is not the choice we would recommend for the reasons we have explained before. 

Quite surprisingly, the concentration properties that we establish here require almost no assumption on $\sM$ and the distribution of the data (apart from independence). They mostly depend on the choices of the prior $\pi$ and the loss function $\ell$. For a suitable element $P$ which belongs to the model $\sM$ and lies close enough to $P\et$, these concentration properties depend on the minimal value of the radius $r$ over which the log-ratio $V(P,r)=\log\cro{\pi(\sB(P,2r))/\pi(\sB(P,r))}$ (with $\sB$ defined in \eref{eq-boules}) increases at most linearly with  $r$. This log-ratio was introduced in  Birg\'e~\citeyearpar{MR3390134} for the purpose of analyzing the behaviour of the classical Bayes posterior distribution. In our Bayes-like paradigm, we show that the behaviour of the quantities $V(P,r)$ for $P\in\sM$ and $r>0$ completely encapsulates the complexity of the model $\sM$. We prove that our posterior distribution $\widehat \pi_{\bsX}$ concentrates on an $\ell$-ball centered at $P\et$ and the radius $r=r(n)$ of which is usually of minimax order as $n$ tends to infinity when the model is well-specified. From a nonasymptotic point of view, we show that $\widehat \pi_{\bsX}$ retains its nice concentration properties as long as $P\et$ remains close enough to an element $P$ in $\sM$ around which the prior puts enough mass, that is, even in the situation where the model is slightly misspecified.  Actually, we establish the stronger result that even when the data are only independent but not i.i.d., the above conclusion remains true for the average $\overline P\et$ of their marginal distributions in place of $P\et$. We therefore show that the posterior distribution $\widehat \pi_{\bsX}$ enjoys some robustness properties with respect to the equidistribution assumption we started from. The main theorems involve as much as possible explicit numerical constants. We illustrate our results with examples which are deliberately chosen to be as general and simple as possible. Our aim is to give a flavour of the results that can be established with our Bayes-like posterior, avoiding as much as possible the technicalities that would result from the choice of {\em ad-hoc} priors introduced to solve specific problems. Instead, we wish to discuss the optimality and robustness properties of our construction for solving general parametric and nonparametric estimation problems in the density framework under assumptions that we wish to be as weak as possible. These posterior distributions will therefore provide a benchmark for comparison with other methods. Their practical implementation will be the subject of future work.

Of special interest is the choice of $\ell$ given by the total variation distance or the Hellinger one.  As we shall see,  for such losses the stability of our posterior distribution automatically leads to estimators $\widehat P\sim \widehat \pi_{\bsX}$ that are naturally robust to  the presence of outliers or contaminating data among the sample. These results contrast sharply with the instability of the classical Bayes posterior distribution we underlined earlier. Nevertheless, our posterior distribution also shares some similarities with the classical Bayes one. When the model is well-specified and one uses the squared Hellinger loss, we show that the credible regions of our posterior distribution asymptotically possess the same ellipsoidal shapes and approximately the same sizes as the ones we derive from the classical Bayes posterior by means of the Bernstein--von Mises theorem. Establishing an analogue of this theorem for our Bayes-like posterior distribution is, however, beyond the scope of the present paper. 

Our paper is organized as follows. We present our statistical setting in Section~\ref{sct-setting}. We consider there independent but not necessarily i.i.d.\ data in order to analyse later on the behaviour of our posterior distribution with respect to a possible departure from equidistribution. The construction of the posterior distribution is described in Section~\ref{sect-CPD}. In this section, we also show how more classical constructions based on the likelihood or the fractional likelihoods are particular cases of ours. We complete this section with some heuristics which, we hope, help understanding the main ideas of our approach. In particular, we bridge there the problem of designing robust posterior distributions to that of testing between two disjoint $\ell$-balls.  Section~\ref{sect-main} is devoted to the main theorems. We describe there the concentration properties of our posterior distribution. The applications of these results to classical loss functions are presented in Section~\ref{sect-ACLF}. We put a special emphasis on the cases of the total variation distance and the squared Hellinger loss. In the remaining part of the paper, we only focus on these two losses. In Section~\ref{sect-ctcbato} we highlight some similarities and differences between the classical Bayes posterior and ours for the squared Hellinger loss. In Section~\ref{sct-app} we explain how our posterior distribution can be used to solve the problem of estimating a density, or a parameter associated with it, in several statistical frameworks of interest. We discuss there how the concentration properties of our posterior distribution deteriorate in the case of a misspecification of the model by the prior. We also consider the problems of estimating a density in a location-scale family and a high-dimensional parameter in a parametric model under a sparsity constraint. We also show how our estimation strategy leads to unusual rates of convergence for estimating a translation parameter in a non-regular statistical model. In Section~\ref{sect-copaic}, we show how to evaluate the concentration radius of our posterior distributions in the parametric framework. Finally, Section~\ref{sct-PMR} is devoted to the proofs of the main theorems and Section~\ref{sct-OP} to the other proofs. 

\section{The statistical setting}\label{sct-setting}
Let $\bsX=(\etc{X})$ be an $n$-tuple of independent random variables with values in a measurable space $(E,\cE)$ and joint distribution $\gP\et=\bigotimes_{i=1}^{n}P_{i}\et$. 
Even though this might not be true, we pretend that the $X_{i}$ are i.i.d.\  and our aim is to estimate their (presumed) common distribution $P\et$ from the observation of $\bsX$. To do so, we introduce a family $\sM$ that consists of candidate probabilities (or merely finite signed measures in the case of the $\L_{j}$-loss). The reason for considering finite signed measures lies in the fact that statisticians sometimes estimate probability densities by integrable functions that are not necessarily densities but elements of a suitable linear space for instance (think of the case of projection estimators). We endow $\sM$ with a $\sigma$-algebra $\cA$ and a probability measure $\pi$, that we call {\em a prior} by analogy to the classical Bayesian framework, and we refer to the resulting pair $(\sM,\pi)$ as our {\em model}. The model $(\sM,\pi)$ plays here a similar role as in the classical Bayes paradigm. It encapsulates the {\em a priori} information that the statistician has on $P\et$. Nevertheless, we do not assume that $P\et$, if it ever exists, belongs to $\sM$ nor that the true marginals $P_{i}\et$ do. We rather assume that the model $(\sM,\pi)$ is approximately correct in the sense that the average distribution
\[
\overline P\et=\frac{1}{n}\sum_{i=1}^{n}P_{i}\et
\]
is close enough to some point $P$ in $\sM$ around which the prior $\pi$ puts enough mass. We assume that $\overline P\et$ belongs to a given set $\sP$ of probability measures on $(E,\cE)$ and we measure the estimation accuracy by means of a loss function $\ell:(\sM\cup\sP)\times \sM\to \R_{+}$ which is not identical to 0 in order to avoid trivialities. Even though $\ell$ may not be a genuine distance in general, we assume that it shares some similar features and we interpret it as if it were. For this reason, we call $\ell$-{\em ball} (or {\em ball} for short)  centered at $P\in \sP\cup \sM$ with radius $r>0$ the subset of $\sM$ 
\[
\sB(P,r)=\ac{Q\in\sM,\; \ell(P,Q)\le r}.
\]
Our aim is to built {\em a posterior distribution} (or posterior for short) $\widehat \pi_{\bsX}$ on $(\sM,\cA)$, depending on our observation $\bsX$, which concentrates with a probability close to 1 on an $\ell$-ball of the form $\sB(\overline P\et,r_{n})$ where we wish the value of $r_{n}>0$ to be small. 

\subsection{The special case of parametrized models}\label{sect-scpm}
In many situations we consider statistical models $\sM=\{P_{\theta},\; \theta\in \Theta\}$ which are parametrized via a one-to-one mapping $\theta\mapsto P_{\theta}$. When $(\Theta, \frak{B},\nu)$ is a measurable space, we endow $\sM$ with the $\sigma$-algebra $\cA=\{A,\; \{\theta\in \Theta,\; P_{\theta}\in A\}\in \frak{B}\}$. This choice possesses several advantages. First, the mapping $\theta\mapsto P_{\theta}$ is measurable from $(\Theta,\frak{B})$ onto $(\sM,\cA)$ and we may therefore define the prior $\pi$ on $(\sM,\cA)$ as the image of $\nu$ by this mapping. Besides, 
a function $F$ is measurable on $(\sM,\cA)$ if and only if the mapping $\theta\mapsto F\circ P_{\theta}$ is measurable on $(\Theta,\frak{B})$. This property makes the measurability of $F$ easier to check in general. In particular, the mapping $F:P_{\theta}\mapsto \theta$ is measurable on $(\sM,\cA)$ because $\theta\mapsto F\circ P_{\theta}=\theta$ is measurable on $(\Theta,\frak{B})$ and we may then define a posterior $\widehat \nu_{\bsX}$ on $(\Theta,\frak{B})$ as the image by $F$ of our posterior $\widehat \pi_{\bsX}$ on $(\sM,\cA)$.  By definition of $\widehat \nu_{\bsX}$, for all $\theta\in \Theta$ and $r>0$ 
%Equation
\begin{equation}\label{eq-lpara}
\widehat \pi_{\bsX}\pa{\sB(P_{\theta},r)}=\widehat \nu_{\bsX}\pa{\ac{\theta'\in \Theta,\; \ell(\theta,\theta')\le r}}
\end{equation}
where $\ell(\theta,\theta')$ denotes, slightly abusively, $\ell(P_{\theta},P_{\theta'})$ for $\theta,\theta'\in \Theta$. The concentration of $\widehat \pi_{\bsX}$ on an $\ell$-ball centered at $P_{\theta}$ with radius $r>0$ is then equivalent to the concentration of 
$\widehat \nu_{\bsX}$ on the set $\{\theta'\in \Theta,\; \ell(\theta,\theta')\le r\}$. Every time we consider a parametrized model, we assume that it is identifiable and implicitly use the construction that we presented above as well as its consequences.

\subsection{Notation and conventions}\label{sect-nc}
Throughout this paper, we use the following notation and conventions. For $a,b\in\R$, $a\vee b $ and $a\wedge b$ denote $\min\{a,b\}$ and $\max\{a,b\}$ respectively. For $x\in\R$, $(x)_{+}=x\vee 0$ while $(x)_{-}=(-x)\vee 0$. The Euclidean spaces $\R^{k}$ with $k\ge 1$ are equipped with their Borel $\sigma$-algebras. The cardinality of a set $A$ is denoted $|A|$ and its complement $\co{A}$. In particular,  for $P\in \sP\cup \sM$ and $r>0$, $\co{\sB}(P,r)=\ac{Q\in\sM,\; \ell(P,Q)> r}$. The elements of $\R^{k}$ with $k>1$ are denoted with bold letters, e.g.\ $\gx=(x_{1},\ldots,x_{k})$ and $\bs{0}=(0,\ldots,0)$. For $\gx\in\R^{k}$, $|\gx|_{\infty}=\max_{i\in\{1,\ldots,k\}}|x_{i}|$ while $\ab{\gx}$ denotes the Euclidean norm of $\gx$. The inner product of $\R^{k}$ is denoted by $\scal{\cdot}{\cdot}$ and the closed Euclidean ball centered at $\gx$ with radius $r\ge 0$ by $\cB(\gx,r)$. By convention $\inf_{\vide}=+\infty$ unless otherwise specified. We write $f\equiv c$ when a function $f$ is constant and equals $c$ on its domain. For all suitable functions $f$ on $(E^{n},\cE\on)$, $\E\cro{f(\bsX)}$ means $\int_{E^{n}}fd\gP\et$ while for $f$ on $(E,\cE)$, $\E_{S}\cro{f(X)}$ denotes the integral $\int_{E}fdS$ with respect to the measure $S$ on $(E,\cE)$. For $j\in [1,+\infty)$, we denote by $\sL_{j}(E,\cE,\mu)$, the set of measurable functions $f$ on $(E,\cE)$ such that $\norm{f}_{j,\mu}=[\int_{E}|f|^{j}d\mu]^{1/j}<+\infty$ while $\norm{f}_{\infty}=\sup_{x\in E}|f(x)|$ is the supremum norm of a function $f$ on $E$. If $\pi'$ is a distribution on $(\sM,\cA)$, $Q\sim \pi'$ means that $Q$ is a random variable with distribution $\pi'$. Finally, all the measures that we consider are implicitly assumed to be $\sigma$-finite.

\section{Construction of the posterior distribution}\label{sect-CPD}
Throughout this section, the model $(\sM,\pi)$ is assumed to be fixed. 

\subsection{The properties of our loss functions}
The construction of the posterior not only depends on the prior $\pi$ but also on the choice of the loss function. We first assume that $\ell$ satisfies some basic properties which are described below.
% HypothÃšse
\begin{ass}\label{Hypo-0}
For all $S\in\sP\cup\sM$, the mapping 
\[
\map{\ell(S,\cdot)}{(\sM,\cA)}{\R_{+}}{P}{\ell(S,P)}
\]
is measurable.
\end{ass}
Under such an assumption, $\ell$-balls are measurable and the quantities $\pi(\sB(P,r))$ for $P\in\sP\cup\sM$ and $r>0$ are therefore well-defined. 

\begin{ass}\label{Hypo-0bis}
There exists a positive number $\tau$ such that, for all $S\in\sP$ and $P,Q\in\sM$,
%Align
\begin{align}
\ell(S,Q)&\le \tau\cro{\ell(S,P)+\ell(P,Q)}\label{eq-triangle}\\
\ell(S,Q)&\ge \tau^{-1}\ell(P,Q)-\ell(S,P).\label{eq-triangleb}
\end{align}
\end{ass}
When $\ell$ is a genuine distance, inequalities~\eref{eq-triangle} and~\eref{eq-triangleb} are satisfied with $\tau=1$ since they correspond to the triangle inequality. When $\ell$ is the square of a distance, these inequalities are satisfied with $\tau=2$. 

Importantly, we assume that $\ell$ is associated with a family $\sT(\ell,\sM)=\ac{t_{(P,Q)},\; (P,Q)\in\sM^{2}}$ of test statistics on $(E,\cE)$ which possesses the properties below. We shall see in Section~\ref{sect-ACLF} that many classical loss functions (among which the total variation distance, the squared Hellinger distance, etc.) can be associated with families $\sT(\ell,\sM)$ satisfying the following assumptions. 
%
% ASSUMPTION
\begin{ass}\label{Hypo-1}
The elements $t_{(P,Q)}$ of $\sT(\ell,\sM)$ satisfy: 
\begin{listi}
\item\label{cond-1} 
The mapping 
\[
\map{t}{(E\times\sM\times\sM,\cE\otimes\cA\otimes\cA)}{\R}{(x,P,Q)}{t_{(P,Q)}(x)}
\]
is measurable. 

\item\label{cond-1b} For all $P,Q\in\sM$, $t_{(P,Q)}=-t_{(Q,P)}$.

\item\label{cond-2}  there exist positive numbers $a_{0},a_{1}$ such that, for all $S\in\sP$ and $P,Q\in\sM$,
%
%beg
\begin{equation}
\E_{S}\cro{t_{(P,Q)}(X)}\le a_{0}\ell(S,P)-a_{1}\ell(S,Q).
\label{eq-SPQ}
\end{equation}
%end
%
\item\label{cond-4} For all $P,Q\in\sM$,
\[
\sup_{x\in E}t_{(P,Q)}(x)-\inf_{x\in E}t_{(P,Q)}(x)\le 1.
\]
\end{listi}
\end{ass}
Under assumption~\ref{cond-1b}, $t_{(P,P)}=0$ and we deduce from~\eref{eq-SPQ} that $ (a_{0}-a_{1})\ell(S,P)\ge 0$, hence that $a_{0}\ge a_{1}$ since $\ell$ is not constantly equal to 0. 

Some families $\sT(\ell,\sM)$ may satisfy the stronger 
% ASSUMPTION
\begin{ass}\label{Hypo-2}
Additionally to Assumption~\ref{Hypo-1}, there exists $a_{2}>0$ such that 
\begin{enumerate}[label=(\roman*)]
\addtocounter{enumi}{3}
\item \label{cond-3} for all {$S\in\sP$ and $P,Q\in\sM$},
\[
\Var_{S}\cro{t_{(P,Q)}(X)}\le a_{2}\cro{\ell(S,P)+\ell(S,Q)}.
\]
\end{enumerate}
\end{ass}

\subsection{Construction of the posterior}\label{subsect-MR}
Let $\sT(\ell,\sM)$ be a family of test statistics that satisfies our Assumption~\ref{Hypo-1} and let $\beta$ and $\lambda$ be two positive numbers such that
%beg
\begin{equation}
\lambda=(1+\cc)\beta\quad\text{with}\quad\cc>0 \:\text{ satisfying }\: \cc_{0}=(1+\cc)-\cc(a_{0}/a_{1})>0.
\label{eq-cc}
\end{equation}
%end
We set 
\[
\gT(\bsX,P,Q)=\sum_{i=1}^{n}t_{(P,Q)}(X_{i})\quad \text{for all $P,Q\in\sM$}
\]
and define $\widetilde \pi_{\bsX}(\cdot|P)$ as the probability on $(\sM,\cA)$ with density
\[
\frac{d\widetilde \pi_{\bsX}(\cdot|P)}{d\pi}:Q\mapsto\frac{\exp\cro{ \lambda \gT(\bsX,P,Q)}}{\int_{\sM}\exp\cro{\lambda  \gT(\bsX,P,Q)}d\pi(Q)}.
\]
Then, for $P\in\sM$ we set
%
%Align
\begin{align*}
\gT(\bsX,P)&=\int_{\sM}\gT(\bsX,P,Q)d\widetilde \pi_{\bsX}(Q|P)\\
&=\int_{\sM}\gT(\bsX,P,Q)\frac{\exp\cro{\lambda \gT(\bsX,P,Q)}}{\int_{\sM}\exp\cro{\lambda  \gT(\bsX,P,Q)}d\pi(Q)}d\pi(Q).
\end{align*}
Finally, we define $\widehat \pi_{\bsX}$ as the posterior distribution  on $(\sM,\cA)$ with density 
%Equation
\begin{equation}\label{def-piX}
\frac{d\widehat \pi_{\bsX}}{d\pi}:P\mapsto \frac{\exp\cro{-\beta \gT(\bsX,P)}}{\int_{\sM}\exp\cro{-\beta \gT(\bsX,P)}d\pi(P)}.
\end{equation}
Our Assumption~\ref{Hypo-1}-\ref{cond-1} ensures that $d\widetilde \pi_{\bsX}(\cdot|P)/d\pi$ is a measurable function of $(\bsX,P,Q)$ and $d\widehat \pi_{\bsX}/d\pi$ a measurable function of $(\bsX,P)$. 

The posterior distribution depends on our choice of $\beta$ and $\lambda$ (or equivalently $c$) even though we drop this dependency with the notation $\widehat \pi_{\bsX}$.

%Besides, this density is a bounded function of $P$ under Assumption~\ref{Hypo-1}-\ref{cond-4}.

\subsection{Monte Carlo computation of functions of the posterior}
Even though we focus on the concentration properties of the posterior $\widehat \pi_{\bsX}$, one may alternatively be interested in some estimators derived from it. For example, estimators of the form 
\[
I=\int_{\sM}F(P)d\widehat \pi_{\bsX}(P)
\]
where $F$ is a real-valued $\pi$-integrable function on $(\sM,\cA)$. For typical choices of $F$, $I$ gives the expected mean, mode or median of the posterior whenever these quantities make sense. One may also choose $F:P\mapsto \1_{P\in\sB(P_{0},\eps)}$ with $P_{0}\in\sM$ and $\eps>0$ in order to compute the (posterior) probability that $\ell(P_{0}, \widehat P)$ is not larger than $\eps$ when $\widehat P\sim \widehat \pi_{\bsX}$. 

Interestingly, the integral $I$ can be approximated by Monte Carlo as follows. Assume that the prior $\pi$ admits a density of the form $C^{-1}\Pi$ with respect to a given probability measure $\frak{m}$, where $\Pi$ is a nonnegative $\frak{m}$-integrable function on $(\sM,\cA)$ and  $C=\int_{\sM}\Pi(P)d\frak{m}(P)>0$ a positive normalizing constant (that will not be involved in our calculation). Let $P_{1},\ldots,P_{N}$ be an $N$-sample with distribution $\frak{m}$ and for each $i\in\{1,\ldots,N\}$, $Q_{i}^{(1)},\ldots,Q_{i}^{(N')}$ an independent $N'$-sample with the same distribution. We may approximate $I$ by
\[
\widehat I_{N,N'}=\sum_{i=1}^{N}F(P_{i})\frac{\exp\cro{-\beta W_{i,N'}(P_{i})}\Pi(P_{i})}{\sum_{i'=1}^{N}\exp\cro{-\beta W_{i',N'}(P_{i'})}\Pi(P_{i'})}
\]
where for all $i\in\{1,\ldots,N\}$,
\[
W_{i,N'}(P_{i})=\sum_{j=1}^{N'}T(\bsX,P_{i},Q_{i}^{(j)})\frac{\exp\cro{\lambda T(\bsX,P_{i},Q_{i}^{(j)})}\Pi(Q_{i}^{(j)})}{\sum_{j'=1}^{N'}\exp\cro{\lambda T(\bsX,P_{i},Q_{i}^{(j')})}\Pi(Q_{i}^{(j')})}.
\]
It is then easy to check that, by the law of large numbers, 
\[
\lim_{N\to +\infty}\cro{\lim_{N'\to +\infty}\widehat I_{N,N'}}=I.
\]

\subsection{Connection with the classical Bayes posterior distribution}\label{sect-cwtcbpd}
The classical Bayes posterior turns out to be a particular case of the posterior-type ones introduced in Section~\ref{subsect-MR}. As we shall see now, they are associated with the Kullback--Leibler divergence loss. We recall that the Kullback--Leibler divergence $\ell(P,Q)=K(P,Q)$ between two probabilities $P,Q$ on $(E,\cE)$ is defined  by
\[
K(P,Q)=
\begin{cases}
\dps{\int_{E}\log\pa{\frac{dP}{dQ}}dP} &\text{when } P\ll Q; \\
+\infty & \text{otherwise.}
\end{cases}
\]
Let us consider now a family $\sM$ of probabilities that satisfy for some $a>0$ and suitable versions of their densities $dQ/dP$ the following inequalities:
%Equation
\begin{equation}\label{eq-rvborne}
e^{-a}\le \frac{dP}{dQ}(x)\le e^{a}\quad \text{for all $x\in E$ and $P,Q\in\sM$}.
\end{equation}
It follows from Baraud~\citeyearpar{BY-TEST}[Proposition~12] that the families of functions
%Equation
\begin{equation}\label{eq-famKull}
\sT(\ell,\sM)=\ac{t_{(P,Q)}=\frac{1}{2a}\log\pa{\frac{dQ}{dP}},\; P,Q\in\sM} 
\end{equation}
satisfies our Assumptions~\ref{Hypo-1} and~\ref{Hypo-2} with $a_{0}=a_{1}=1/(2a)$ and $a_{2}=2a/[\tanh(a/2)]$. Note that given $P,Q\in\sM$, $P\ne Q$, the test based on the sign of $t_{(P,Q)}$ is the classical likelihood ratio test between $P$ and $Q$.

If we apply the construction described in Section~\ref{subsect-MR} to the family $\sT(\ell,\sM)$ we obtain that for all $P,Q,P_{0}\in\sM$, 
\[
\gT(\bsX,P,Q)=\gT(\bsX,P_{0},Q)-\gT(\bsX,P_{0},P).
\]
For all $\lambda>0$, the density of $\widetilde \pi_{\bsX}(\cdot|P)$
%Align
\begin{align*}
Q\mapsto \frac{\exp\cro{ \lambda \gT(\bsX,P,Q)}}{\int_{\sM}\exp\cro{\lambda  \gT(\bsX,P,Q)}d\pi(Q)}=\frac{\exp\cro{ \lambda \gT(\bsX,P_{0},Q)}}{\int_{\sM}\exp\cro{\lambda  \gT(\bsX,P_{0},Q)}d\pi(Q)}
\end{align*}
is independent of $P$ and writing $\widetilde \pi_{\bsX}(\cdot)$ in place of $\widetilde \pi_{\bsX}(\cdot|P)$ we obtain that 
%Align
\begin{align*}
\gT(\bsX,P)&=\int_{\sM}\gT(\bsX,P,Q)d\widetilde \pi_{\bsX}(Q)\\
&=\int_{\sM}\gT(\bsX,P_{0},Q)d\widetilde \pi_{\bsX}(Q)-\gT(\bsX,P_{0},P)\\
&=C-\frac{1}{2a}\sum_{i=1}^{n}\log\pa{\frac{dP}{dP_{0}}}(X_{i})\text{ with }C=\int_{\sM}\gT(\bsX,P_{0},Q)d\widetilde \pi_{\bsX}(Q).
\end{align*}
Finally, the density of our posterior $\widehat \pi_{\bsX}$ at $P\in\sM$ is given by
%Align
\begin{align*}
\frac{d\widehat \pi_{\bsX}}{d\pi}(P)&=\frac{\exp\cro{-\beta \gT(\bsX,P)}}{\int_{\sM}\exp\cro{-\beta \gT(\bsX,P)}d\pi(P)}= \frac{\cro{\prod_{i=1}^{n}(dP/dP_{0})(X_{i})}^{\beta/(2a)}}{\int_{\sM}\cro{\prod_{i=1}^{n}(dP/dP_{0})(X_{i})}^{\beta/(2a)}d\pi(P)}.
\end{align*}
This is the density of the classical Bayes posterior when $\beta=2a$ while for other values of $\beta$ it is that of fractional Bayes ones. 

Nevertheless, in the present paper we restrict our study to loss functions that satisfy some triangle-type inequality -- see Assumption~\ref{Hypo-0bis}. This excludes the Kullback--Leibler divergence unless one is ready to make strong assumptions on the unknown distribution of the data, which we do not want to do here.   

\subsection{Some heuristics}
In this section, we present the basic ideas that underline our approach. In particular, we shall see how the estimation problem we want to solve is linked to the one of testing between two disjoint $\ell$-balls $\sB(P,r)$ and $\sB(Q,r)$ with $P,Q\in\sM$. 

In order to avoid unnecessary details, we assume here that we observe i.i.d.\ data $\etc{X}$ with distribution $P\et\in \sP$ and that we have at disposal a family $\sT(\ell,\sM)$ of functions  that satisfies our Assumption~\ref{Hypo-1}. In particular it follows from Assumption~\ref{Hypo-1}-\ref{cond-2} that
%Align
\begin{align*}
\E\cro{\frac{\gT(\bsX,P,Q)}{n}}=\frac{1}{n}\sum_{i=1}^{n}\E\cro{t_{(P,Q)}(X_{i})}\le a_{0}\ell(P\et,P)-a_{1}\ell(P\et,Q).
\end{align*}
The antisymmetric property required by Assumption~\ref{Hypo-1}-\ref{cond-1b} entails that 
\[
\gT(\bsX,P,Q)=-\gT(\bsX,Q,P)
\]
and leads to the lower bound 
%Align
\begin{align*}
\E\cro{\frac{\gT(\bsX,P,Q)}{n}}\ge a_{1}\ell(P\et,P)-a_{0}\ell(P\et,Q).
\end{align*}
Assuming for the sake of simplicity that $a_{0}=a_{1}=1$, these calculations show that $n^{-1}\gT(\bsX,P,Q)=n^{-1}\sum_{i=1}^{n}t_{(P,Q)}(X_{i})$ is an unbiased and consistent estimator of $\ell(P\et,P)-\ell(P\et,Q)$. In particular, if the two $\ell$-balls $\sB(P,r)$, $\sB(Q,r)$ are disjoint and $P\et$ belongs to one of them, the sign of $n^{-1}\gT(\bsX,P,Q)=n^{-1}\sum_{i=1}^{n}t_{(P,Q)}(X_{i})$ provides a consistent test for deciding which one contains $P\et$. In fact, the test does not depend on the value of $r$ and consequently chooses the element among $\{P,Q\}$ which is the closest to $P\et$ (with respect to $\ell$), at least when $n$ is large enough. As compared to the classical likelihood ratio test between $P$ and $Q$, this test has the advantage not to assume that $P\et$ is either $P$ or $Q$ but only that it lies in a small enough $\ell$-vicinity around one of these two probabilities. The test is said to be {\em robust} with respect to the model $\{P,Q\}$. Its nonasymptotic properties have been studied in Baraud~\citeyearpar{BY-TEST}.

Let us now explain how such families $\{\gT(\bsX,P,Q), (P,Q)\in\sM^{2}\}$ of test statistics can be used to build robust estimators and not only tests. In the frequentist paradigm, the construction of $\ell$-estimators is based on the following heuristics. If, with a probability close to 1,  $n^{-1}\gT(\bsX,P,Q)$ is close to its expectation $\ell(P\et,P)-\ell(P\et,Q)$ uniformly with respect to $(P,Q)\in\sM^{2}$ then $n^{-1}\gT'(\bsX,P)=\sup_{Q\in\sM}\cro{n^{-1}\gT(\bsX,P,Q)}$ is close to 
\[
\sup_{Q\in\sM}\cro{\ell(P\et,P)-\ell(P\et,Q)}=\ell(P\et,P)-\inf_{Q\in\sM}\ell(P\et,Q).
\]
We therefore expect that a minimizer over $\sM$ of the function $P\in\sM\mapsto n^{-1}\gT'(\bsX,P)$ be close to a minimizer over $\sM$ of the function $P\in\sM\mapsto \ell(P\et,P)-\inf_{Q\in\sM}\ell(P\et,Q)$, that is an element that minimizes the loss $\ell(P\et,P)$ among the probabilities $P\in\sM$.

In the Bayesian paradigm, we may argue in a similar way as follows. Replacing $n^{-1}\gT(\bsX,P,Q)$ by its expectation $\ell(P\et,P)-\ell(P\et,Q)$, as we did before, amounts to replacing $\gT(\bsX,\gP)$ by 
%Align
\begin{align*}
&\overline \gT(\bsX,\gP)\\
&=n\int_{\sM}\pa{\ell(P\et,P)-\ell(P\et,Q)}\frac{\exp\cro{n\lambda \pa{\ell(P\et,P)-\ell(P\et,Q)}}d\pi(Q)}{\int_{\sM}\exp\cro{n\lambda  \pa{\ell(P\et,P)-\ell(P\et,Q)}}d\pi(Q)}\\
&=n\ell(P\et,P)-n\int_{\sM}\ell(P\et,Q)\frac{\exp\cro{-n\lambda\ell(P\et,Q)}}{\int_{\sM}\exp\cro{-n\lambda\ell(P\et,Q)}d\pi(Q)}d\pi(Q).
\end{align*}
Note that the second term in the right-hand side does not depend on $P$. Consequently, replacing $\gT(\bsX,\gP)$ by $\overline \gT(\bsX,\gP)$ 
in the expression \eref{def-piX} of the density of $\widehat \pi_{\bsX}$ leads to the density    
%Align
\begin{align*}
P\mapsto  \frac{\exp\cro{-\beta \overline \gT(\bsX,P)}}{\int_{\sM}\exp\cro{-\beta \overline \gT(\bsX,P)}d\pi(P)}=\frac{\exp\cro{-n\beta \ell(P\et,P) }}{\int_{\sM}\exp\cro{-n\beta \ell(P\et,P)}d\pi(P)}.
\end{align*}
We recognize here the density of a Gibbs measure associated with the energy $\ell(P\et,P)$ at point $P\in\sM$ and inverse temperature $n\beta>0$. We know that when the temperature goes to 0 (or equivalently $n\beta$ to infinity), Gibbs measures concentrate their masses in vicinities of low energy points in $\sM$. In our case, these low energy points are those for which $\ell(P\et,P)$ is minimal.  

Similar ideas can be found in Catoni's work and more specifically in his construction of Gibbs estimators -- see Catoni~\citeyearpar{Catoni04}[Chapter 4]. There, Catoni shows how to aggregate a continuous family of estimators in order to minimize a risk. In the present paper, we do not aim at aggregating estimators but we use similar ideas and tools that are due to Catoni and his co-authors for the construction of our robust posterior distribution. 

\section{The main results}\label{sect-main}
\subsection{Linking the prior to the complexity of the model}\label{sect-iotp}
For $P\in\sM$ and $r>0$, we recall that 
\[
V(P,r)=\log\pa{\frac{\pi(\sB(P,2r)}{\pi(\sB(P,r)}}
\]
where we use the convention $a/0=+\infty$ for all $a\ge 0$. We said in the Introduction that such quantities
encapsulate in some sense the complexity of the model $(\sM,\pi)$ and we shall now explain why. If $\sM=\{P_{\gtheta},\; \gtheta\in \R^{k}\}$ is a parametric model endowed with a loss $\ell$ such that $\ell(\gtheta,\gtheta')=\ab{\gtheta-\gtheta'}$, so that $(\sM,\ell)$ is isometric to $(\R^{k},\ab{\cdot})$, and if the prior $\nu$ on $\Theta=\R^{k}$ is improper and given by the Lebesgue measure, we obtain that for all $P\in\sM$ and $r>0$
%Align
\begin{align}\label{eq-dim}
V(P,r)&=\log\pa{\frac{\pi(\sB(P,2r))}{\pi(\sB(P,r))}}=\log\pa{\frac{(2r)^{k}}{r^{k}}}=k\log 2
\end{align}
and $V(P,r)$ corresponds thus to the usual dimension of $\R^{k}$ (up the factor $\log 2$). For more general models $(\sM,\pi)$ and loss functions $\ell$, we may interpret $V(P,r)$ as some notion of dimension (or complexity) associated with the element $P\in\sM$ at the scale $r>0$. 
As we do not consider improper priors but probability distributions,  $\lim_{r\to +\infty}\pi(\sB(P,r))=1$ and consequently $\lim_{r\to +\infty}V(P,r)=0$. This means that the connection with the notion of ``dimension'' is only relevant for values of $r$ which are not too large.

Given $\gamma\in (0,1]$, the set 
\[
\cR(\beta,P)=\ac{\ray\ge \frac{1}{n\beta a_{1}},\text{ such that } \sup_{r'\ge r}\frac{V(P,r')}{r'}\le \gamma n\beta a_{1}}
\]
is the subinterval of $\R_{+}$ on which the mapping $r\mapsto V(P,r)$ is not larger than $r\mapsto \gamma n\beta a_{1}\ray$. We denote by 
%beg
\begin{equation}
\ray_{n}(\beta,P)=\inf\cR(\beta,P)
%\quad\text{with}\quad\sup\varnothing=1/(n\beta a_{1}).
\label{df-epsn}
\end{equation}
%end
the left endpoint of $\cR(\beta,P)$. Since $\cR(\beta,P)$ is increasing with $\beta$ with respect to set inclusion, $\ray_{n}(\beta,P)$ is a nonincreasing function of $\beta$. For example, in the ideal situation given in \eref{eq-dim} where $V(P,r)\equiv k\log 2$ with $k\log 2\ge 1$, $\ray_{n}(\beta,P)=(\gamma a_{1})^{-1}[k\log 2/(n\beta)]$. When the model $\sM=\{P_{\gtheta},\; \gtheta\in \Theta\}$ is parametric and the parameter space $\Theta$ is an open subset of $\R^{k}$ endowed with a prior $\nu$, we shall see in Section~\ref{sect-BrnA} that under suitable assumptions $\ray_{n}(\beta,P_{\gtheta})$ is indeed of order $k/(n\beta)$, at least for $n$ sufficiently large.

The Bayesian paradigm offers the possibility to favour some elements of $\sM$ as compared to others. The order of magnitude of $\ray_{n}(\beta,P)$ allows one to quantify how much the prior $\pi$ advantages or disadvantages $P\in\sM$. It follows from the definition of $\ray_{n}(\beta,P)$ that 
%Align\overline P\et
\begin{align}
0<\pi\pa{\sB(P,2\ray)}\le \exp\pa{\gamma n\beta a_{1}\ray}\pi\pa{\sB(P,\ray)}\quad\text{for all }\ray>\ray_{n}(\beta,P).
\label{prop-epsn}
\end{align}
Letting $\ray$ decrease to $\ray_{n}(\beta,P)$, we derive that \eref{prop-epsn} {also} holds for $\ray=\ray_{n}(\beta,P)$.  In particular,  $\pi\pa{\sB(P,r)}>0$ for $\ray=\ray_{n}(\beta,P)$. If the prior puts no mass on the $\ell$-ball $\sB(P,\ray)$, which clearly corresponds to a situation where the prior disadvantages $P$, $\ray_{n}(\beta,P)>\ray$ and $\ray_{n}(\beta,P)$ is therefore large if $\ray$ is large. In the opposite case, if the prior puts enough mass on  $\sB(P,\ray)$ in the sense that 
%Equation
\begin{equation}\label{eq-mass}
\pi\pa{\sB(P,\ray)}\ge \exp\pa{-\gamma n\beta a_{1}\ray},
\end{equation}
then for all $\ray'\ge \ray$, 
%Align
\begin{align*}
\pi\pa{\sB(P,\ray')}&\ge \exp\pa{-\gamma n\beta a_{1}\ray}\ge \exp\pa{-\gamma n\beta a_{1}\ray'}\\
&\ge  \exp\pa{-\gamma n\beta a_{1}\ray'}\pi\pa{\sB(P,2\ray')}
\end{align*}
hence, 
\[
\frac{\pi\pa{\sB(P,2\ray')}}{\pi\pa{\sB(P,\ray')}}\le \exp\pa{\gamma n\beta a_{1}\ray'}\quad \text{and $\ray_{n}(\beta,P)\le \ray$.}
\]
The quantity $\ray_{n}(\beta,P)$ is therefore small if $\ray$ is small. Although~\eref{eq-mass} is not equivalent to~\eref{prop-epsn} (it is actually stronger), the previous arguments provide a partial view on the relationship between $\pi$ and $\ray_{n}$ and conditions to decide whether $P$ is favoured by $\pi$ or not, according to the size of $\ray_{n}(\beta,P)$.

\subsection{A general result on the concentration property of the posterior distribution}\label{ssect-thm1}
According to the discussion of Section~\ref{sect-iotp}, we see that, when the set 
%Equation
\begin{equation}\label{def-Mbeta}
\sM(\beta)=\ac{P\in\sM,\; \ray_{n}(\beta,P)\le a_{1}^{-1}\beta}
\end{equation}
is nonempty, it contains the most favoured elements of the model $(\sM,\pi)$ at level $a_{1}^{-1}\beta$. Since $\ray_{n}(\beta,P)$ is nonincreasing with $\beta$, the set $\sM(\beta)$ is increasing with $\beta$ with respect to set inclusion. If $a_{1}^{-1}\beta\ge(n\beta a_{1})^{-1}$ or equivalently $\beta\ge 1/\sqrt{n}$, the set $\sM(\beta)$ can alternatively be defined from $V(P,r)$ as follows:
%Align
\begin{align}
\sM(\beta)&=\ac{P\in\sM,\; V(P,r)\le \gamma n\beta a_{1}r \text{ for all }\ray\ge a_{1}^{-1}\beta}.\label{def-Mbetab}
\end{align}
%
%It is sometimes easier to use this latter form for calculations. Under this form, we see that the set $\sM(\beta)$ is increasing with respect to $\beta$ for the set inclusion. 
This set plays a crucial role in our first result.

\begin{thm}\label{main1}
Assume that the model $(\sM,\pi)$ and the loss $\ell$ satisfy Assumptions~\ref{Hypo-0} and~\ref{Hypo-0bis} and the family $\sT(\ell,\sM)$ Assumption~\ref{Hypo-1}. Let $\gamma<(c_{0}\wedge\cc)/(2\tau)$ and $\beta\ge 1/\sqrt{n}$ be chosen in such a way that the set $\sM(\beta)$ defined by~\eref{def-Mbeta} is not empty. Then, the posterior $\widehat \pi_{\bsX}$ defined by~\eref{def-piX} possesses the following property. There exists $\kappa_{0}>0$ only depending on $\cc,\tau,\gamma$ and the ratio $a_{0}/a_{1}$ such that, for all $\xi>0$ and any distribution $\gP\et$ with marginals in $\sP$,  
%Align
\begin{align}
\E\cro{\widehat \pi_{\bsX}\pa{\co{\sB}(\overline P\et,\kappa_{0}\ray)}}\le 2e^{-\xi}\label{eq-thm01}
\end{align}
with 
%Equation
\begin{equation}\label{eq-main1-ray}
\ray=\inf_{P\in\sM(\beta)}\ell(\overline P\et,P)+\frac{1}{a_{1}}\pa{\beta+\frac{2\xi}{n\beta}}.
\end{equation}
In particular, 
%Align
\begin{align*}
\P\cro{\widehat \pi_{\bsX}\pa{\co{\sB}(\overline P\et,\kappa_{0}\ray)}\ge e^{-\xi/2}}\le 2e^{-\xi/2}.
\end{align*}
\end{thm}
The value of $\kappa_{0}$ is given by \eref{def-kappa0} in the proof. It only depends on the choice of the family $\sT(\ell,\sM)$ but not on the prior $\pi$. Hence, for a given family $\sT(\ell,\sM)$, $\kappa_{0}$ is a numerical constant. 

Let us now comment on Theorem~\ref{main1}. When $\etc{X}$ are truly i.i.d.\ with distribution $P\et$ and the prior puts enough mass around $P\et$, in the sense that $P\et\in\sM(\beta)$, then $r=a_{1}^{-1}[\beta+2\xi/(n\beta)]$ in \eref{eq-main1-ray}. When this ideal situation is not met, either because the data are not identically distributed or because $P\et$ does not belong to $\sM(\beta)$, $r$ increases by at most an additive term of order $\inf_{P\in\sM(\beta)}\ell(\overline P\et,P)$. When this approximation term remains small as compared to $a_{1}^{-1}\beta$, the value of $r$ does not deteriorate too much as compared to the previous situation. 

The value of $\ray$ given by~\eref{eq-main1-ray} depends on the choice of the parameter $\beta$. Since the set $\sM(\beta)$ is increasing (with respect to set inclusion) as $\beta$ gets larger, the two terms $\inf_{P\in\sM(\beta)}\ell(\overline P\et,P)$ and $a_{1}^{-1}\beta$ vary in opposite directions as $\beta$ increases. The set $\sM(\beta)$ must be large enough to provide a suitable approximation of $\overline P\et$ while $\beta$ must not be too large in order to keep $a_{1}^{-1}\beta$ to a reasonable size. Practically, we recommend to choose $\beta=\beta(\alpha)\ge 1/\sqrt{n}$ such that 
%Equation
\begin{equation}\label{def-betaquantile}
\pi\pa{\sM(\beta)}\ge 1-\alpha \quad \text{for $\alpha\in (0,1/10)$.}
\end{equation}
In Example~\ref{ex-euclidien} below and in Section~\ref{sect-trans00}, we give some examples of choices of $\beta$.
\begin{exa}\label{ex-euclidien}
Let $(\sM,\pi)$ be a model where the prior $\pi$ satisfies for some $k\ge 1$ and constants $0<A\le (2/e)B$,
%Equation
\begin{equation}\label{eq-ex-euclidien}
\pa{A\ray}^{k}\wedge 1\le \pi\pa{\sB(P,\ray)}\le \pa{B\ray}^{k}\wedge 1\quad \text{for all $P\in\sM$ and $\ray>0$.}
\end{equation}
This means that the prior $\pi$ behaves like the Lebesgue measure on an Euclidean space of dimension $k$ for small enough values of $r$. 
%This assumption is for example satisfied when the model $\sM$ is parametric, the loss function on the parameter space equivalent to some power of the Euclidean distance, in the sense that for some positive constants $\underline a,\overline a$
%%
%\[
%\underline a\ab{\gtheta-\gtheta'}^{s}\le \ell(\gtheta, \gtheta')\le \overline a\ab{\gtheta-\gtheta'}^{s}\quad \text{for all $\gtheta,\gtheta'\in \Theta$,}
%\]
%%
%and the prior on the parameter space bounded away from 0 and infinity.  
Then, 
\begin{equation}\label{eq-epsP}
V(P,r)=\log \frac{\pi\pa{\sB(P,2\ray)}}{ \pi\pa{\sB(P,\ray)}}\le  k\log \pa{\frac{2B}{A}}\quad\text{for all $P\in\sM$ and $r>0$}
\end{equation}
which implies that for all $P\in\sM$
%Equation
\begin{equation}\label{eq-controlpara}
r_{n}(P,\beta)\le \frac{k}{\gamma a_{1} n \beta}\log\pa{\frac{2B}{A}}.
\end{equation}
The right-hand side is not larger than $a_{1}^{-1}\beta$ for  
%
%Equation
\begin{equation}\label{eq-choixbeta1}
\beta=\sqrt{\frac{k\log(2B/A)}{\gamma n }}
\end{equation}
which is larger than $1/\sqrt{n}$ since $(2B/A)\ge e$ and $\gamma\in (0,1]$. For such a value of $\beta$, which does not depend on the distribution of the data, the element $P$ belongs to $\sM(\beta)$ by ~\eref{def-Mbetab}, and since $P$ is arbitrary we derive that $\sM(\beta)=\sM$. Applying Theorem~\ref{main1} we conclude that the distribution $\widehat \pi_{\bsX}$ concentrates on an $\ell$-ball centered at $\overline P\et$ with a radius $\ray$ of order 
%Equation
\begin{equation}\label{eq-ex-euclidien01}
\overline \ray_{n}=\inf_{P\in\sM}\ell(\overline P\et,P)+\frac{1}{a_{1}}\pa{\sqrt{\frac{k}{n}}+\frac{2\xi}{\sqrt{nk}}}.
\end{equation}
 \end{exa}

\subsection{A refined result under Assumption~\ref{Hypo-2}}\label{sect-thm2}
In this section, we assume that the family $\sT(\ell,\sM)$ satisfies the stronger Assumption~\ref{Hypo-2}.

Let us first introduce the mapping
%Equation
\begin{equation}\label{def-phi}
\map{\phi}{(0,+\infty)}{\R_{+}}{z}{\dps{\phi(z)=\frac{2\pa{e^{z}-1-z}}{z^{2}}}.}
\end{equation}
The function $\phi$ is increasing on $(0,+\infty)$ and tends to $1$ when $z$ tends to 0. Given $\beta>0$ and a family $\sT(\ell,\sM)$ that satisfies Assumption~\ref{Hypo-2}, we define 
%Align
\begin{align}
\overline\cc_{1}&=\cc_{0}-\beta a_{2}a_{1}^{-1}\tau^{2}\phi\cro{\beta(1+2\cc)}(1+2\cc(1+\cc));\label{def-delta1}\\
\overline \cc_{2}&= \cc-\beta a_{2}a_{1}^{-1}\tau^{2}\phi\cro{\beta(1+2\cc)}\cc^{2};\label{def-delta2}\\
\overline\cc_{3}&=(2+\cc) -\beta a_{2}a_{1}^{-1}\tau^{2}\phi\cro{\beta(3+2\cc)}(2+\cc)^{2}.\label{def-delta3}
\end{align}
Note that the value of $\overline\cc_{1}\wedge \overline\cc_{2} \wedge \overline\cc_{3}$ is positive for $\beta=0$ and decreases continuously to $-\infty$ when $\beta$ grows to infinity. Consequently, there exists some $\beta_{0}>0$ for which  $\overline\cc_{1}\wedge \overline\cc_{2} \wedge \overline\cc_{3}=0$ and $\overline\cc_{1}\wedge \overline\cc_{2} \wedge \overline\cc_{3}$ is positive for all values $\beta\in (0,\beta_{0})$. 

Let us now present our second result on the concentration property of our posterior $\widehat \pi_{\bsX}$.
\begin{thm}\label{main2}
Assume that the model $(\sM,\pi)$ and the loss $\ell$ satisfy Assumption~\ref{Hypo-0} and~\ref{Hypo-0bis} and the family $\sT(\ell,\sM)$ Assumption~\ref{Hypo-2}. For $\beta\in (0,\beta_{0})$ and $\gamma<(\overline\cc_{1}\wedge \overline\cc_{2} \wedge \overline\cc_{3})/(2\tau)$, the posterior $\widehat \pi_{\bsX}$ defined by~\eref{def-piX} satisfies the following property. There exists $\kappa_{0}>0$ only depending on $a_{0}/a_{1},a_{2}/a_{1},\cc,\tau,\beta$ and $\gamma$ such that, for all $\xi>0$ and any distribution $\gP\et$ with marginals in $\sP$,
%Align
\begin{align}\label{eq-thm02}
\E\cro{\widehat \pi_{\bsX}\pa{\co{\sB}(\overline P\et,\kappa_{0}\ray)}}\le 2e^{-\xi}
\end{align}
with 
%Equation
\begin{equation}\label{eq-raythm2}
\ray=\inf_{P\in\sM}\cro{\ell(\overline P\et,P)+\ray_{n}(\beta,P)}+\frac{2\xi}{n\beta a_{1}}.
\end{equation}
In particular, 
%Align
\begin{align*}
\P\cro{\widehat \pi_{\bsX}\pa{\co{\sB}(\overline P\et,\kappa_{0}\ray)}\ge e^{-\xi/2}}\le 2e^{-\xi/2}.
\end{align*}
\end{thm}
The value of $\kappa_{0}$ is given by \eref{def-Km0} in the proof.
Note that the constraints on $\beta$ and $\gamma$, that are required in our Theorem~\ref{main2}, and that on $\cc$ given in \eref{eq-cc} 
only depend on $a_{0},a_{1}$ and $a_{2}$, hence on the choice of the family $\sT(\ell,\sM)$. When $a_{0},a_{1}$ and $a_{2}$ do not depend on $\sM$, the value of $\beta$ can be chosen as a universal constant. In particular, it neither depends on the model $(\sM,\pi)$ nor on the sample size $n$.  

\begin{exa}[Example~\ref{ex-euclidien} continued]\label{ex-euclidien2}
Let us go back to the framework of our Example~\ref{ex-euclidien} and assume that $\sT(\ell,\sM)$ satisfies the requirements of Theorem~\ref{main2}, hence Assumption~\ref{Hypo-2}. Applying our construction with some numerical value of $\beta$ which satisfies the constraint of our Theorem~\ref{main2},  we deduce from~\eref{eq-controlpara} that $\widehat \pi_{\bsX}$ concentrates on an $\ell$-ball with radius of order 
%Equation
\begin{equation}\label{eq-ex-euclidien02}
\overline \ray=\inf_{P\in\sM}\ell(\overline P\et,P)+\frac{\log(2B/A)}{\gamma a_{1} \beta} \frac{k}{n}+\frac{2}{a_{1}\beta}\frac{\xi}{n}
%\frac{k+\xi}{n}.
\end{equation}
When the model is well-specified, $\inf_{P\in\sM}\ell(\overline P\et,P)=0$ and $\ray=\ray(n)$ contracts thus at the rate $1/n$. Applying our Theorem~\ref{main1} under Assumption~\ref{Hypo-1}, ignoring the fact that the family $\sT(\ell,\sM)$ also satisfies Assumption~\ref{Hypo-2}, would lead to the weaker result that the posterior concentrates on an $\ell$-ball with radius of order $\inf_{P\in\sM}\ell(\overline P\et,P)+\sqrt{k/n}$ as shown by~\eref{eq-ex-euclidien01}.
\end{exa}

\subsection{Concentrated priors}
Theorem~\ref{main1} and \ref{main2} show that starting from a prior $\pi$ that puts enough mass around most of the elements of $\sM$, the posterior $\widehat \pi_{\bsX}$ concentrates on an $\ell$-ball with radius of order $\inf_{P\in\sM}\ell(\overline P\et,P)+r_{n}$ where $r_{n}$ is small, at least under suitable assumptions and for $n$ sufficiently large. The situation we want to investigate now is what happens when the prior is very concentrated on a small $\ell$-ball with radius $\eps>0$ around an element $\PAP\in \sM$ that might not be the true distribution of the data. More precisely, assume the following 
\begin{ass}\label{ass-dege}
For $\PAP\in\sM$ and $\eps>0$, 
\[
\pi\pa{\co{\sB}(\PAP,\eps)}\le e^{-(2\xi+1)}\pi\pa{\sB(\PAP,\eps)}\quad \text{with $\xi>0$.}
\]
\end{ass}

In this case, we establish the following result.
\begin{thm}\label{main1b}
Assume that the model $(\sM,\pi)$ and the loss $\ell$ satisfy Assumption~\ref{Hypo-0} and~\ref{Hypo-0bis} and the family $\sT(\ell,\sM)$ Assumption~\ref{Hypo-1}. If Assumption~\ref{ass-dege} is satisfied, there exists $\kappa_{0}>0$ only depending on $\cc,\tau$ and the ratio $a_{0}/a_{1}$ such that for any distribution $\gP\et$ with marginals in $\sP$,  
%Align
\begin{align}
\E\cro{\widehat \pi_{\bsX}\pa{\co{\sB}(\overline P\et,\kappa_{0} \ray)}}\le 2e^{-\xi}\quad \text{with}\quad \ray=\ell(\overline P\et,\PAP)\vee \frac{\beta}{a_{1}}\vee \eps.
\end{align}
In particular, for the choice $\beta=a_{1} \eps$, $r=\ell(\overline P\et,\PAP)\vee \eps$.

If furthermore, Assumption~\ref{Hypo-2} is satisfied and $\beta\in (0,\beta_{0})$ (where $\beta_{0}$ is defined in Section~\ref{sect-thm2}), there exists $\kappa_{0}'>0$ only depending on $\tau,\beta, a_{0}/a_{1}$ and $a_{2}/a_{1}$ such that for any distribution $\gP\et$ with marginals in $\sP$, 
%Align
\begin{align}
\E\cro{\widehat \pi_{\bsX}\pa{\co{\sB}(\overline P\et,\kappa_{0}' \ray)}}\le 2e^{-\xi}\quad \text{with}\quad \ray=\ell(\overline P\et,\PAP)\vee \eps.
\end{align}
\end{thm}
This result shows that for a suitable choice of $\beta$, the posterior $\widehat \pi_{\bsX}$ also concentrates on an $\ell$-ball centred at $\overline P\et$ with radius of order $\eps$ when the model is well-specified, that is, when the data are i.i.d.\ with distribution $\overline P\et=\PAP$. When the model is misspecified, the radius of the ball is of order $\ell(\overline P\et,\PAP)\vee \eps$ and therefore does not inflate more than the distance of $\overline P\et$ to the center $\PAP$. This result illustrates the stability of the posterior  $\widehat \pi_{\bsX}$ with respect to misspecification.  

\section{Applications to classical loss functions}\label{sect-ACLF}
The aim of this section is to show how our general construction can be applied to loss functions $\ell$ of interest. The propositions contained in this section about the corresponding families $\sT(\ell,\sM)$ have been established in Baraud~\citeyearpar{BY-TEST} except for the squared Hellinger loss for which we refer to Baraud and Birg\'e~\citeyearpar{BarBir2018}[Proposition~3]. The list of loss functions we present here is not exhaustive. Our results also apply to all loss functions that derive from a variational formula of the form
\[
\ell(P,Q)=\sup_{f\in\sF}\cro{\int_{E}fdP-\int_{E}fdQ}
\]
where $\sF$ is a suitable class of bounded functions. For such losses, we refer the reader to Baraud~\citeyearpar{BY-TEST}. 

In this section, we consider models $\sM=\{P=p\cdot\mu, p\in\cM\}$ which are dominated by a measure $\mu$ on $(E,\cE)$ and we denote by $\cM\subset \sL_{1}(E,\cE,\mu)$ the corresponding families of densities with respect to $\mu$. Elements $P,Q,\ldots$ in $\sM$ are associated with their densities in $\cM$ by using lower case letters $p,q,\ldots$. In all the cases we consider, $t_{(P,Q)}(x)$ is a measurable function of $(p(x),q(x))$ for $P,Q\in\sM$ and $x\in E$. In order to satisfy our measurability Assumption~\ref{Hypo-1}-\ref{cond-1}, it is therefore sufficient to assume that  
\[
\begin{array}{rcl}
(E\times\sM,\cE\otimes\cA) & \longrightarrow & \R \\
(x,P) & \longmapsto & p(x) 
\end{array}
\]
is measurable. In the case of a parametrized model $\sM=\{P_{\theta}=p_{\theta}\cdot\mu, \theta\in \Theta\}$, as described in Section~\ref{sect-scpm}, this condition is satisfied as soon as  the mapping 
\[
\map{p}{(E\times \Theta,\cE\otimes \frak{B})}{\R_{+}}{(x,\theta)}{p_{\theta}(x)}
\]
is measurable. Throughout this section, we assume that such measurability assumptions are satisfied.

\subsection{The case of the total variation distance}
In this section, $\sP$ is the set of all probability measures on $(E,\cE)$ and 
%
%Equation
\begin{equation}\label{def-VT}
\norm{P-Q}=\frac{1}{2}\int_{E}\ab{p-q}d\mu
\end{equation}
denotes the total variation loss (TV-loss for short) between $P,Q\in \sP$. 
\begin{prop}\label{prop-loss-TV}
The family $\sT(\ell,\sM)$ which consists of all the functions $t_{(P,Q)}$ defined for $P=p\cdot\mu$ and $Q=q\cdot \mu$ in $\sM$ by 
%Equation
\begin{equation}\label{eq-t-TV}
t_{(P,Q)}=\frac{1}{2}\cro{\1_{q>p}-Q(q>p)}-\frac{1}{2}\cro{\1_{p>q}-P(p>q)}
\end{equation}
satisfies Assumption~\ref{Hypo-0bis} with $\tau=1$ and Assumption~\ref{Hypo-1} with $a_{0}=3/2$ and $a_{1}=1/2$.
\end{prop}

It follows from Proposition~\ref{prop-loss-TV} that we may apply our general construction to the so-defined family $\sT(\ell,\sM)$ with the values  $\cc=\cc_{0}=1/3$ (hence $\lambda=4/3$). The reader can check that the value $\gamma=1/100$ satisfies the requirement of our Theorem~\ref{main1} and that \eref{eq-thm01} is satisfied with $\kappa_{0}=220$. Theorem~\ref{main1} can therefore be rephrased as follows.  
% Corollaire
\begin{cor}\label{cor-TV}
Let $\beta\ge 1/\sqrt{n}$, $\cc=1/3$ and $\widehat \pi_{\bsX}^{{\rm TV}}$ be the posterior defined by~\eref{def-piX} and associated with the family $\sT(\ell,\sM)$ given in Proposition~\ref{prop-loss-TV}.
%Assume that the model $(\sM,\pi)$ satisfies the measurability condition of Assumption~\ref{Hypo-0} for the TV-loss and the family $\sT(\ell,\sM)$, given in Proposition~\ref{prop-loss-TV}, that  of Assumption~\ref{Hypo-1}-\ref{cond-1}. 
For all $\xi>0$ and any distribution $\gP\et$, with a probability at least $1-2e^{-\xi/2}$,  the posterior $\widehat \pi_{\bsX}^{{\rm TV}}$ satisfies 
%Align
\begin{align}
\widehat \pi_{\bsX}^{{\rm TV}}&\pa{\ac{P\in\sM,\; \ell(\overline P\et,P)\le 220\cro{\inf_{P'\in\sM(\beta)}\ell(\overline P\et,P')+2\pa{\beta+\frac{2\xi}{n\beta}}}}}\nonumber \\
&\ge 1-e^{-\xi/2}\label{eq-cor-TV}
\end{align}
where 
\[
\sM(\beta)=\ac{P\in\sM,\; \sup_{r\ge 2\beta}\cro{\frac{200}{n\ray}\log\pa{\frac{\pi\pa{\sB(P,2\ray)}}{\pi\pa{\sB(P,\ray)}}}}\le \beta}.
\]
\end{cor}

By convexity, we may write that 
\[
\inf_{P\in\sM(\beta)}\norm{P-\overline P\et}\le \inf_{P\in\sM(\beta)}\cro{\frac{1}{n}\sum_{i=1}^{n}\norm{P-P_{i}\et}}
\]
and the left-hand side is therefore small  when there exists $P\in\sM(\beta)$ that approximates well enough most of the marginals of $\gP\et$. The concentration properties of $\widehat \pi_{\bsX}^{{\rm TV}}$ remain thus stable with respect to a possible misspecification of the model and a departure from the equidistribution assumption. 

In fact, as we shall see in our Example~\ref{exa-TV} below, the average distribution $\overline P\et$ may belong to
$\sM(\beta)$ even when none of the marginals $P_{i}\et$ does. This means that in good situations, the posterior may concentrate around $\overline P\et$, as it would do in the i.i.d.\ case when the distribution of the data does belong to $\sM(\beta)$, even when the data are non-i.i.d.\ and their marginals do not belong to $\sM(\beta)$.  

\begin{exa}\label{exa-TV}[Example~\ref{ex-euclidien} continued]
Going back to Example~\ref{ex-euclidien} and taking for $\ell$ the TV-loss (then $a_{1}=1/2$), we deduce from~\eref{eq-ex-euclidien01} that 
\[
\overline \ray_{n}=\inf_{P\in\sM}\norm{\overline P\et-P}+2\pa{\sqrt{\frac{k}{n}}+\frac{2\xi}{\sqrt{n k}}}.
\]
In particular, if for each $i\in\{1,\ldots,n\}$, $P_{i}\et$ is the uniform distribution on $[(i-1)/n,i/n]$ and $\sM$ contains the uniform distribution $\cU([0,1])$ on $[0,1]$, $\sM$ contains  $\overline P\et=\cU([0,1])$, even if none of the marginals $P_{i}\et$ belongs to $\sM$. We then get that
\[
\overline \ray_{n}=2\pa{\sqrt{\frac{k}{n}}+\frac{2\xi}{\sqrt{n k}}}
\]
and the posterior concentrates around $\overline P\et$ at a parametric rate.
\end{exa}

\subsection{Case of the $\L_{j}$-loss}
Let $j\in (1,+\infty)$. We denote by $\sP_{j}$ the set of all finite and signed measures on $(E,\cE,\mu)$ which are of the form $P=p\cdot \mu$ with $p\in \sL_{j}(E,\mu)\cap \sL_{1}(E,\mu)$. Let $\ell_{j}$ be the loss defined by $\ell_{j}(P,Q)=\norm{p-q}_{\mu,j}$ for all $P=p\cdot \mu$ and $Q=q\cdot \mu$ in $\sP_{j}$. In this section, $\sP$ is the subset that consists of all the probability measures in $\sP_{j}$. 

\begin{prop}\label{prop-loss-Lj}
Let $\sM=\ac{P=p\cdot\mu,\; p\in\cM}$ be a subset of $\sP_{j}$ for which $\cM$ satisfies for some $R>0$
% Equation
\begin{equation}\label{eq-LinftyLj}
\hspace{8mm}\norm{p-q}_{\infty}\le R\norm{p-q}_{\mu,j}\quad \text{for all $p,q\in\cM$.}
\end{equation}
Define for $P=p\cdot\mu$ and $Q=q\cdot\mu$ in $\sM$, 
\[
\hspace{12mm}f_{(P,Q)}=\frac{\pa{p-q}_{+}^{j-1}-\pa{p-q}_{-}^{j-1}}{\norm{p-q}_{\mu,j}^{j-1}}\quad \text{when }P\ne Q\quad \text{and}\quad f_{(P,P)}= 0.
\]
Then, the family $\sT(\ell_{j},\sM)$ which contains the functions $t_{(P,Q)}$ defined for $P,Q\in\sM$ by
% Equation
\begin{equation}\label{T-Lj}
t_{(P,Q)}=\frac{1}{2R^{j-1}}\cro{\int_{E}f_{(P,Q)}\frac{dP+dQ}{2}-f_{(P,Q)}}
\end{equation}
satisfies Assumption~\ref{Hypo-0bis} with $\tau=1$ and Assumption~\ref{Hypo-1} with $a_{0}=3/(4R^{j-1})$ and $a_{1}=1/(4R^{j-1})$.
\end{prop}
When $j=2$, \eref{eq-LinftyLj} is typically satisfied when $\cM$ is a subset of a linear space enjoying good connections between the $\L_{2}(\mu)$ and the supremum norms. Many finite dimensional linear spaces with good approximation properties do satisfy such connections (e.g.\ piecewise polynomials of a fixed degree on a regular partition of $[0,1]$, trigonometric polynomials on $[0,1)$ etc.).  We refer the reader to Birg\'e and Massart~\citeyearpar{MR1653272}[Section~3] for additional examples. The property may also hold for infinite dimensional linear spaces as proven in Baraud~\citeyearpar{BY-TEST}. 

It follows from Proposition~\ref{prop-loss-Lj} that one may choose $\cc=\cc_{0}=1/3$ in \eref{eq-cc} and $\gamma=1/100$ in Theorem~\ref{main1}. Besides, Theorem~\ref{main1} applies with $\kappa_{0}=220$. 

\begin{exa}[Example~\ref{ex-euclidien} continued]
Let us go back to our Example~\ref{ex-euclidien} with $\ell=\ell_{j}$ and $\sT(\ell,\sM)$ given in Proposition~\ref{prop-loss-Lj}. For the choice of $\beta$ given in~\eref{eq-choixbeta1} and $\gamma=1/100$, we deduce from~\eref{eq-ex-euclidien01} (with $a_{1}=1/(4R^{j-1})$) that the resulting posterior $\widehat \pi_{\bsX}$ concentrates on an $\ell_{j}$-ball around $\overline P\et$ with a radius of order 
\[
\overline r_{n}=\inf_{p\in\cM}\norm{\frac{1}{n}\sum_{i=1}^{n}p_{i}\et-p}_{\mu,j}+4R^{j-1}\pa{\sqrt{\frac{k}{n}}+\frac{2\xi}{\sqrt{n k}}}.
\]
\end{exa}

\subsection{The case of the squared Hellinger loss}
In this section, $\sP$ is the set of all probability measures on $(E,\cE)$ and 
\begin{equation}\label{def-h2}
\ell(P,Q)=h^{2}(P,Q)=\frac{1}{2}\int_{E}\pa{\sqrt{p}-\sqrt{q}}^{2}d\mu,
\end{equation}
is the squared Hellinger distance between two probabilities $P,Q\in\sP$. 
\begin{prop}\label{prop-loss-h2}
Let $\psi$ be the function defined by
\[
\map{\psi}{[0,+\infty]}{[-1,1]}{x}{{
\begin{cases}
\dps{\frac{x-1}{x+1}} &\text{ if $x\in [0,+\infty)$}\\
1 &\text{ if $x=+\infty$.}
\end{cases}
}}
\]
The family $\sT(\ell,\sM)$ containing the functions $t_{(P,Q)}$ defined for $P=p\cdot\mu$ and $Q=q\cdot \mu$ in $\sM$ by 
%Equation
\begin{equation}\label{eq-t-h}
t_{(P,Q)}=\frac{1}{2}\psi\pa{\sqrt{\frac{q}{p}}}
\end{equation}
(with the conventions $0/0=1$ and $x/0=+\infty$ for all $x>0$) satisfies Assumption~\ref{Hypo-0bis} with $\tau=2$ and Assumption~\ref{Hypo-2} with $a_{0}=2$, $a_{1}=3/16$, $a_{2}=3\sqrt{2}/4$.
\end{prop}

With such a choice of family $\sT(\ell,\sM)$, \eref{eq-cc} is satisfied with $\cc=1/125$, then $\cc_{0}\in [0.922,0.923]$,  
and the requirements of Theorem~\ref{main2} are satisfied with $\beta=2\gamma=1/500$. Then the value $\kappa_{0}=1694$ suits. The definition \eref{df-epsn} of $\ray_{n}(\beta,P)$ for $P\in\sM$ becomes 
%
%Align
\begin{align}
&\ray_{n}(\beta,P)\nonumber\\
&\quad =\inf\ac{\ray\ge \frac{8000}{3n},\; \frac{\pi\pa{\sB(P,2\ray')}}{\pi\pa{\sB(P,\ray')}}\le \exp\pa{\frac{3n\ray'}{8.10^{6}}}\text{ for all $\ray'\ge \ray$}},\label{eq-rnb-h}
\end{align}
with the convention $\sup\vide=8000/(3n)$. Theorem~\ref{main2} can then be rephrased as follows.

% Corollaire
\begin{cor}\label{cor-hell}
Let $\pi_{\bsX}^{h}$ be the posterior defined by~\eref{def-piX} and associated with the family $\sT(\ell,\sM)$ given in Proposition~\ref{prop-loss-h2} and the choices $\cc=1/125$ and $\beta=1/500$. 
%
%Assume that the model $(\sM,\pi)$ satisfies the measurability condition of Assumption~\ref{Hypo-0} for the squared Hellinger loss and the family $\sT(\ell,\sM)$, given in Proposition~\ref{prop-loss-h2}, that  of Assumption~\ref{Hypo-1}-\ref{cond-1}.
%The posterior $\widehat \pi_{\bsX}^{h}$  defined by~\eref{def-piX} and based on $\sT(\ell,\sM)$ satisfies the following property. 
For all $\xi>0$ and any distribution $\gP\et$, with a probability at least $1-2e^{-\xi/2}$, 
\[
\widehat \pi_{\bsX}^{h}\pa{\ac{P\in\sM,\; h^{2}\pa{\overline P\et,P}\le 1694\ray}}\ge 1-e^{-\xi/2}
\]
where
\[
\ray=\inf_{P\in\sM}\cro{h^{2}\pa{\overline P\et,P}+\ray_{n}(\beta,P)}+\frac{5334\xi}{n}
\]
and $\ray_{n}(\beta,P)$ is given by \eref{eq-rnb-h}.
\end{cor}

As for the total variation distance, we may write that 
\[
\inf_{P\in\sM}h^{2}\pa{\overline P\et,P}\le \inf_{P\in\sM}\cro{\frac{1}{n}\sum_{i=1}^{n}h^{2}\pa{P_{i}\et,P}}.
\]
The left-hand side is small when there exists an element $P\in\sM$ that approximates well most of the marginal distribution $P_{i}\et$. If for such a $P$, the quantity $\ray_{n}(\beta,P)$ is small enough, the posterior concentrates around $\overline P\et$ just as it would do if the data were truly i.i.d.\ with distribution $P\in\sM$.

\begin{exa}[Example~\ref{ex-euclidien} continued]
Let us go back to Example~\ref{ex-euclidien}, more precisely Example~\ref{ex-euclidien2}, with $\ell=h^{2}$ and $\sT(\ell,\sM)$ given in Proposition~\ref{prop-loss-h2}. Inequality \eref{eq-controlpara} is satisfied with $\beta=2\gamma=1/500$ and $a_{1}=3/16$. It follows from~\eref{eq-ex-euclidien02} that $\widehat \pi_{\bsX}^{h}$ concentrates on an $h^{2}$-ball around $\overline P\et$ with a radius of order 
\[
\overline r=\inf_{P\in\sM}h^{2}\pa{\overline P\et,P}+\frac{k+\xi}{n}.
\]
\end{exa}

\section{Comparing the classical Bayesian approach to ours}\label{sect-ctcbato}
In this section, our aim is to highlight some similarities and differences between the Bayesian posterior and ours. Throughout this section, we consider the squared Hellinger loss $\ell=h^{2}$ and denote by $\widehat \pi_{\bsX}^{K}$ the Bayes posterior associated with the model $(\sM,\pi)$. The letter $K$ in the notation $\widehat \pi_{\bsX}^{K}$ refers to the fact that the Bayesian posterior can be obtained from our general construction by using the Kullback--Leibler divergence as explained in Section~\ref{sect-cwtcbpd}. When $\sM=\{P_{\gtheta},\; \gtheta\in \Theta\}$ is parametric with $\Theta\subset \R^{k}$, we denote by $\widehat \nu_{\bsX}^{K}$ the Bayesian posterior on the parameter space $\Theta$ and $\widehat \nu_{\bsX}^{h}$ that associated to $\widehat \pi_{\bsX}^{h}$. 

\subsection{Some classical concentration results for the Bayes posterior distribution}
Most of the results that have been established about the concentration properties of the Bayesian posterior are asymptotic in nature. It seems difficult to establish a general nonasymptotic version of those as we do for our posterior. One of the only exceptions we are aware of 
is Birg\'e~\citeyearpar{BirgeLucien2015Atnb}. 

When the data are i.i.d.\ with a distribution $P\et\in\sM$, a typical asymptotic form of these results is the following one (see Ghosal, Ghosh and van der Vaart~\citeyearpar{MR1790007} Theorems~2.1 and 2.4 for example). Let $\eps_{n}$ be a sequence of positive numbers that converges to zero when $n$ goes to infinity. If $P\et$ fulfils some suitable conditions, that we shall discuss later on and which depend on the prior $\pi$ and $\eps_{n}$, the following convergence in probability holds true
%the Bayesian posterior distribution $\widehat \pi^{K}_{\bsX}$ satisfies 
%Align
\begin{align}\label{eq-vaartetmoi0}
\widehat \pi^{K}_{\bsX}\pa{\{P\in\sM,\; h^{2}(P\et,P)\ge M_{n}\eps_{n}^{2}\}}\CP{n\to +\infty} 0 \quad \text{under $P\et$}.
\end{align}
In \eref{eq-vaartetmoi0}, $M_{n}=M$ denotes some large enough positive constant if $n\eps_{n}^{2}\to +\infty$ as $n\to +\infty$ while $M_{n}$ is increasing to infinity as $n\to +\infty$ if $\liminf n\eps_{n}^{2}>0$ as $n\to +\infty$. The first condition on $\eps_{n}$ is typically satisfied when $\sM$ is a nonparametric model while the second one generally applies to parametric ones. 

In comparison, in this well-specified framework, our Corollary~\ref{cor-hell} leads to the following result. For all $P\et\in\sM$ and $\xi>0$ 
%Equation
\begin{align}
&\P\cro{\widehat \pi_{\bsX}^{h}\pa{\ac{P\in\sM,\; h^{2}(P\et,P)\ge \kappa_{0}'\pa{r_{n}(\beta,P\et)+\frac{\xi}{n}}}}\ge e^{-\xi/2}}\nonumber \\
&\le 2e^{-\xi/2}\label{eq-controlpih}
\end{align}
for some numerical constant $\kappa_{0}'>0$. If $P\et$ satisfies $r_{n}(\beta,P\et)\le \eps_{n}^{2}$, we recover~\eref{eq-vaartetmoi0} by setting $\xi=\xi_{n}=(M_{n}/(\kappa_{0}')-1)n\eps_{n}^{2}$. However, our condition that $r_{n}(\beta,P\et)\le \eps_{n}^{2}$ is not equivalent to that imposed on $P\et$ by Ghosal, Ghosh and van der Vaart~\citeyearpar{MR1790007}. It is actually weaker. In their paper, this condition is fulfilled when the prior puts enough mass on Kullback--Leibler type balls around $P\et$. Our approach allows one to consider  Hellinger balls only, which are larger and make our assumption weaker. In fact, as already underlined in the Introduction, these Kullback--Leibler type balls could be empty, and the condition unsatisfied, while our theorem would still apply. 

The result established by Birg\'e~\citeyearpar{BirgeLucien2015Atnb} provides an improvement as compared to the one presented above and established by Ghosal, Ghosh and van der Vaart. Birg\'e  shows that it is essentially possible to get rid of the Kullback--Leibler divergence (see his Theorem~2) but only when the model is parametric and well-specified. Apart for the nonparametric framework, this result leaves little place for improvement since we know that the Bayesian posterior may fail to concentrate around the true parameter when the model becomes slightly ill-specified. 

Another consequence of our Corollary~\ref{cor-hell}, as compared to \eref{eq-vaartetmoi0}, is that it allows one to control  
\[
\widehat \pi_{\bsX}^{h}\pa{\ac{P\in\sM,\; h^{2}(P\et,P)\ge \kappa_{0}'\pa{\eps_{n}^{2}+\frac{\xi}{n}}}}
\]
uniformly over the set $\{P\et\in\sM, r_{n}(\beta,P\et)\le \eps_{n}^{2}\}$. For example, in the framework of Example~\ref{ex-euclidien2}, for the choice $\eps_{n}^{2}=ck/n$ with $c=\log(2B/A)/(\gamma a_{1}\beta)$, we know that $r_{n}(\beta,P\et)\le \eps_{n}^{2}$ for all $P\et\in \sM$ and we deduce from~\eref{eq-controlpih}  that 
%Equation
\begin{align*}
\sup_{P\et\in\sM}\P\cro{\widehat \pi_{\bsX}^{h}\pa{\ac{P\in\sM,\; h^{2}(P\et,P)\ge \kappa_{0}'\pa{\eps_{n}^{2}+\frac{\xi}{n}}}}\ge e^{-\xi/2}}\le 2e^{-\xi/2}.
\end{align*}
The concentration properties of our posterior is therefore uniform over the statistical model $\sM$. 

\subsection{About the shapes and sizes of the credible regions}
A nice feature of the Bayesian approach lies in the fact that it allows one to build credible regions. In practice, they often play the same role as the confidence regions in the frequentist paradigm. 
When the data are i.i.d.\ with distribution $P\et=P_{\gtheta\et}$ in a parametric model $\sM=\{P_{\gtheta},\; \gtheta\in \Theta\}$, $\Theta\subset \R^{k}$, a credible set for the parameter $\gtheta\et$ is a subset $\widehat \Theta_{n,\bsX}\subset \Theta$ (only depending on observable quantities) that satisfies $\widehat \nu_{\bsX}^{K}(\widehat \Theta_{n,\bsX})\ge 1-e^{-\xi}$ for some choice of $\xi>0$. When $\sM$ is a regular parametric model with a nonsingular Fisher information matrix $\gJ$, and provided that it satisfies additional assumptions -- see  van der Vaart~\citeyearpar{MR1652247} -- the Bernstein-von Mises theorem applies and tells us that 
\[
\norm{\widehat \nu_{\bsX}^{K}-\cN\pa{\widehat \gtheta_{n},(n\gJ(\gtheta\et))^{-1}}}\CP{n\to +\infty} 0\quad \text{under $P_{\gtheta\et}$}
\]
where $\widehat \gtheta_{n}$ denotes the Maximum Likelihood Estimator (MLE for short). Denoting by $\overline \chi_{k}^{-1}(\xi)$ the $(1-e^{-\xi})$-quantile of a chi-square random variable with $k$ degrees of freedom and 
%Equation
\begin{equation}\label{eq-ConfReg}
\Theta_{n,\bsX}=\ac{\gtheta\in \Theta,\; n\ab{\gJ^{1/2}(\gtheta\et)\pa{\widehat \gtheta_{n}-\gtheta}}^{2}\le \overline \chi_{k}^{-1}(\xi)},
\end{equation}
we deduce that 
\[
\ab{\widehat \nu_{\bsX}^{K}\pa{\Theta_{n,\bsX}}-(1-e^{-\xi})}\le \norm{\widehat \nu_{\bsX}^{K}-\cN\pa{\widehat \gtheta_{n},(n\gJ(\gtheta\et))^{-1}}}\CP{n\to +\infty} 0
\]
hence 
\[
\widehat \nu_{\bsX}^{K}\pa{\Theta_{n,\bsX}}\CP{n\to +\infty}1-e^{-\xi}\quad \text{under $P_{\gtheta\et}$}.
\]
The asymptotic level of ``credibility'' of the set $\Theta_{n,\bsX}$ is therefore $1-e^{-\xi}$. This set is not, however, a credible region since it depends on the unknown parameter $\gtheta\et$. We would obtain a genuine credible region by replacing $\gtheta\et$ by $\widehat \gtheta_{n}$ in the expression of $\Theta_{n,\bsX}$. This substitution would change the level of credibility but not the shape of the region, which is an ellipsoid centred at $\widehat \gtheta_{n}$ and the axes of which are given by the eigenvectors of the Fisher information matrix. 

The aim of this section is to show that our posterior concentrates its mass on regions that have the same shape and approximately the same size. The size of $\Theta_{n,\bsX}$ is determined by the value of the quantile $\overline \chi_{k}^{-1}(\xi)$. The aim of the following lemma is to specify the order of  magnitude of this quantile as a function of $k$ and $\xi$. In fact, we consider below the more general case of the quantiles of a gamma distribution $\gamma(s,\sigma)$ with parameters $s,\sigma>0$, that is, the distribution with density $x\mapsto (x^{s-1}e^{-s/\sigma})/(\sigma^{s}\Gamma(s))$ with respect to the Lebesgue measure on $\R_{+}$. The proof is postponed to Section~\ref{sect-pf-lem-gamma}.
% Lemme
\begin{lem}\label{lem-gamma}
For $s,\sigma,\xi>0$, let $\overline \gamma_{s,\sigma}^{-1}(\xi)$ be the $(1-e^{-\xi})$-quantile of the gamma distribution $\gamma(s,\sigma)$ and $\overline \Phi^{-1}(\xi)$ that of a standard Gaussian random variable. Then, 
%Equation
\begin{equation}\label{lem-gamma00}
\overline \gamma_{s,\sigma}^{-1}(\xi)\le \sigma\pa{\sqrt{s}+\sqrt{\xi}}^{2}
\end{equation}
and for all $s=t+1>1$ and $\xi\ge \log 2+1/(12t)$, 
%Equation
\begin{equation}\label{lem-gamma01}
\overline \gamma_{s,\sigma}^{-1}(\xi)\ge \sigma\cro{t+\cro{\sqrt{t}\;\overline \Phi^{-1}\pa{\xi-\frac{1}{12t}}}\vee \cro{\xi +\log\pa{\frac{e^{-1/(12t)}}{\sqrt{2\pi t}}}}}.
\end{equation}
\end{lem}
Since $\overline \Phi^{-1}(\xi)$ is equivalent to $\sqrt{2\xi}$ for large values of $\xi>0$, these two inequalities show that for $s$ and $\xi$ large enough, $\overline \gamma_{s,\sigma}^{-1}(\xi)$ is of order $\sigma\cro{s+\xi}$. In particular, $\overline \chi_{k}^{-1}(\xi)=\overline \gamma_{k/2,2}^{-1}(\xi)$ is of order $k+\xi$ for $k$ and $\xi$ large enough.

%% Hypothèse
%\begin{ass}\label{hypo-reg}
%The statistical model $\sM=\{P_{\gtheta},\; \gtheta\in \Theta\}$ is regular, in the sense of Definition~\ref{df-regulier}. The prior $\nu$ on $\Theta\subset \R^{k}$ admits a density $q$ with respect to the Lebesgue measure on $\R^{k}$ which is continuous and positive at $\gtheta\et$. There exists $z_{0}>0$ such that $\cB(\theta\et,z_{0})\subset \Theta$ and $h(\gtheta,\gtheta\et)\ge \eta>0$ for all $\gtheta\in \Theta\setminus \cB(\theta\et,z_{0})$. 
%\end{ass}
%%
In comparison, we prove in Section~\ref{sect-pf-prop-credibleset} the result below for our posterior. This result is based on the assumption that the statistical model $\sM$ is regular in the sense that is defined in Ibragimov and Has'minski\u \i~\citeyearpar{MR620321}. In order to avoid too many technicalities here, we refer the reader to our Section~\ref{sect-RegMod}, more precisely Corollary~\ref{cor-odgh}, for a complete description of the assumptions on the statistical model $\sM$.
% Proposition
\begin{thm}\label{prop-credibleset}
Assume that the statistical model $\sM$ satisfies the assumptions of Corollary~\ref{cor-odgh}. If $X_{1},\ldots,X_{n}$  are i.i.d.\ with distribution $P_{\gtheta\et}\in\sM$, for all $\xi>0$ and $n$ large enough, with a probability $1-2e^{-\xi}$, 
%Equation
\begin{equation}\label{eq-credibleset}
\widehat \nu_{\bsX}^{h}\pa{\ac{\gtheta\in \Theta,\; n\ab{\gJ^{1/2}(\gtheta\et)\pa{\gtheta-\gtheta\et}}^{2}\le \kappa\et\pa{k+\xi}}}\ge 1-e^{-\xi}
\end{equation}
where  $\kappa\et$ is a positive numerical constant.
\end{thm}
The set 
\[
\ac{\gtheta\in \Theta,\; n\ab{\gJ^{1/2}(\gtheta\et)\pa{\gtheta-\gtheta\et}}^{2}\le \kappa\et\pa{k+\xi}}
\]
possesses the same shape and, by Lemma~\ref{lem-gamma}, approximately the same size as the set $\Theta_{n,\bsX}$ defined by~\eref{eq-ConfReg}. We deduce from Theorem~\ref{prop-credibleset} that the classical Bayes posterior and ours concentrate both on similar sets. If $\widehat \gtheta_{n}$ is an asymptotically efficient estimator of $\gtheta\et$, it is therefore reasonable to look for a credible region of the form 
\[
\ac{\gtheta\in \Theta,\;n\ab{\gJ^{1/2}(\widehat \gtheta_{n})\pa{\gtheta-\widehat \gtheta_{n}}}^{2}\le t}, t>0
\]
for $\widehat \nu_{\bsX}^{h}$ as we would do for the classical Bayes one.

\subsection{Robustness}\label{sect-rob}
As already mentioned, our approach allows the statistician to design robust posteriors by choosing as a loss function the squared Hellinger  loss or the total variation one. In this section, we illustrate this property on a concrete example. Consider the statistical model $\sM=\{P_{\theta}=\cN(\theta,1),\; \theta\in\R\}$ and the prior $\pi$ associated with the distribution $\nu=\cN(0,1)$ on $\Theta=\R$. Then, the Bayes posterior on $\Theta$ is $\widehat \nu_{\bsX}^{K}=\cN(\widehat m_{n}, \sigma_{n}^{2})$ with $\widehat m_{n}=(n+1)^{-1}\sum_{i=1}^{n} X_{i}$ and $\sigma_{n}^{2}=1/(n+1)$. It concentrates on intervals of the form $[\widehat m_{n}-c/\sqrt{n+1},\widehat m_{n}+c/\sqrt{n+1}]$ for $c>0$ large enough. If the distribution of the data is contaminated so that $X_{1},\ldots,X_{n}$ are i.i.d.\ with distribution 
\[
P\et=\pa{1-\frac{1}{n}}P_{0}+\frac{1}{n}\cN\pa{10^{4}(n+1),1/n},
\]
then with a probability at least $1-(1-1/n)^{n}\ge 1-1/e>1/2$, the posterior concentrates around $\widehat m_{n}\approx 10^{4}$, hence far away from 0, even though $P\et$ and $P_{0}$ are close: $\norm{P\et-P_{0}}\le 1/n$. 

In this specific framework, the model $\sM$ is regular, the Fisher information is constant and positive, $\nu$ admits a positive density which is continuous at $\theta\et=0$ and for all $\theta,\theta'\in \Theta$, $h^{2}(\theta,\theta')=1-e^{-|\theta-\theta'|^{2}/8}$. We shall see in Section~\ref{sect-copaic}, more precisely in Corollary~\ref{cor-odgh}, that for such regular statistical models $r_{n}(\beta,P_{0})\le \kappa\et/n$ for some numerical constant $\kappa\et>0$, at least for $n$ large enough. Since $h^{2}(P\et,P_{0})\le \norm{P\et-P_{0}}\le 1/n$, we deduce from Corollary~\ref{cor-hell} that the posterior $\widehat \nu_{\bsX}^{h}$ concentrates on a set of the form 
\[
\ac{\theta\in \R,\; h^{2}(\theta,0)\le\frac{c}{n}}=\ac{\theta\in \R,\; |\theta|\le \sqrt{8\log\pa{\frac{1}{1-c/n}}}}
\]
with $c>0$. This set is an interval around 0 of length $1/\sqrt{n}$, at least for $n$ sufficiently large. Despite the contamination of the data, the concentration property of $\widehat \nu_{\bsX}^{h}$ remains thus the same as in the well-specified case.

\section{Applications}\label{sct-app}
\subsection{How to choose $\beta$ in Theorem~\ref{main1} for a translation model?}\label{sect-trans00}
In this section, we consider the  translation model $\sM=\{P_{\theta}=p(\cdot-\theta)\cdot \mu,\; \theta\in\R\}$ where  $p$ is a density on $\R$ with respect to the Lebesgue measure $\mu$. Our aim is to estimate the translation parameter $\theta$ by using a prior $\nu_{\sigma}$ on $\Theta=\R$ with a density (with respect to $\mu$) of the form $q(\cdot/\sigma)/\sigma$ for some density $q$ and positive number $\sigma$. We evaluate the estimation error by means of the total variation loss. In order to use our construction we need to tune the parameter $\beta$.  In Section~\ref{ssect-thm1}, we suggested to choose $\beta\ge 1/\sqrt{n}$ satisfying \eref{def-betaquantile}. In order to find such a value of $\beta=\beta(\alpha)$, we may proceed as follows. Consider a symmetric bounded interval $I=[-l/2,l/2]\subset \R$ of length $l>0$ satisfying $\nu_{\sigma}(I)\ge 1-\alpha$, hence concentrating most of the mass of the prior $\nu_{\sigma}$. If the set 
$\sM(\beta)$ is large enough to contain $\{P_{\theta},\; \theta\in I\}$, 
%Equation
\begin{equation}\label{eq-sect-trans0}
\pi\pa{\sM(\beta)}\ge \pi\pa{\{P_{\theta},\; \theta\in I\}}=\nu_{\sigma}(I)\ge 1-\alpha
\end{equation}
and $\beta$ satisfies \eref{def-betaquantile}. We deduce from our Corollary~\ref{cor-TV} that the corresponding posterior $\widehat \pi_{\bsX}^{{\rm TV}}$ concentrates with a probability at least $1-2e^{-\xi/2}$ on a TV-ball with a radius of order
%
%Equation
\begin{equation}\label{eq-sect-trans00}
\inf_{P'\in\sM(\beta)}\ell(\overline P\et,P')+2\pa{\beta+\frac{2\xi}{n\beta}}\le\inf_{\theta\in I}\ell(\overline P\et,P_{\theta})+ 2\beta+\frac{4\xi}{\sqrt{n}}=r( \beta).
\end{equation}
The approximation term $\inf_{\theta\in I}\ell(\overline P\et,P_{\theta})$ is small as soon as  $\overline P\et$ is close enough to a distribution $P_{\theta\et}$ whose parameter $\theta\et$ belongs to $I$. If we want to prevent us from the situation where $\argmin_{\theta\in \Theta}\ell(\overline P\et,P_{\theta})$ is far from 0, we need to increase $I$ (or equivalently diminish $\alpha$). What would be the consequence on the value of $ \beta=\beta(\alpha)$? What if we increase $\sigma$, to make the prior distribution flatter, or diminish $\sigma$ to make it more picky? Finally, what is the influence of the choice of the density $q$ on the size of $ \beta$?  

These are the questions we want to answer in this section. In order to simplify the presentation of our results and avoid technicalities, we make the change of variables $l=2\sigma t$, or equivalently $t=l/(2\sigma)>0$, and assume the following.
\begin{ass}\label{Hypo-Trans}
The density $q$ is positive, symmetric and decreasing on $\R_{+}$. There exists some nonnegative and nondecreasing function $\varphi:[0,1)\to \R_{+}$ such that 
\[
\norm{P_{0}-P_{\theta}}\le r\iff \ab{\theta}\le \varphi(r)\quad \text{for all $r\in [0,1)$.}
\]
%
%
%Besides, there exists $L\in (0,+\infty]$ such that the mapping $H$ defined by 
%%
%\begin{equation}\label{def-H}
%\map{H}{[0,L)}{[0,1)}{t}{\norm{P_{t}-P_{0}}}
%\end{equation}
%%
%is bijective.
\end{ass}
When $p$ is symmetric and nonincreasing on $\R_{+}$, the total variation distance between $P_{0}$ and $P_{\theta}$ is given by
\[
\norm{P_{0}-P_{\theta}}=2P_{0}\pa{[0,|\theta|/2]}\quad \text{for all $\theta\in\R$.}
\]
Our Assumption~\ref{Hypo-Trans} is then satisfied with  $\varphi(r)=F_{0}^{-1}[(r+1)/2]$ for all $r\in [0,1)$, where $F_{0}^{-1}$ denotes the quantile function of the distribution $P_{0}$.  
%Under Assumption~\ref{Hypo-Trans}, $H$ is necessarily increasing on $[0,L)$ and we may define its inverse $G:[0,1)\to [0,L)$. 
We set
%Equation
\begin{equation}\label{def-Gammabar}
\overline \Gamma=\max\ac{\cro{\sup_{0<r\le 1/4}\frac{\varphi(2r)}{\varphi(r)}}q(0),\frac{1}{2\varphi(1/4)}}
\end{equation}
and assume that this quantity is finite. Note that it only depends on $q(0)$ and $p$. For example, if $p$ is the density $x\mapsto (1/2)e^{-|x|}$, 
\[
\norm{P_{0}-P_{\theta}}=1-\exp\cro{-|\theta|/2}\quad \text{and }\quad  \varphi: r\mapsto-2\log(1-r).
\]
Since the mapping $r\mapsto [\varphi(2r)/\varphi(r)]$ is increasing, we obtain in this case
\[
\overline \Gamma=\frac{1}{\log(4/3)}\max\ac{q(0)\log 2,\frac{1}{4}}.
\]
If now $p:x\mapsto (s/2)(1-|x|)^{s-1}\1_{|x|<1}$ with $s>0$, 
\[
\norm{P_{0}-P_{\theta}}=1-(1-|\theta|/2)^{s}\quad \text{and}\quad  \varphi: r\mapsto 2[1-(1-r)^{1/s}].
\]
The mapping $r\mapsto \varphi(2r)/\varphi(r)$ has a continuous extension on $[0,1/4]$ and is therefore bounded.  Given $q(0)$, $\overline \Gamma$ is therefore a finite number.

 The following result is proven in Section~\eref{Pf-Mbeta-trans}. 
\begin{prop}\label{Mbeta-trans}
Assume that Assumption~\ref{Hypo-Trans} is satisfied and $\overline \Gamma$ is finite. Let $t$ be a $(1-\alpha/2)$-quantile of $q$ with $\alpha\le 1/2$. The set $\sM(\beta)$  contains the subset $\ac{P_{\theta}, \theta\in [-\sigma t,\sigma t]}$ and therefore satisfies  \eref{eq-sect-trans0} if
%Equation
\begin{equation}\label{eq-beta}
\beta\ge \overline \beta=\sqrt{\frac{1}{n\gamma}\max\ac{\log\pa{\frac{\overline \Gamma \pa{\sigma \vee 1}}{q(2 t)}},\log 4}}.
\end{equation}
\end{prop}

Let us now comment on this result. The quantity $\overline \beta$ may be written as $C/\sqrt{n}$ with  
%
%
%For a given prior $\nu_{\sigma}$, we see that $\overline \beta$ is of order $1/\sqrt{n}$ and by choosing $\beta=\overline \beta$,  our posterior $\widehat \pi_{\bsX}$ concentrates around $\overline P\et$ with a radius of order 
%%
%\[
%r=\inf_{P\in\sM(\overline \beta)}\ell(\overline P\et,P)+\frac{C}{\sqrt{n}}\le \inf_{\theta\in [-\sigma t,\sigma t]}\ell(\overline P\et,P_{\theta})+\frac{C}{\sqrt{n}}
%\]
%%
%where
%
\[
C=\sqrt{\frac{1}{\gamma}\max\ac{\log\pa{\frac{\overline \Gamma \pa{\sigma \vee 1}}{q(2 t)}},\log 4}}.
\]
%
%We note that the larger $I=[-\sigma t,\sigma t]$, the smaller the approximation term $ \inf_{\theta\in [-\sigma t,\sigma t]}\ell(\overline P\et,P_{\theta})$. 
Increasing the value of $\sigma$ or that of $t$ enlarges the interval $I=[-\sigma t,\sigma t]$. It also makes the value of $C=C(\sigma,t)$ larger. Increasing $\sigma$ makes the prior $\nu_{\sigma}$ flatter and for a fixed value of $t>0$, $C=C(\sigma)$ increases as $\sqrt{\log \sigma}$ when $\sigma$ is larger than 1. In the other case, for a fixed value of $\sigma$, $C=C(t)$ increases as $\sqrt{\log(1/q(2t))}$. For example, when $q$ is the density of a standard Gaussian random variable, $\sqrt{\log(1/q(2t))}$ is of order $t$, while for the Laplace and the Cauchy distributions it is of order $\sqrt{t}$ and $\sqrt{\log t}$ respectively. This result illustrates the fact that it is safer to use priors with heavy tails when the size of the location parameter is uncertain. In case of a light-tailed prior, it may be wise to introduce a scaling parameter $\sigma>1$. 
By taking $\sigma=10$, the concentration radius only increases by a factor less than 1.6, while the interval $I$ is ten times longer.
\subsection{Fast rates}\label{sect-FR}
We go back to the statistical framework described in Section~\ref{sect-trans00} and consider the special case of the density $p:x\mapsto s x^{s-1}\1_{(0,1]}$ with $s\in (0,1]$. As before, we choose the TV-loss. In this specific situation, 
%Equation
\begin{equation}\label{eq-ex30}
\norm{P_{\theta}-P_{\theta'}}=\ab{\theta-\theta'}^{s}\wedge 1\quad \text{for all $\theta,\theta'\in\R$}
\end{equation}
and consequently, $\varphi(r)=r^{1/s}$ for all $r\in [0,1)$. Besides, the family $\sT(\ell,\sM)$ given by~\eref{eq-t-TV} satisfies not only Assumption~\ref{Hypo-1} but also Assumption~\ref{Hypo-2} with $a_{2}=1$. These two facts are proven in Baraud~\citeyearpar{BY-TEST} [Examples 5 and 6]. As a consequence, Theorem~\ref{main2} applies. The reader can check that the constants $\cc=\beta=0.1$ and $\gamma=0.01$ satisfy the requirements of Theorem~\ref{main2} and that its conclusion holds true with $\kappa_{0}=144$.

In order to be more specific about the concentration radius of our posterior $\widehat \pi_{\bsX}^{{\rm TV}}$, the following Proposition provides an upper bound for the quantity $r_{n}(\beta,P_{\theta})$. The proof is postponed to Section~\ref{Sect-PLemex30}.
\begin{prop}\label{prop-priorex3}
Let $t_{0}$ be the third quartile of $\nu_{1}$. If the density $q$ is positive, symmetric and decreasing on $[0,+\infty)$, for all $\theta\in\R$ the quantity $r_{n}(\beta,P_{\theta})$ is not larger than
%Equation
\begin{equation}\label{eq-prop-priorex3}
\overline r_{n}(\beta,P_{\theta})= \frac{2000}{ n}\max\ac{\log\pa{\frac{\overline \Gamma \pa{\sigma \vee 1}}{{q\cro{2\pa{\frac{|\theta|}{\sigma}\vee t_{0}}}}}},\log 4}.
\end{equation}
\end{prop}
Then, our Theorem~\ref{main2} tells us that for all $\xi>0$, with a probability at least $1-2e^{-\xi/2}$, the posterior satisfies 
\[
\widehat \pi_{\bsX}^{{\rm TV}}\pa{\sB(\overline P\et,144r)}\ge 1-e^{-\xi/2}
\]
with 
%
%Equation
\begin{equation}\label{eq-defr-TVspecial}
\ray\le \inf_{\theta\in\R}\cro{\norm{\overline P\et-P_{\theta}}+\overline r_{n}(\beta,P_{\theta})}+\frac{40\xi}{n}.
\end{equation}
When the data are i.i.d.\ with distribution $P_{\theta\et}$, with probability close to 1, a randomized estimator $P_{\widehat \theta}$ with distribution $\widehat \pi_{\bsX}^{{\rm TV}}$ satisfies with high probability 
%Align
\begin{align*}
\ab{\theta\et-\widehat \theta}^{s}\wedge 1&=\norm{P_{\theta\et}-P_{\widehat \theta}}\le \frac{C(\xi,s,q,\theta\et,\sigma)}{n}.
\end{align*}
This inequality implies, at least for $n$ large enough, that
%Align
\begin{align*}
\ab{\theta\et-\widehat \theta}\le \frac{C^{1/s}(\xi,s,q,\theta\et,\sigma)}{n^{1/s}},
\end{align*}
which means that the parameter $\theta\et$ is estimated at rate $n^{-1/s}$. This rate is much faster than the usual $(1/\sqrt{n})$-parametric one that is reached by an estimator based on a moment method for instance. For example, when $s=1/3$ and $n=100$, a moment estimator provides an accuracy of order $10^{-1}$ while that of $\widehat \theta$ is of order $10^{-6}$. Since $p$ is unbounded, note that the maximum likelihood estimator for $\theta\et$ does not exist and is therefore useless. 

It follows from the work of Le Cam that in a translation model $\sM$ of the form $\{P_{\theta}=p(\cdot-\theta)\cdot \mu, \theta\in\R\}$, where $p$ is a density with respect to the Lebesgue measure $\mu$, it is impossible to estimate a distribution $P\et\in\sM$ from an $n$-sample at a rate faster than $1/n$ for the TV-loss. Because of~\eref{eq-ex30}, the rate we get is not only optimal for estimating $P_{\theta\et}$ but also for estimating $\theta\et$ with respect to the Euclidean distance. 

An alternative rate-optimal estimator for estimating $\theta\et$  is that given by the minimum of the observations. This estimator is unfortunately obviously non-robust to the presence of an outlier among the sample. Our construction provides an estimator which possesses the property of being both rate-optimal and robust.

%When the data are i.i.d.\ with distribution $P_{\theta\et}$, we deduce from~\eref{eq-defr-TVspecial} that the concentration radius that is not larger than $\overline r_{n}(\beta,P_{\theta\et})+(40\xi)/n$, with $\overline r_{n}(\beta,P_{\theta\et})$ given by~\eref{eq-prop-priorex3} and 
It also interesting to see how the quantity $\overline r_{n}(\beta,P_{\theta})$ given in \eref{eq-prop-priorex3} deteriorates under a misspecification of the prior $\nu_{\sigma}$, that is, when the size of the parameter $\theta\et$ is large compared to $\sigma$. When $q$ is Gaussian, $\overline r_{n}(\beta,P_{\theta\et})$ increases by a factor of order $(\theta\et/\sigma)^{2}$ while for the Laplace and Cauchy distributions it is of order $|\theta\et|/\sigma$ and $\log(|\theta\et|/\sigma)$ respectively. From these results, we conclude as before that the Cauchy distribution possesses some advantages over the other two distributions when little information is available on the location of the parameter $\theta\et$. 

\subsection{A general result under entropy}
In this section, we equip $E=\R^{k}$ with the Lebesgue measure $\mu$ and the norm $\ab{\cdot}_{\infty}$. We consider the TV-loss  and the location-scale family 
\begin{equation}\label{eq-defLF}
\sM=\ac{P_{(p,\gm,\sigma)}=\frac{1}{\sigma^{k}}p\pa{\frac{\cdot-\gm}{\sigma}}\cdot \mu,\ p\in\cM_{0},\; \gm\in\R^{k}, \sigma>0},
\end{equation}
where $\cM_{0}$ is a set of densities on $\R^{k}$. Given independent observations $\etc{X}$ with presumed distribution $P\et=P_{(p\et,\gm\et,\sigma\et)}\in \sM$, our aim is to estimate the density $p\et\in\cM_{0}$, the location parameter $\gm\et\in\R^{k}$ and the scale parameter $\sigma\et>0$, hence the parameter $\theta\et=(p\et,\gm\et,\sigma\et)\in \Theta=\cM_{0}\times \R^{k}\times (0,+\infty)$. We assume that the set of densities $\cM_{0}$ satisfies the following conditions.
\begin{ass}\label{hypo-entro}
Let $\widetilde D$ be a continuous nonincreasing mapping from $(0,+\infty)$ to $[1,+\infty)$ such that $\lim_{\eta\to +\infty}\eta^{-2}\widetilde D(\eta)=0$. For all $\eta>0$, there exists a finite subset $\cM_{0}[\eta]\subset \cM_{0}$ satisfying
%Equation
\begin{equation}\label{eq-cardi}
\card{\cM_{0}[\eta]}\le \exp\cro{\widetilde D(\eta)}
\end{equation}
such that for all $p\in\cM_{0}$, there exists $\overline p\in \cM_{0}[\eta]$ that satisfies
%Equation
\begin{equation}\label{eq-trans}
\norm{P_{(p,\bs{0},1)}-P_{(\overline p,\bs{0},1)}}=\frac{1}{2}\int_{\R^{k}}\ab{p-\overline p}d\mu\le \eta.
\end{equation}
Besides, we assume that there exist $A,s>0$ such that for all $p\in\cM_{0}$, $\gm\in\R^{k}$ and $\sigma\ge 1$, 

%Equation
\begin{equation}\label{app00}
\norm{P_{(p,\bs{0},1)}-P_{(p,\gm,\sigma)}}\le \cro{A\pa{\left(\ab{\frac{\gm}{\sigma}}_{\infty}\right)^{s}+\pa{1-\frac{1}{\sigma}}^{s}}}\bigwedge 1.
\end{equation}
\end{ass}

The first part of Assumption~\ref{hypo-entro}, which corresponds to inequalities~\eref{eq-cardi} and~\eref{eq-trans}, aims at measuring the size of the set $\cM_{0}$ 
by means of its entropy. The entropy of a set controls its metric dimension and usually determines the minimax rate of convergence over it as shown in Birg\'e~\citeyearpar{MR722129}. With the second part of Assumption~\ref{hypo-entro}, namely inequality~\eref{app00}, we require some regularity properties of the TV-loss with respect to the location and scale parameters. It will be commented on later. We shall see that this condition may be satisfied even when the densities in $\cM_{0}$ are not smooth. 

Let us now turn to the choice of our prior. We first consider a countable subset of the parameter space $\Theta$ that will be proven to possess good approximation properties. Namely,  we define for $\eta,\delta>0$ 
\[
\Theta[\eta,\delta]=\ac{\pa{\overline p,(1+\delta)^{j_{0}}\delta\gj,(1+\delta)^{j_{0}}},\; (\overline p,j_{0},\gj)\in\cM_{0}[\eta]\times \Z\times \Z^{k}}
\]
and we associate a positive weight $L_{\theta}$ with any element $\theta=\theta(\overline p,j_{0},\gj)\in \Theta[\eta,\delta]$ as follows 
%Equation
\begin{equation}\label{def-poids}
L_{\theta}=(k+1)L+\log\ab{\cM_{0}[\eta]}+2\sum_{i=0}^{k}\log(1+|j_{i}|)
\end{equation}
with $L=\log\cro{(\pi^{2}/3)-1}$. It is not difficult to check that $\sum_{\theta\in \Theta[\eta,\delta]}e^{-L_{\theta}}=1$, and we may therefore endow $\sM$ with the (discrete) prior $\pi$ defined as
%Equation
\begin{equation}\label{def-pi00}
\pi(\ac{P_{\theta}})=e^{-L_{\theta}}\quad \text{for all $\theta\in \Theta[\eta,\delta]$.}
\end{equation}

With such a  prior, our posterior $\widehat \pi_{\bsX}^{{\rm TV}}$ given in Corollary~\ref{cor-TV} possesses the following properties.
\begin{cor}\label{cor-Entrop}
Let $\xi>0$ \and $K>1$. Assume that $\cM_{0}$ satisfies Assumption~\ref{hypo-entro} and define
%Align
\begin{align}
\eta&=\eta_{n}=\inf\sD_{n}\quad \text{with}\quad \sD_{n}=\ac{\eta>0,\; \widetilde D(\eta)\le \frac{n\eta^{2}}{24}}\label{eq-etan}\\
\delta&=\delta_{n}=\pa{\frac{\eta_{n}}{2A}}^{1/s},\label{eq-deltan}\\
\beta&=\beta_{n}= \frac{1}{2}\cro{K\eta_{n}+2\sqrt{\frac{18.6(k+1)}{n}}}\label{def-b00}
\end{align}
and the subset $\sM_{n}(K)$ of $\sM$ that consists of the elements $P_{(p,\gm,\sigma)}$ for which 
%Equation
\begin{equation}\label{def-sMK}
|\log\sigma|\vee \ab{\frac{\gm}{\sigma}}_{\infty}\le \Lambda_{n}=\exp\cro{\frac{(K^{2}-1)n\eta_{n}^{2}}{48(k+1)}+\log\log(1+\delta_{n})}.
\end{equation}
Then, the posterior $\widehat \pi_{\bsX}^{{\rm TV}}$ satisfies the following property: there exists a numerical constant $\kappa_{0}'>0$ such that for all $\xi>0$,
\begin{equation}\label{eq-cor-Entrop0}
\E\cro{\widehat \pi_{\bsX}\pa{\co{\sB}(\overline P\et,\kappa_{0}'\ray_{n})}}\le 2e^{-\xi}
\end{equation}
with
%Equation
\begin{equation}\label{eq-cor-Entrop}
\ray_{n}=\inf_{P\in\sM_{n}(K)}\ell(\overline P\et,P)+K\eta_{n}+\sqrt{\frac{k+1}{n}}+\frac{\xi}{\sqrt{n(k+1)}}\wedge \frac{\xi}{Kn\eta_{n}}.
\end{equation}
\end{cor}

Let us now comment on this result. The radius $\ray_{n}$ is the sum of three main terms, omitting the dependency with respect to $\xi$. The first one, $\inf_{P\in\sM_{n}(K)}\ell(\overline P\et,P)$, corresponds to the approximation of $\overline P\et$ by an element of $\sM$ whose location and scale parameters satisfy the constraints given in \eref{def-sMK}. The quantity $\eta_{n}$, involved in the second term, usually corresponds to the minimax rate for solely estimating a density $p\in\cM_{0}$ from an $n$-sample. Finally, the third term $\sqrt{(k+1)/n}$ corresponds to the rate we would get for solely estimating the location and translation parameters $(\gm,\sigma)\in\R^{k+1}$ when the density $p$ is known.

%As a particular case, Corollary~\ref{cor-Entrop} includes the elementary situation where $\cM_{0}$ reduces to a single density $p$ that satisfies~\eref{app00}. Then, $\ab{\cM_{0}[\eta]}=1$ for all $\eta>0$, one may take $\widetilde D=1$, $\eta_{n}=\sqrt{24/n}$ and $\delta_{n}=(A^{2}n/6)^{-1/(2s)}$. For $n\ge 6/A^{2}$, so that $\delta_{n}\le 1$, and by taking 
%%
%\[
%K=\sqrt{1+2(C+1)(k+1)\log \frac{2}{\delta_{n}}}=\sqrt{1+\frac{(C+1)(k+1)}{s}\log \pa{\frac{4^{s}A^{2}n}{6}}}
%\]
%%
%with $C>0$, we obtain that all the probabilities $P_{p,\gm,\sigma}$ that satisfy 
%%Align
%\begin{align*}
%|\log\sigma|\vee \ab{\frac{\gm}{\sigma}}_{\infty}&\le \pa{\frac{4^{s}A^{2}n}{6}}^{C/(2s)}=\exp\cro{\frac{(K^{2}-1)n\eta_{n}^{2}}{48(k+1)}-\log\frac{2}{\delta_{n}}}
%\end{align*}
%%
%also satisfy~\eref{def-sMK} since 
%%
%\[
%-\log\frac{2}{\delta_{n}}\le -\log\frac{1+\delta_{n}}{\delta_{n}}\le\log  \log\pa{1+\delta_{n}}.
%\]
%%
%The posterior then satisfies \eref{eq-cor-Entrop0} 

Let us now provide some examples for which our condition~\eref{app00} is satisfied. We start with an example where the densities in $\cM_{0}$ are smooth. 
\begin{lem}\label{lem-app00}
Assume that the set $\cM_{0}$ consists of densities $p$ that are supported on $[0,1]^{k}$, satisfy $\sup_{p\in\cM_{0}}\norm{p}_{\infty}\le L_{0}$ and 
%Equation
\begin{equation}\label{app00b}
\sup_{p\in\cM_{0}}\ab{p(\gx)-p(\gx')}\le L_{1}\ab{\gx-\gx'}^{s}\quad \text{for all $\gx,\gx'\in \R^{k}$,}
\end{equation}
with constants $L_{0},L_{1}>0$ and $s\in (0,1]$. Then \eref{app00} is satisfied with $A=L_{1}\vee [(1+L_{1}k^{s/2}+L_{0})/2]$.
\end{lem}

Nevertheless, condition~\eref{app00} may also be satisfied for families $\cM_{0}$ of densities which are not smooth, as shown in  Lemma~\ref{reg00} below. It makes it possible to consider the following example.
\begin{exa}\label{examonotone}
We consider here the situation where $k=1$ and $\cM_{0}$ is the set of all nonincreasing densities on $[0,1]$ that are bounded by $B>1$. Then, $\sM$ consists of all the probabilities whose densities are supported on intervals $I$ with positive lengths, nonincreasing on $I$ and which are bounded by $B/\mu(I)$. 
Birman and Solomjak~\citeyearpar{MR0217487} proved that $\cM_{0}$ satisfies Assumption~\ref{hypo-entro} with $\widetilde D(\eta)$ of order $(1/\eta)\vee 1$ (up to some constant that depends on $B$). We deduce from~\eref{eq-etan} that $\eta_{n}$ is therefore of order $n^{-1/3}$. Besides, it follows from Lemma~\ref{reg00} below that \eref{app00} is satisfied with $A=B$ and $s=1$. We may therefore apply Corollary~\ref{cor-Entrop}. For a value of $K$ large enough compared to 1, $\Lambda_{n}$ defined by \eref{def-sMK} is larger than $\exp\cro{CK^{2}n^{1/3}}$ for some constant $C>0$ (depending on $A$). In particular, if $X_{1},\ldots,X_{n}$ are i.i.d.\ with a density of the form
\[
x\mapsto p\et(x)=\frac{1}{\sigma\et}p\pa{\frac{x-m\et}{\sigma\et}}
\]
where $p\in\cM_{0}$, $|m\et/\sigma\et|\le \exp\cro{CK^{2}n^{1/3}}$ and 
\[
\exp\cro{-\exp\cro{CK^{2}n^{1/3}}}\le \sigma\et \le \exp\cro{\exp\cro{CK^{2}n^{1/3}}},
\]
\eref{eq-cor-Entrop0} is satisfied with  $r_{n}$ of order $C'n^{-1/3}$ where the constant $C'>0$ only depends on $\xi,K,B$ but not on $m\et$ and $\sigma\et$. This means that the concentration properties of $\widehat \pi_{\bsX}$ hold true uniformly over a huge range of translation and scale parameters $\gm$ and $\sigma$ when $n$ is large enough.

\begin{lem}\label{reg00}
Let $p$ be a nonincreasing density on $(0,+\infty)$. For all $\sigma\ge 1$
%Equation
\begin{equation}\label{eq-reg002}
\frac{1}{2}\int_{\R}\ab{\frac{1}{\sigma}p\pa{\frac{x}{\sigma}}-p(x)}dx\le \pa{1-\frac{1}{\sigma}}.
\end{equation}
If, furthermore, $p$ is bounded by $B\ge 1$, for all $m\in\R$, 
%Equation
\begin{equation}\label{eq-reg001}
\frac{1}{2}\int_{\R}\ab{p(x)-p(x-m)}dx\le (|m|B)\wedge 1.
\end{equation}
In particular, for all $m\in\R$ and $\sigma\ge 1$, 
\begin{equation}\label{eq-reg003}
\frac{1}{2}\int_{\R}\ab{\frac{1}{\sigma}p\pa{\frac{x-m}{\sigma}}-p(x)}dx\le \cro{B\ab{\frac{m}{\sigma}}+\pa{1-\frac{1}{\sigma}}}\wedge 1 .
\end{equation}
\end{lem}
\end{exa}

\subsection{Estimating a parameter under sparsity}
Let us consider a parametric dominated  model $\sM=\ac{P_{\gtheta}=p_{\gtheta}\cdot\mu,\; \bst\in \R^{k}}$ where the dimension $k$ of the parameters  is large. 
We presume, even though this might not be true, that the data are i.i.d.\ with distribution $P_{\bst\et}\in\sM$ and that the coordinates of the true parameter $\bst\et=(\theta_{1}\et,\ldots,\theta_{k}\et)$ are all zero except for a small number of them. Our aim is to estimate $P_{\bst\et}$ from the observation of $\etc{X}$ by using the squared Hellinger loss. 

To tackle this problem, we partition the model $\sM$ into the sub-models $\{\sM_{m},\; m\subset \{1,\ldots,k\}\}$ where $\sM_{m}$ consists of those distributions $P_{\bst}\in\sM$ for which the coordinates of $\bst=(\theta_{1},\ldots,\theta_{k})$ are all zero except those with an index $i\in m$. We denote by $\Theta_{m}$ the set of such parameters, so  that $\sM_{m}=\{P_{\bst},\; \bst\in \Theta_{m}\}$, and we use the conventions $\Theta_{\vide}=\{\0\}$ and $\sM_{\vide}=\{P_{\0}\}$. Given some positive number $R>0$, 
we equip each parameter space $\Theta_{m}$, $m\subset \{1,\ldots,k\}$, with the uniform distribution $\nu_{m}$ on $\Theta_{m}(R)=[-R,R]^{k}\cap \Theta_{m}$ when $m\ne \vide$ and the Dirac mass $\nu_{\vide}=\delta_{\0}$ at $\0\in\R^{k}$ when $m=\vide$. We may then define on $\R^{k}=\bigcup_{m\subset \{1,\ldots,k\}}\Theta_{m}$, the hierarchical prior
%Equation
\begin{equation}\label{def-nu}
\nu=\sum_{m\subset \{1,\ldots,k\}}e^{-L_{m}}\nu_{m}\quad \text{with}\quad L_{m}=|m|\log k+k\log\pa{1+\frac{1}{k}}.
\end{equation}

We endow $\sM$ with the $\sigma$-algebra and the prior $\pi$ as described in Section~\ref{sect-scpm}. Besides, we assume that there exists $s\in (0,1]$ and a positive number $B_{k}=B_{k}(R)$, possibly depending on $k$ and $R$ (although we drop the dependency with respect to $R$), such that 
%Equation
\begin{equation}\label{cond-hell}
%\sqrt{2A_{k}}\pa{\ab{\bst-\bst'}_{\infty}^{s}\wedge 1}\le 
h^{2}\pa{P_{\bst},P_{\bst'}}\le B_{k}\ab{\bst-\bst'}_{\infty}^{2s}\quad \text{for all $\bst,\bst'\in [-R,R]^{k}$.}
\end{equation}
The following result is proven in Section~\ref{Proof-cor-sparcity}.
\begin{prop}\label{cor-sparcity}
Assume that 
\[
\map{p}{E\times \R^{k}}{\R_{+}}{(x,\bst)}{p_{\bst}(x)}
\]
is measurable. If $RB_{k}^{1/s}\ge 1$ there exists a numerical constant $\kappa_{0}'>0$ such that for any distribution $\gP\et$ and $\xi>0$
\[
\E\cro{\widehat \pi_{\bsX}^{h}\pa{\co{\sB}(\overline P\et,\kappa_{0}'\ray)}}\le 2e^{-\xi}
\]
where 
%Equation
\begin{equation}\label{def-r-cor-sparcity}
r=\inf_{m\subset \{1,\ldots,k\}}\cro{\inf_{\bst\in \Theta_{m}(R)}\ell(\overline P\et,P_{\bst})+\frac{|m|\log\pa{2kR(nB_{k})^{1/s}}+\xi}{n}}.
\end{equation}
\end{prop}

Let us now comment on this result. First of all, the mapping 
\[
R\mapsto \sup\ac{\frac{h^{2}\pa{P_{\bst},P_{\bst'}}}{\ab{\bst-\bst'}_{\infty}^{s}},\; \bst\neq \bst,'\; \bst,\bst'\in [-R,R]^{k}}
\]
being nondecreasing, our condition $RB_{k}^{1/s}=R[B_{k}(R)]^{1/s}\ge 1$ is always satisfied for a value of $R$ sufficiently large. 

When $B_{k}$ does not increase faster than a power of $k$, the radius $r$ given in \eref{def-r-cor-sparcity} only depends logarithmically on the dimension $k$ of the parameter space, as expected. 

Let us now illustrate Proposition~\ref{cor-sparcity} by choosing some specific models $\sM=\{P_{\bst},\; \bst\in\R^{k}\}$. If $P_{\bst}$ is the Gaussian distribution with mean $\bst\in\R^{k}$ and covariance matrix $\sigma^{2} I_{k}$, where $I_{k}$ denotes the $k\times k$ identity matrix,  
%Align
\begin{align*}
h^{2}(P_{\bst},P_{\bst'})=1-\exp\cro{-\frac{\ab{\bst-\bst'}^{2}}{8\sigma^{2}}}\le \frac{\ab{\bst-\bst'}^{2}}{8\sigma^{2}}\le \frac{k\ab{\bst-\bst'}_{\infty}^{2}}{8\sigma^{2}}.
\end{align*}
Then, inequality \eref{cond-hell} is satisfied with $B_{k}=k/(8\sigma^{2})$ and $s=2$. In particular, our condition $RB_{k}^{1/s}\ge 1$ is equivalent to $R\ge 2\sigma\sqrt{(2/k)}$. In this case, the value of $r$ given by~\eref{def-r-cor-sparcity} is of order 
\[
\inf_{m\subset \{1,\ldots,k\}}\cro{\inf_{\bst\in \Theta_{m}(R)}\ell(\overline P\et,P_{\bst})+\frac{|m|\log\pa{knR/\sigma}+\xi}{n}}.
\]
More generally, if $\sM=\{P_{\bst},\bst\in\R^{k}\}$ is a regular statistical model with a nonsingular Fisher information matrix $\gJ(\bst)$ for all $\bst\in \R^{k}$, we know from the book of Ibragimov and Has'minski\u{\i}~\citeyearpar{MR620321}[Theorem 7.1 p.81]  that for all $\bst,\bst'\in\R^{k}$ such that $\bst,\bst'\in [-R,R]^{k}$
\[
h^{2}(P_{\bst},P_{\bst'})\le \frac{\ab{\bst-\bst'}^{2}}{8}\sup_{\bst''\in\R^{k}, \ab{\bst''}_{\infty}\le R}{\rm tr}\pa{\gJ(\bst'')}.
\]
Then, Assumption~\eref{cond-hell} holds with $s=2$ and we may take
\[
B_{k}=\frac{k^{2}}{8}\sup_{\bst''\in\R^{k}, \ab{\bst''}_{\infty}\le R}\varrho\pa{\gJ(\bst'')}
\]
where $\varrho\pa{\gJ(\bst'')}$ denotes the largest eigenvalue of the matrix $\gJ(\bst'')$. This value is independent of $\bst''$ when $\sM$ is a translation model. 

Finally note that the second term in \eref{def-r-cor-sparcity} only increases logarithmically with respect to $R$, at least when $B_{k}=B_{k}(R)$ does not increase faster than a power of $R$. By taking larger values of $R$ one may therefore considerably enlarge the sizes of the cubes $\Theta_{m}(R)$, and therefore diminish the approximation term in \eref{def-r-cor-sparcity}, while only slightly increasing the second  term $[|m|\log(2kR(nB_{k})^{1/s})+\xi]/n$.

\section{Some tools for evaluating $\ray_{n}(\beta,P)$}\label{sect-copaic}
The aim of this section is to provide some mathematical results that allow one to bound the quantity $\ray_{n}(\beta,P)$ from above, or at least evaluate its order of magnitude, when $n$ is sufficiently large. Throughout this section, we consider a parametric statistical model $\sM=\{P_{\gtheta},\; \gtheta\in \Theta\}$ where the parameter space $\Theta\subset \R^{k}$ is endowed with a prior $\nu$ which admits a density $q$ with respect to the Lebesgue measure on $\R^{k}$. In order to use the definition~\eref{df-epsn} of the quantity $\ray_{n}(\beta,P)$, we assume that we have at disposal a family $\sT(\ell,\sM)$ that satisfies our Assumption~\ref{Hypo-1}, which provides us with a value of $a_{1}>0$, as well as a value $\gamma$ that satisfy the requirements of our main theorems. Our aim is to bound $\ray_{n}(\beta,P)$ as a function of $a_{1},\gamma,\beta, k$ and $n$ under suitable assumptions on the density $q$ and the behaviour of the loss $\ell$. Once $\ell$ and $\sT(\ell,\sM)$ are given,  $a_{1}$ and $\gamma$ can be considered as fixed numerical constants. The value of $\beta$ can also be considered as a numerical constant when Theorem~\ref{main2} applies. Otherwise, it can be chosen of order $\sqrt{k/n}$ as in our Example~\ref{ex-euclidien}.

\subsection{Bounding $\ray_{n}(\beta,P_{\gtheta})$ in parametric models}\label{sect-BrnNA}
In what follows, $\ab{\cdot}_{*}$ denotes some arbitrary norm on $\R^{k}$ and $\cB_{*}(\gx,z)$ the corresponding closed ball centered at $\gx\in\R^{k}$ with radius $z\ge 0$. 
% Hypothèse
\begin{ass}\label{Ass-param}
Let $\gtheta\et$ be an element of $\Theta\subset \R^{k}$.  
\begin{listi}
\item\label{Ass-param-i} There exist positive numbers $\underline a,\overline a$ and $\es$ such that 
\begin{equation}
\underline a\ab{\gtheta-\gtheta\et}_{*}^{s}\le \ell(\gtheta,\gtheta\et)\le \overline a\ab{\gtheta-\gtheta\et}_{*}^{s}\quad\mbox{for all }\gtheta\in \Theta.
\label{eq-conEh}
\end{equation}
\item\label{Ass-param-ii} There exists a positive nonincreasing function  $\upsilon_{\gtheta}$ on $\R_{+}$ such that 
%beg
\begin{equation}
\nu(\cB_{*}(\gtheta\et,2x))\le\upsilon_{\gtheta\et}(x)\nu(\cB_{*}(\gtheta\et,x))\quad\mbox{for all }x>0.
\label{eq-hypnu}
\end{equation}
%end
\end{listi}
\end{ass}
Under Assumption~\ref{Ass-param}-\ref{Ass-param-i}, the loss function behaves like a power of a norm between the parameters. 

The following result is an extension of Proposition~10 in Baraud and Birg\'e~\citeyearpar{BarBir2020}. It was established there for the special case of the squared Hellinger loss and we provide here an extension to an arbitrary one. Since the proof follows the same lines, we omit it.
%
% PROPOSITION
\begin{prop}\label{casconv}
Under Assumption~\ref{Ass-param},
%Align
\begin{align}
\ray_{n}(\beta,P_{\gtheta\et})\le\inf\ac{r\ge \frac{1}{n\beta a_{1}},\; r\ge \frac{\varrho_{0}\log\cro{\upsilon_{\gtheta\et}\pa{[r/\overline a]^{1/\es}}}}{\gamma n \beta a_{1}}}\label{eq-eta-param0}
\end{align}
%end
%
with  $\varrho_{0}=1+\log(2\overline a/\underline a)/[\es \log 2]$. If $\upsilon_{\gtheta\et}\equiv \upsilon>0$, then 
%beg
\begin{equation}
\ray_{n}(\beta,P_{\gtheta\et})\le \frac{\pa{\varrho_{0}\log \upsilon}\vee 1}{a_{1} n \gamma\beta }.
\label{eq-eta-param9}
\end{equation}

If Assumption~\ref{Ass-param}-\ref{Ass-param-i} is satisfied and if the parameter space $\gTheta$ is convex and $q$  satisfies 
%beg
\begin{equation}
\underline b\le q(\gtheta)\le \overline b\quad\mbox{for all }\gtheta\in \gTheta
\quad\text{with }0<\underline b\le \overline b,
\label{eq-densborn}
\end{equation}
%end
then Assumption~\ref{Ass-param}-\ref{Ass-param-ii} holds with $\upsilon_{\gtheta\et}\equiv 2^{k}(\overline b/\underline b)$. Consequently, 
%beg
\begin{equation}
\ray_{n}(\beta,P_{\gtheta\et})\le \frac{\varrho_{1}}{a_{1}\gamma}\frac{k}{n\beta}\quad \text{with}\quad \varrho_{1}=\cro{\varrho_{0}\log\left(2\left[\overline b/\underline b\right]^{1/k}\right)}\vee 1.
\label{eq-eta-param}
\end{equation}
%end
\end{prop}
When Assumption~\ref{Ass-param}-\ref{Ass-param-i}  is satisfied and $\nu$ admits a density which is bounded away from 0 and infinity on a convex parameter space $\Theta\subset \R^{k}$, $\ray_{n}(\beta,P_{\gtheta})$ is of order $k/(n\beta)$ for all $\gtheta\in \Theta$. This result may also hold true when the density is not bounded away from 0 as shown in the following example. If $k=1$, $\Theta=[-1,1]$ and $q:\theta\mapsto (t/2)|\theta|^{t-1}\1_{[-1,1]}(\theta)$ with $t\in (0,1)$, Assumption~\ref{Ass-param}-\ref{Ass-param-ii} holds with $\upsilon_{\theta}\equiv 2^{1+t}\left(2^{t}-1\right)^{-1}$ for all $\theta\in [-1,1]$ -- see Baraud and Birg\'e~\citeyearpar{BarBir2020}[Proposition~10]. Then \eref{eq-eta-param9} applies even though $q$ is not bounded from below on the parameter space. In the other direction, when the density $q$ takes very small values in the neighbourhood of the parameter $\gtheta$, the function $\upsilon_{\gtheta}$ may take large values around 0. This is for example the case when $q$ is proportional to $\theta\mapsto \exp\left[-1/\left(2|\theta|^{t}\right)\right]\1_{[-1,1]}(\theta)$, $t>0$, and $\theta=0$. It follows from Baraud and Birg\'e~\citeyearpar{BarBir2020}[Proposition~12] (and its proof) that Assumption~\ref{Ass-param}-\ref{Ass-param-ii} is satisfied with $\upsilon_{\theta}:x\mapsto \exp(c(t)/x^{t})$ for some quantity $c(t)>0$. Applying~\eref{eq-eta-param0} leads to an upper bound on $\ray_{n}(\beta,P_{\theta})$ of order $(n \beta)^{-s/(s+t)}$.

\subsection{Some asymptotic order of magnitude}\label{sect-BrnA}
In Section~\ref{sect-BrnA}, we have given some general tools for controlling the quantity $r_{n}(\beta,P_{\gtheta})$  for a given value of $n$. In this section, we present some sufficient conditions under which $r_{n}(\beta,P_{\gtheta})$ is of order $k/(n\beta)$ at least when $n$ is large enough. These conditions are not the weakest possible ones but they have the advantage to be relatively easy to check on many examples.

\begin{ass}\label{ass-para-1}
The density $q$ is continuous and positive at $\gtheta\et\in \Theta$. The loss function $\ell$ satisfies the following properties for some positive number $s>0$ and a norm  $\ab{\cdot}_{*}$ on $\R^{k}$. 
\begin{listi}
\item\label{ass-para-1i} For all $\eps>0$, there exists $z=z(\eps)>0$ such that 
\[
(1-\eps)\ab{\gtheta-\gtheta\et}_{*}^{s}\le \ell(\gtheta,\gtheta\et)\le (1+\eps)\ab{\gtheta-\gtheta\et}_{*}^{s}\quad \text{for all $\gtheta\in\cB_{*}(\gtheta\et,z)$.}
\]
\item\label{ass-para-1ii} There exists a subset $\cK\subset \Theta$, the interior of which contains $\gtheta\et$, that satisfies for some positive numbers $\underline a_{\cK}$ and $\eta$:
%Equation
\begin{equation}\label{eq-para-1ii}
\underline a_{\cK}\ab{\gtheta-\gtheta\et}_{*}^{s}\le  \ell(\gtheta,\gtheta\et)\quad \text{for $\gtheta\in \cK$ and for $\gtheta\not \in \cK$}\quad  \ell(\gtheta,\gtheta\et)\ge \eta>0.
\end{equation}
\end{listi}
\end{ass}
Under these assumptions, we establish the following proposition, the proof of which is postponed to Section~\ref{sect-pf-prop-odg}.
\begin{prop}\label{prop-odg}
Under Assumption~\ref{ass-para-1}, at least for $n$ sufficiently large, 
%Equation
\begin{equation}\label{eq-odg}
r_{n}\pa{\beta,P_{\gtheta\et}}\le \frac{(1+1/s)}{a_{1}\gamma}\frac{k}{n\beta}.
\end{equation}
\end{prop}

\subsection{The case of the squared Hellinger loss on a regular statistical model}\label{sect-RegMod}
Of particular interest is the situation where the statistical model $\sM=\{P_{\gtheta},\; \gtheta\in \Theta\}$, $\Theta\subset \R^{k}$,  is regular. There exists several way of defining a regular model in statistics and we adopt here the definition of Ibragimov and Has'minski\u \i~\citeyearpar{MR620321}. 
% Definition
\begin{df}\label{df-regulier}
Let $\mu$ be a measure on $(E,\cE)$ and $\Theta$ an open subset of $\R^{k}$. The statistical model $\sM=\{P_{\gtheta}=p_{\gtheta}\cdot\mu,\; \gtheta\in \Theta\}$ is said to be {\em regular} if the family of functions $\{\zeta_{\gtheta}=\sqrt{p_{\gtheta}},\; \theta\in \Theta\}\subset \sL_{2}(E,\cE,\mu)$ satisfies the following properties.
\begin{listi}
\item For $\mu$-almost all $x\in E$, $\gtheta\mapsto \zeta_{\gtheta}(x)$ is continuous.
\item For all $\gtheta\in \Theta$, there exists $\dot{\bs{\zeta}}_{\gtheta}=(\dot{\zeta}_{\gtheta,1},\ldots,\dot{\zeta}_{\gtheta,k}):E\to \R^{k}$ such that 
\[
\int_{E}\ab{\dot{\bs{\zeta}}_{\gtheta}(x)}^{2}d\mu(x)<+\infty
\]
and
\[
\int_{E}\ab{\zeta_{\gtheta+\bs{\epsilon}}(x)-\zeta_{\gtheta}(x)-\scal{\dot{\bs{\zeta}}_{\gtheta}(x)}{\bs{\epsilon}}}^{2}d\mu(x)=o(\ab{\bs{\epsilon}}^{2})\quad \text{when $\ab{\bs{\epsilon}}\to 0$.}
\]
\item  For all $i\in\{1,\ldots,k\}$, the mapping $\gtheta\mapsto \dot{\zeta}_{\gtheta,i}$ is continuous in $ \sL_{2}(E,\cE,\mu)$.
\end{listi}
When the model is regular, the matrix
\[
\gJ(\gtheta)=\pa{4\int_{E}\dot{\zeta}_{\gtheta,i}(x)\dot{\zeta}_{\gtheta,j}(x)d\mu(x)}_{\atop{1\le i\le k}{1\le j\le k}},
\]
is called the {\em Fisher information matrix}. 
\end{df}
The matrix $\gJ(\gtheta)$ is symmetric and nonnegative and we may therefore consider its square root $\gJ^{1/2}(\gtheta)$, that is, the symmetric $(k\times k)$-nonnegative matrix that satisfies $\gJ^{1/2}(\gtheta)\gJ^{1/2}(\gtheta)=\gJ(\gtheta)$.

Regular statistical models enjoy nice metric properties that are described in Proposition~\ref{prop-regulstat} below. For a proof we refer the reader to Ibragimov and Has'minski\u \i~\citeyearpar{MR620321} -- Lemma~7.1 page 65, Theorem~7.6 page 81 and its proof. 
% Proposition
\begin{prop}\label{prop-regulstat}
Let $\Theta$ be an open subset of $\R^{k}$ and $\gtheta\et\in \Theta$. If $\sM=\{P_{\gtheta}=p_{\gtheta}\cdot\mu,\; \gtheta\in \Theta\}$ is regular  and the Fisher information matrix $\gJ(\gtheta\et)$ nonsingular at $\gtheta\et\in \Theta$, Assumption~\ref{ass-para-1}-\ref{ass-para-1i} is satisfied with $\ell=h^{2}$, $s=2$ and for the norm $\ab{\cdot}_{*}$ defined by 
%Equation
\begin{equation}\label{eq-normh}
\ab{\gx}_{*}=\frac{1}{\sqrt{8}}\ab{\gJ^{1/2}(\gtheta\et)\gx}\quad \text{for all $\gx\in\R^{k}$.}
\end{equation}
Besides, for any compact subset $\cK\subset \Theta$ there exist positive numbers $ \overline a_{\cK},\underline a_{\cK}$ such that 
%Equation
\begin{equation}\label{eq-regulstat01}
\underline a_{\cK}\ab{\gtheta-\gtheta\et}_{*}^{2}\le h^{2}\pa{\gtheta,\gtheta\et}\le \overline a_{\cK}\ab{\gtheta-\gtheta\et}_{*}^{2}\quad \text{for all $\gtheta\in \cK$.}
\end{equation}
%\begin{listi}
%\item\label{regulstat-i} For all $\gtheta\et\in \Theta$ and $\eps>0$, there exists $z=z(\eps)>0$ such that $\cB(\gtheta\et,z)\subset \Theta$ and for all $\gtheta\in \cB(\gtheta\et,z)$
%%
%\[
%\ab{h^{2}\pa{\gtheta,\gtheta\et}-\frac{1}{8}\ab{\gJ^{1/2}(\gtheta\et)\pa{\gtheta-\gtheta\et}}^{2}}\le \eps \ab{\gtheta-\gtheta\et}^{2}.
%\]
%%
%\item\label{regulstat-ii} If the Fisher information matrix $\gJ(\gtheta\et)$ is nonsingular at $\gtheta\et$, for any compact set $\cK\subset \Theta$ there exist positive numbers $ \overline a_{\cK},\underline a_{\cK}$ such that 
%%Equation
%\begin{equation}\label{eq-regulstat01}
%\underline a_{\cK}\ab{\gtheta-\gtheta\et}\le h\pa{\gtheta,\gtheta\et}\le \overline a_{\cK}\ab{\gtheta-\gtheta\et}\quad \text{for all $\gtheta\in \cK$.}
%\end{equation}
%%
%\end{listi}
\end{prop}

Using Proposition~\ref{prop-odg}, we immediately infer the following result.
% Corollaire
\begin{cor}\label{cor-odgh}
Let $\Theta$ be an open subset of $\R^{k}$. Assume that $\sM=\{P_{\gtheta}=p_{\gtheta}\cdot\mu,\; \gtheta\in \Theta\}$ is regular and the Fisher information matrix $\gJ(\gtheta\et)$ nonsingular at $\gtheta\et\in \Theta\subset \R^{k}$. Assume that there exists a compact set $\cK\subset \Theta$, containing $\gtheta\et$ in its interior, such that $h(\gtheta,\gtheta\et)\ge \eta>0$ for all $\gtheta\not \in \cK$. Furthermore, assume that the density $q$ is continuous and positive at $\gtheta\et$. Then, $r_{n}(\beta,P_{\gtheta\et})\le [3/(2a_{1}\gamma \beta)] (k/n)$, at least for $n$ sufficiently large. 
\end{cor}

\section{Proofs of Theorems~\ref{main1}, ~\ref{main2} and \ref{main1b}}\label{sct-PMR}
Throughout this proof we fix some $\PAP\in\sM$, $\ray,\beta>0$ and use  the following notation: $\cc_{1}=1+\cc$, $\cc_{2}=2+\cc$, 
\[
\cV(\pi,\PAP)=\ac{\ray>0,\pi\pa{\sB(\PAP,\ray)}>0}
\]
and for $\ray\in \cV(\pi,\PAP)$ , $\sB=\sB(\PAP,\ray)$ and $\pi_{\sB}=\cro{\pi(\sB)}^{-1}\1_{\sB}\cdot \pi$.

\subsection{Main parts of the proofs of Theorems~\ref{main1} and~\ref{main2}}
Throughout the proofs of these two theorems we fix some positive number $z$, that will be chosen later on, and $r\ge \ray_{n}(\beta,\PAP)$. Then we set 
\[
A=\ac{\int_{\sM}\exp\cro{-\beta \gT(\bsX,P)}d\pi(P)>z}.
\]
It follows from the definition~\eref{def-piX} of $\widehat \pi_{\bsX}$ that for all $J\in\N$
%Align
\begin{align}
\E\cro{\widehat \pi_{\bsX}\pa{\co{\sB}(\PAP,2^{J}\ray)}}&=\E\cro{\widehat \pi_{\bsX}\pa{\co{\sB}(\PAP,2^{J}\ray)}\1_{\co{A}}}+\E\cro{\widehat \pi_{\bsX}\pa{\co{\sB}(\PAP,2^{J}\ray)}\1_{A}}\nonumber\\
&\le \P(\co{A})+\frac{1}{z}\E\cro{\int_{\co{\sB}(\PAP,2^{J}\ray)}\exp\cro{-\beta \gT(\bsX,P)}d\pi(P)}\nonumber\\
&= \P(\co{A})+\frac{1}{z}\int_{\co{\sB}(\PAP,2^{J}\ray)}\E\cro{\exp\cro{-\beta \gT(\bsX,P)}}d\pi(P)\label{eq-02}.
\end{align}
In a first step, we prove that for some well chosen values of $\beta,z,\ray$ and for $J$ large enough, each of the two terms in the right-hand side of~\eref{eq-02} is not larger than $e^{-\xi}$. To achieve this goal, we bound the first term of the right-hand side of~\eref{eq-02} by applying Markov's inequality
%Align
\begin{align*}
\P(\co{A})&=\P\cro{\int_{\sM}\exp\cro{-\beta \gT(\bsX,P)}d\pi(P)\le z}\\
&=\P\cro{\cro{\int_{\sM}\exp\cro{-\beta \gT(\bsX,P)}d\pi(P)}^{-1}\ge  z^{-1}}\\
&\le z\E\cro{\frac{1}{\int_{\sM}\exp\cro{-\beta \gT(\bsX,P)}d\pi(P)}}
\end{align*}
and then by using Lemma~\ref{Denom}, we obtain that
%Align
\begin{align}
\P(\co{A})\le \frac{z}{\pi^{2}(\sB)} \cro{\int_{\sB^{2}}\exp\cro{-\gL(P,Q)}d\pi_{\sB}(P)d\pi_{\sB}(Q)}^{-1}.\label{eq-borne01}
\end{align}
We therefore have a control of $\P(\co{A})$ by choosing $z$ small enough. We bound the second term of~\eref{eq-02} by using Lemma~\ref{LaplaceZ}. 

We then finish the proofs of Theorems~\ref{main1} and~\ref{main2} as follows. In the context of Theorem~\ref{main1}, we finally establish that for a suitable value of $J$ and all $\PAP\in\sM(\beta)$, 
%Align
\begin{align*}
\E\cro{\widehat \pi_{\bsX}\pa{\co{\sB}(\PAP,2^{J}\ray)}}\le 2e^{-\xi}\quad \text{with}\quad r=r(\PAP)=\ell(\overline P\et,\PAP)+a_{1}^{-1}\pa{\beta+\frac{2\xi}{n\beta}}.
\end{align*}
By~\eref{eq-triangle}, $\sB(\PAP,2^{J}r)\subset \sB(\overline P\et,\tau \ell(\overline P\et,\PAP)+\tau 2^{J}r)$ for all $\PAP\in\sM(\beta)$, and consequently $\E\cro{\widehat \pi_{\bsX}\pa{\co{\sB}(\overline P\et,\overline \ray}}\le 2e^{-\xi}$ with 
\begin{align*}
\overline \ray=\overline\ray(\PAP)=\tau\cro{\ell(\overline P\et,\PAP)+2^{J}\ray}=\tau\cro{(1+2^{J})\ell(\overline P\et,\PAP)+2^{J}a_{1}^{-1}\pa{\beta+\frac{2\xi}{n\beta}}}.
\end{align*}
We obtain~\eref{eq-thm01} by monotone convergence, taking a sequence $(\PAP_{N})_{N\ge 0}\subset \sM(\beta)$ such that $\ell(\overline P\et,\PAP_{N})$ is nonincreasing to $\inf_{P\in\sM(\beta)}\ell(\overline P\et,P)$, 
so that 
%Align
\begin{align*}
\lim_{N\to +\infty}\overline \ray(\PAP_{N})&=\tau\cro{(1+2^{J})\inf_{\PAP\in\sM(\beta)}\ell(\overline P\et,\PAP)+2^{J}a_{1}^{-1}\pa{\beta+\frac{2\xi}{n\beta}}}\\
&\le \tau(1+2^{J})\cro{\inf_{\PAP\in\sM(\beta)}\ell(\overline P\et,\PAP)+a_{1}^{-1}\pa{\beta+\frac{2\xi}{n\beta}}}
\end{align*}
and \eref{eq-thm01} holds provided that $\kappa_{0}\ge\tau(2^{J}+1)$.

In the context of Theorem~\ref{main2}, we show that for some suitable value of $J$ and all $\PAP\in\sM$,
\begin{align*}
\E\cro{\widehat \pi_{\bsX}\pa{\co{\sB}(\PAP,2^{J}\ray)}}\le 2e^{-\xi}\quad \text{with}\quad r=\ell(\overline P\et,\PAP)+\ray_{n}(\PAP,\beta)+ \frac{2\xi}{n\beta a_{1}},
\end{align*}
and we get~\eref{eq-thm02} by arguing similarly. 

\subsection{Preliminary results}
In the proofs of Theorems~\ref{main1} and~\ref{main2}, we use the following consequence of our Assumption~\ref{Hypo-1}. We may write
\[
\frac{1}{n}\sum_{i=1}^{n}\E\cro{t_{(P,Q)}(X_{i})}=\E_{S}\cro{t_{(P,Q)}(X)}\quad \text{with $S=\overline P\et=\frac{1}{n}\sum_{i=1}^{n}P_{i}\et\in\sP$}
\]
and we deduce from~\eref{eq-SPQ} that for all $P,Q\in\sM$,
%Equation
\begin{equation}\label{eq-SPQbis}
\frac{1}{n}\sum_{i=1}^{n}\E\cro{t_{(P,Q)}(X_{i})}\le a_{0}\ell(\overline P\et,P)-a_{1}\ell(\overline P\et,Q).
\end{equation}
Besides, using the anti-symmetry property~\ref{cond-1b} we also obtain that 
%Equation
\begin{equation}\label{eq-SPQter}
\frac{1}{n}\sum_{i=1}^{n}\E\cro{t_{(P,Q)}(X_{i})}\ge a_{1}\ell(\overline P\et,P)-a_{0}\ell(\overline P\et,Q).
\end{equation}

In the proof of Theorems~\ref{main2}, we also use the following consequence of our Assumption~\ref{Hypo-2}. By taking $S=\overline P\et$ and using the convexity of the mapping $u\mapsto u^{2}$, we deduce that for all $P,Q\in\sM$
%Align
\begin{align*}
\frac{1}{n}\sum_{i=1}^{n}\Var\cro{t_{(P,Q)}(X_{i})}&=\E_{S}\cro{t_{(P,Q)}^{2}(X)}-\frac{1}{n}\sum_{i=1}^{n}\pa{\E\cro{t_{(P,Q)}(X_{i})}}^{2}\\
&\le \E_{S}\cro{t_{(P,Q)}^{2}(X)}-\pa{\E_{S}\cro{t_{(P,Q)}(X)}}^{2}\\
&=\Var_{S}\cro{t_{(P,Q)}(X)}
\end{align*}
and it derives thus from Assumption 4~\ref{cond-3} that for all $P,Q\in\sM$
%Equation
\begin{equation}\label{cond-3bis}
\frac{1}{n}\sum_{i=1}^{n}\Var\cro{t_{(P,Q)}(X_{i})}\le a_{2}\cro{\ell(\overline P\et,P)+\ell(\overline P\et,Q)}.
\end{equation}

The proofs of our main results rely on the following lemmas.
\begin{lem}\label{lem-Catoni}
Let $(U,V)$ be a pair of random variables with values in a product space $(E\times F,\cE\otimes \cF)$ and marginal distributions $P_{U}$ and $P_{V}$ respectively. For all measurable function $h$ on $(E\times F,\cE\otimes \cF)$, 
\[
\E_{U}\cro{\frac{1}{\E_{V}\cro{\exp\cro{-h(U,V)}}}}\le \cro{\E_{V}\cro{\frac{1}{\E_{U}\cro{\exp\cro{h(U,V)}}}}}^{-1}.
\]
\end{lem}
This lemma is proven in Audibert and Catoni~\citeyearpar{audibert2011linear} [Lemma 4.2, page 28].

\begin{lem}\label{LaplaceZ}
For $P,Q\in\sM$, we set 
\[
\gM(P,Q)=\log\cro{\int_{\sM}\E\cro{\exp\cro{\beta\pa{\cc\gT(\bsX,P,Q')-\cc_{1}\gT(\bsX,P,Q)}}d\pi(Q')}}.
\]
For all $\ray\in \cV(\pi,\PAP)$ and $P\in\sM$, 
%Align
\begin{align}
\E\cro{\exp\cro{-\beta\gT(\bsX,P)}}\le \frac{1}{\pi(\sB)}\cro{\int_{\sB}\exp\cro{-\gM(P,Q)}d\pi_{\sB}(Q)}^{-1}\label{eq-LaplaceZ}.
\end{align}
\end{lem}

\begin{proof}
Let $\ray\in \cV(\pi,\PAP)$. For $P,Q\in\sM$, we set 
\[
I(\bsX,P,Q)=\cc_{1}\beta\gT(\bsX,P,Q)-\log\int_{\sM}\exp\cro{\cc\beta\gT(\bsX,P,Q')}d\pi(Q').
\]
Then, 
%Align
\begin{align}
&\E\cro{\exp\cro{-I(\bsX,P,Q)}}\nonumber\\
&=\E\cro{\exp\cro{-\cc_{1}\beta\gT(\bsX,P,Q)+\log\int_{\sM}\exp\cro{\cc\beta\gT(\bsX,P,Q')}d\pi(Q')}}\nonumber\\
&=\E\cro{\int_{\sM}\exp\cro{\cc\beta\gT(\bsX,P,Q')-\cc_{1}\beta\gT(\bsX,P,Q)}d\pi(Q')}\nonumber\\
&=\exp\cro{\gM(P,Q)}.\label{num-01}
\end{align}

Since $\lambda=\cc_{1}\beta=(1+\cc)\beta$, it follows from the convexity of the exponential that 
%Align
\begin{align*}
\E\cro{\exp\cro{-\beta \gT(\bsX,P)}}&=\E\cro{\exp\cro{\int_{\sM}[-\beta\gT(\bsX,P,Q)]d\widetilde \pi_{\bsX}(Q)}}\\
&\le \E\cro{{\int_{\sM}{\exp\cro{-\beta \gT(\bsX,P,Q)}d\widetilde \pi_{\bsX}(Q)}}}\\
&=\E\cro{\frac{\int_{\sM}\exp\cro{\cc\beta\gT(\bsX,P,Q)}d\pi(Q)}{\int_{\sM}\exp\cro{\cc_{1}\beta\gT(\bsX,P,Q)}d\pi(Q)}}\\
&\le \E\cro{\frac{\int_{\sM}\exp\cro{\cc\beta\gT(\bsX,P,Q)}d\pi(Q)}{\int_{\sB}\exp\cro{\cc_{1}\beta\gT(\bsX,P,Q)}d\pi(Q)}}.
\end{align*}
Hence, 
%Align
\begin{align*}
\E\cro{\exp\cro{-\beta\gT(\bsX,P)}}&\le \E\cro{\frac{1}{\int_{\sB}\exp\cro{I(\bsX,P,Q)}d\pi(Q)}}\\
&=\frac{1}{\pi(\sB)}\E\cro{\frac{1}{\int_{\sB}\exp\cro{I(\bsX,P,Q)}d\pi_{\sB}(Q)}}.
\end{align*}
Applying Lemma~\ref{lem-Catoni} with $U=\bsX$, $V=Q$ with distribution $\pi_{\sB}$, and $h(U,V)=-I(\bsX,P,Q)$, we obtain that 
\begin{align*}
\E\cro{\exp\cro{-\beta\gT(\bsX,P)}}&\le\frac{1}{\pi(\sB)}\cro{\int_{\sB}\frac{1}{\E\cro{\exp\cro{-I(\bsX,P,Q)}}}d\pi_{\sB}(Q)}^{-1}
\end{align*}
and \eref{eq-LaplaceZ} follows from~\eref{num-01}.
\end{proof}

\begin{lem}\label{Denom}
For $P,Q\in\sM$, we set 
\[
\gL(P,Q)=\log \int_{\sM}\E\cro{\exp\cro{\beta\pa{\cc_{2}\gT(\bsX,P,Q')-\cc_{1}\gT(\bsX,P,Q)}}}d\pi(Q').
\]
For all $\ray\in\cV(\pi,\PAP)$, 
%Align
\begin{align*}
&\E\cro{\frac{1}{\int_{\sM}\exp\cro{-\beta \gT(\bsX,P)}d\pi(P)}}\\
&\le \frac{1}{\pi^{2}(\sB)} \cro{\int_{\sB^{2}}\exp\cro{-\gL(P,Q)}d\pi_{\sB}(P)d\pi_{\sB}(Q)}^{-1}.
\end{align*}
\end{lem}

\begin{proof}
For $P,Q\in\sM$, we set 
\[
H(\bsX,P,Q)=\beta \cc_{1}\gT(\bsX,P,Q)-\log\cro{\int_{\sM}\exp\cro{\cc_{2}\beta\gT(\bsX,P,Q')}d\pi(Q')}.
\]
Then, 
%Align
\begin{align}
&\E\cro{\exp\cro{-H(\bsX,P,Q)}}\nonumber\\
&=\E\cro{\exp\cro{-\beta\cc_{1}\gT(\bsX,P,Q)}\int_{\sM}\exp\cro{\cc_{2}\beta\gT(\bsX,P,Q')}d\pi(Q')}\nonumber\\
&=\E\cro{\int_{\sM}\exp\cro{\beta\pa{\cc_{2}\gT(\bsX,P,Q')-\cc_{1}\gT(\bsX,P,Q)}}d\pi(Q')}\nonumber\\
&=\exp\cro{\gL(P,Q)}.\label{Deno-eq01}
\end{align}

It follows from the convexity of the exponential and the fact that $\lambda=\cc_{1}\beta$ that for all $P\in\sM$,
%Align
\begin{align*}
\E\cro{\exp\cro{\beta \gT(\bsX,P)}}&=\E\cro{\exp\cro{\int_{\sM}[\beta\gT(\bsX,P,Q)]d\widetilde \pi_{\bsX}(Q)}}\\
&\le \E\cro{{\int_{\sM}{\exp\cro{\beta \gT(\bsX,P,Q)}d\widetilde \pi_{\bsX}(Q)}}}\\
&=\E\cro{\frac{\int_{\sM}\exp\cro{\cc_{2}\beta\gT(\bsX,P,Q)}d\pi(Q)}{\int_{\sM}\exp\cro{\cc_{1}\beta\gT(\bsX,P,Q)}d\pi(Q)}}\\
&=\E\cro{\frac{1}{\int_{\sM}\exp\cro{H(\bsX,P,Q)}d\pi(Q)}}.
\end{align*}
Applying Lemma~\ref{lem-Catoni} with $U=\bsX$ and $V=Q$ with distribution $\pi$ we obtain that
%Align
\begin{align*}
\E\cro{\exp\cro{\beta\gT(\bsX,P)}}\le \cro{\int_{\sM}\frac{1}{\E\cro{\exp\cro{-H(\bsX,P,Q)}}}d\pi(Q)}^{-1}.
\end{align*}
We deduce from~\eref{Deno-eq01} that for all $P\in\sM$
%Align
\begin{align}
\E\cro{\exp\cro{\beta\gT(\bsX,P)}}&\le \cro{\int_{\sM}\exp\cro{-\gL(P,Q)}d\pi(Q)}^{-1}\nonumber\\
&\le \frac{1}{\pi(\sB)}\cro{\int_{\sB}\exp\cro{-\gL(P,Q)}d\pi_{\sB}(Q)}^{-1}.\label{Deno-eq02}
\end{align}

Applying Lemma~\ref{lem-Catoni} with $U=\bsX$, $V=P$ with distribution $\pi$ and $h(U,V)=\beta \gT(\bsX,P)$, gives 
%Align
\begin{align*}
\E\cro{\frac{1}{\int_{\sM}\exp\cro{-\beta \gT(\bsX,P)}d\pi(P)}}&\le  \cro{\int_{\sM}\frac{1}{\E\cro{\exp\cro{\beta\gT(\bsX,P)}}}d\pi(P)}^{-1}\\
&\le \frac{1}{\pi(\sB)}\cro{\int_{\sB}\frac{1}{\E\cro{\exp\cro{\beta\gT(\bsX,P)}}}d\pi_{\sB}(P)}^{-1}
\end{align*}
which together with~\eref{Deno-eq02} leads to the result.

\end{proof}

The proofs of Theorems~\ref{main1} and~\ref{main2} rely on suitable bounds on the Laplace transforms of sums of independent random variables and on a summation lemma.  These results are presented below. 

\begin{lem}\label{lem-Hoeffding}
For all $\beta\in\R$ and random variable $U$ with values in an interval of length $l\in (0,+\infty)$, 
%Equation
\begin{equation}\label{eq-Ho1}
\log\E\cro{\exp\cro{\beta U}}\le \beta\E\cro{U}+\frac{\beta^{2}l^{2}}{8}.
\end{equation}
\end{lem}

\begin{lem}\label{lem-Pois}
Let $U$ be a squared integrable random variable not larger than $b> 0$. For all $\beta>0$, 
%Equation
\begin{equation}\label{eq-Pois01}
\log\E\cro{\exp\cro{\beta U}}\le\beta \E\cro{U}+\beta^{2}\E\cro{U^{2}}\frac{\phi(\beta b)}{2},
\end{equation}
where $\phi$ is defined by~\eref{def-phi}.
\end{lem}
The proofs of Lemmas~\ref{lem-Hoeffding} and~\ref{lem-Pois} can be found on pages 21 and 23 in Massart~\citeyearpar{MR2319879} (where our function $\phi$ is defined as twice his).
\begin{lem}\label{lem-som}
Let $J\in\N$, $\gamma>0$ and $\PAP\in \sM$. If $\ray$ satisfies $n\beta a_{1}\ray\ge 1$ and~\eref{prop-epsn}, for all $\gamma_{0}>2\gamma$ 
%Align
\begin{align}
&\int_{\co{\sB}(\PAP,2^{J}\ray)}\exp\cro{-\gamma_{0}n\beta a_{1}\ell(\PAP,P)}d\pi(P)\nonumber\\
&\le  \pi\pa{\sB}\exp\cro{\Xi-\pa{\gamma_{0}-2\gamma}n \beta a_{1}2^{J}\ray}\label{eq-lem-som01}
\end{align}
with 
%Align
\begin{align*}
\Xi&=  -\gamma +\log\cro{\frac{1}{1-\exp\cro{-\pa{\gamma_{0}-2\gamma}}}}
\end{align*}
Besides, 
%Align
\begin{align}
\int_{\sM}\exp\cro{-\gamma_{0}n\beta a_{1}\ell(\PAP,P)}d\pi(P)\le\pi(\sB)\exp\cro{\Xi'}\label{eq-lem-som02}
\end{align}
with 
\[
%\Xi'=\log\cro{1+\frac{\exp\cro{-\pa{\gamma_{0}-\gamma}n \beta a_{1}\ray}}{1-\exp\cro{-\pa{\gamma_{0}-2\gamma}n\beta  a_{1}\ray}}}.
\Xi'=\log\cro{1+\frac{\exp\cro{-\pa{\gamma_{0}-\gamma}}}{1-\exp\cro{-\pa{\gamma_{0}-2\gamma}  }}}.
\]
\end{lem}

\begin{proof}
From~\eref{prop-epsn}, we deduce by induction that for all $j\ge 0$
%Align
\begin{align*}
\pi\pa{\sB(\PAP,2^{j+1}\ray)}&\le \exp\cro{\gamma  n \beta a_{1}\ray\sum_{k=0}^{j}2^{k}}\pi\pa{\sB}\\
&= \exp\cro{(2^{j+1}-1)\gamma n \beta a_{1}\ray}\pi\pa{\sB}
\end{align*}
Consequently, 
%Align
\begin{align*}
&\int_{\co{\sB}(\PAP,2^{J}\ray)}\exp\cro{-\gamma_{0}n\beta a_{1}\ell(\PAP,P)}d\pi(P)\\
&=\sum_{j\ge J}\int_{\sB(\PAP,2^{j+1}\ray)\setminus \sB(\PAP,2^{j}\ray)}\exp\cro{-\gamma_{0}\beta n a_{1}\ell(\PAP,P)}d\pi(P)\\
&\le \pi\pa{\sB}\sum_{j\ge J}\frac{\pi\pa{\sB(\PAP,2^{j+1}\ray)}}{\pi\pa{\sB}}\exp\cro{-\gamma_{0}n\beta a_{1} 2^{j}\ray}\\
&\le  \pi\pa{\sB}\sum_{j\ge J}\exp\cro{\gamma n\beta a_{1}(2^{j+1}-1)\ray-\gamma_{0}n\beta a_{1} 2^{j}\ray}\\
&=  \pi\pa{\sB}\exp\cro{-\gamma n \beta a_{1}\ray} \sum_{j\ge J}\exp\cro{-\pa{\gamma_{0}-2\gamma}n\beta a_{1} 2^{j}\ray}\\
&=  \pi\pa{\sB}\exp\cro{-\gamma n \beta a_{1}\ray} \sum_{j\ge 0}\exp\cro{-\pa{\gamma_{0}-2\gamma}n\beta a_{1} 2^{j}2^{J}\ray}.
\end{align*}
Since $2^{j}\ge j+1$ for all $j\ge 0$ we obtain that 
%Align
\begin{align*}
&\int_{\co{\sB}(\PAP, 2^{J}r)}\exp\cro{-\gamma_{0}n\beta a_{1}\ell(\PAP,P)}d\pi(P)\\
&\le  \pi\pa{\sB}\exp\cro{-\gamma n \beta a_{1}\ray} \sum_{j\ge 0}\exp\cro{-\pa{\gamma_{0}-2\gamma}n\beta a_{1} (j+1)2^{J}\ray}\\
&\le  \pi\pa{\sB}\exp\cro{-\gamma n \beta a_{1}\ray-\pa{\gamma_{0}-2\gamma}n\beta a_{1}2^{J}\ray}\sum_{j\ge 0}\exp\cro{-j\pa{\gamma_{0}-2\gamma}n\beta  a_{1}2^{J}\ray}\\
&= \pi\pa{\sB}\frac{\exp\cro{-\gamma n \beta a_{1}\ray}}{1-\exp\cro{-\pa{\gamma_{0}-2\gamma}n\beta  a_{1}2^{J}\ray}}\exp\cro{-\pa{\gamma_{0}-2\gamma}n \beta a_{1}2^{J}\ray}.
\end{align*}
which leads to~\eref{eq-lem-som01} since $n\beta  a_{1}2^{J}\ray\ge n\beta a_{1}\ray\ge 1$. Finally, by applying this inequality with $J=0$ we obtain that  
%Align
\begin{align*}
&\int_{\sM}\exp\cro{-\beta n a_{1}\gamma_{0}\ell(\PAP,P)}d\pi(P)\\
&=\int_{\sB}\exp\cro{-\beta n a_{1}\gamma_{0}\ell(\PAP,P)}d\pi(P)+\int_{\co{\sB}}\exp\cro{-\beta n a_{1}\gamma_{0}\ell(\PAP,P)}d\pi(P)\\
&\le \pi(\sB)\cro{1+\frac{\exp\cro{-\gamma -\pa{\gamma_{0}-2\gamma}n \beta a_{1}\ray}}{1-\exp\cro{-\pa{\gamma_{0}-2\gamma}}}}\\
&\le \pi(\sB)\cro{1+\frac{\exp\cro{-\pa{\gamma_{0}-\gamma}}}{1-\exp\cro{-\pa{\gamma_{0}-2\gamma}}}},
\end{align*}
which is~\eref{eq-lem-som02}.
\end{proof}

\subsection{Proof of Theorem~\ref{main1}}
For all $i\in\{1,\ldots,n\}$ and $P,Q,Q'\in\sM$, let us set 
%Align
\begin{align}
U_{i}&=\cc\pa{t_{(P,Q')}(X_{i})-\E\cro{t_{(P,Q')}(X_{i})}}\label{def-U}\\
&\quad  -\cc_{1}\pa{t_{(P,Q)}(X_{i})-\E\cro{t_{(P,Q)}(X_{i})}}\nonumber\\
V_{i}&=\cc_{2}\pa{t_{(P,Q')}(X_{i})-\E\cro{t_{(P,Q')}(X_{i})}}\label{def-V}\\
&\quad  -\cc_{1}\pa{t_{(P,Q)}(X_{i})-\E\cro{t_{(P,Q)}(X_{i})}}.\nonumber
\end{align}
The random variables $U_{i}$ are independent and under Assumption~\ref{Hypo-1}-\ref{cond-4}, they takes their values in an interval of length $l_{1}=\cc+\cc_{1}=1+2\cc$. The $V_{i}$ are also independent and they takes their values in an interval of length $l_{2}=\cc_{1}+\cc_{2}=3+2\cc$. Applying Lemma~\ref{lem-Hoeffding}, we obtain that 
%Align
\begin{align}
\prod_{i=1}^{n}\E\cro{\exp\cro{\beta U_{i}}}&\le \exp\cro{\frac{l_{1}^{2} n\beta^{2}}{8}}\label{eq-Th1-01b}
\end{align}
and
%Align
\begin{align}
\prod_{i=1}^{n}\E\cro{\exp\cro{\beta V_{i}}}&\le \exp\cro{\frac{l_{2}^{2} n\beta^{2}}{8}}.\label{eq-Th1-01}
\end{align}

By using Assumption~\ref{Hypo-0bis} and the fact that $\cc_{0}=\cc_{1}-\cc a_{0}/a_{1}>0$, 
%Align
\begin{align}
&\cc \pa{a_{0}\ell(\overline P\et,P)- a_{1}\ell(\overline P\et,Q')}-\cc_{1}\pa{a_{1}\ell(\overline P\et,P)-a_{0}\ell(\overline P\et,Q)}\nonumber\\
&= -\pa{\cc_{1}a_{1}-\cc a_{0}}\ell(\overline P\et,P)-\cc a_{1}\ell(\overline P\et,Q')+\cc_{1}a_{0}\ell(\overline P\et,Q)\nonumber\\
&\le -\cc_{0}a_{1}\cro{\tau^{-1}\ell(\PAP,P)-\ell(\overline P\et,\PAP)}-\cc a_{1}\cro{\tau^{-1}\ell(\PAP,Q')-\ell(\overline P\et,\PAP)}\nonumber\\
&\quad +\tau \cc_{1}a_{0}\cro{\ell(\overline P\et,\PAP)+\ell(\PAP,Q)}\nonumber\\
&=e_{0}a_{1}\ell(\overline P\et,\PAP)-\tau^{-1}\cc_{0}a_{1}\ell(\PAP,P)-\tau^{-1}\cc a_{1}\ell(\PAP,Q')+\tau \cc_{1}a_{0}\ell(\PAP,Q)\label{eq-Th1-03b}
\end{align}
with 
%Equation
\begin{equation}\label{def-c0}
e_{0}=\cc_{0}+\cc+\frac{\tau \cc_{1} a_{0}}{a_{1}}.
\end{equation}

It follows from \eref{eq-Th1-03b} and Assumptions~\ref{Hypo-1}-\ref{cond-2}, more precisely its consequences~\eref{eq-SPQbis} and~\eref{eq-SPQter}, that
%Align
\begin{align}
&n^{-1}\ac{\cc\E\cro{\gT(\bsX,P,Q')}-\cc_{1}\E\cro{\gT(\bsX,P,Q)}}\nonumber\\
&\le \cc \cro{a_{0}\ell(\overline P\et,P)- a_{1}\ell(\overline P\et,Q')}-\cc_{1}\cro{a_{1}\ell(\overline P\et,P)-a_{0}\ell(\overline P\et,Q)}\nonumber\\
&\le e_{0}a_{1}\ell(\overline P\et,\PAP)-\tau^{-1}\cc_{0}a_{1}\ell(\PAP,P)-\tau^{-1}\cc a_{1}\ell(\PAP,Q')+\tau \cc_{1}a_{0}\ell(\PAP,Q).\label{eq-Th1-02b}
\end{align}

Since $a_{0}\ge a_{1}$ and $\cc_{2}>\cc_{1}$,  $\cc_{0}'=\cc_{2}(a_{0}/a_{1})-\cc_{1}>0$ and   by arguing as above, we obtain similarly that
%Align
\begin{align}
&n^{-1}\ac{\cc_{2}\E\cro{\gT(\bsX,P,Q')}-\cc_{1}\E\cro{\gT(\bsX,P,Q)}}\nonumber\\
&\le  \cc_{2} \pa{a_{0}\ell(\overline P\et,P)- a_{1}\ell(\overline P\et,Q')}-\cc_{1}\pa{a_{1}\ell(\overline P\et,P)-a_{0}\ell(\overline P\et,Q)}\nonumber\\
&=\cc_{0}'a_{1}\ell(\overline P\et,P)-\cc_{2}a_{1}\ell(\overline P\et,Q')+\cc_{1}a_{0}\ell(\overline P\et,Q)\nonumber\\
&\le\tau\cc_{0}'a_{1}\cro{\ell(\overline P\et,\PAP)+\ell(\PAP,P)}-\cc_{2}a_{1}\cro{\tau^{-1}\ell(\PAP,Q')-\ell(\overline P\et,\PAP)}\nonumber\\
&\quad +\tau\cc_{1}a_{0}\cro{\ell(\overline P\et,\PAP)+\ell(\PAP,Q)}\nonumber\\
&\le \pa{e_{1}+\cc_{2}}a_{1}\ell(\overline P\et,\PAP)+\tau\cc_{0}'a_{1}\ell(\PAP,P)\nonumber\\
&\quad -\tau^{-1}\cc_{2} a_{1}\ell(\PAP,Q')+\tau \cc_{1}a_{0}\ell(\PAP,Q),\label{eq-Th1-02}
\end{align}
with 
%Equation
\begin{equation}\label{def-c1}
e_{1}=\tau\cro{\cc_{0}'+\cc_{1} a_{0}/a_{1}}=\tau\cro{\cc_{2}(a_{0}/a_{1})+\ \cc_{1} \pa{a_{0}/a_{1}-1}}.
\end{equation}

Using~\eref{eq-Th1-01b} and~\eref{eq-Th1-02b}, we deduce that for all $P,Q,Q'\in\sM$
%Align
\begin{align}
&\E\cro{\exp\cro{\beta\pa{\cc\gT(\bsX,P,Q')-\cc_{1}\gT(\bsX,P,Q)}}}\nonumber \\
&=\prod_{i=1}^{n}\E\cro{\exp\cro{\beta\pa{\cc t_{(P,Q')}(X_{i})-\cc_{1}t_{(P,Q)}(X_{i})}}}\nonumber\\
&=\exp\cro{\beta\pa{\cc\E\cro{\gT(\bsX,P,Q')}-\cc_{1}\E\cro{\gT(\bsX,P,Q)}}}\prod_{i=1}^{n}\E\cro{\exp\cro{\beta U_{i}}}\nonumber\\
&\le \exp\cro{n\beta\cro{\Delta_{1}(P,Q)-\tau^{-1}\cc a_{1}\ell(\PAP,Q')}}\label{eq-th01-10b}
\end{align}
with
%Align
\begin{align}
&\Delta_{1}(P,Q)=e_{0}a_{1}\ell(\overline P\et,\PAP)+\tau\cc_{1}a_{0}\ell(\PAP,Q)+\frac{l_{1}^{2} \beta}{8}-\tau^{-1}\cc_{0}a_{1}\ell(\PAP,P).\label{def-Delta1}
\end{align}

Using~\eref{eq-Th1-01} and~\eref{eq-Th1-02}, we obtain similarly that for all $P,Q,Q'\in\sM$ 
%Align
\begin{align}
&\E\cro{\exp\cro{\beta\pa{\cc_{2}\gT(\bsX,P,Q')-\cc_{1}\gT(\bsX,P,Q)}}}\nonumber \\
&\le \exp\cro{n\beta\cro{\Delta_{2}(P,Q)-\tau^{-1}\cc_{2}a_{1}\ell(\PAP,Q')}}\label{eq-th01-10}
\end{align}
with 
%Align
\begin{align}
\Delta_{2}(P,Q)&=\pa{e_{1}+\cc_{2}}a_{1}\ell(\overline P\et,\PAP)+\tau\cc_{0}'a_{1}\ell(\PAP,P)+\tau\cc_{1}a_{0}\ell(\PAP,Q)\nonumber\\
&\quad +\frac{l_{2}^{2}\beta}{8}.\label{def-Delta2}
\end{align}

Since $2\gamma<\tau^{-1}\cc<\tau^{-1}\cc_{2}$, we may apply Lemma~\ref{lem-som} with $\gamma_{0}=\tau^{-1}\cc $ and $\gamma_{0}=\tau^{-1}\cc_{2}$ successively which leads to
%Align
\begin{align}
\int_{\sM}\exp\cro{-\tau^{-1}\cc n \beta a_{1}\ell(\PAP,Q')}d\pi(Q')\le  \pi\pa{\sB}\exp\cro{\Xi_{1}}\label{eq-th01-11b}
\end{align}
and
%Align
\begin{align}
\int_{\sM}\exp\cro{-\tau^{-1}\cc_{2}n \beta a_{1}\ell(\PAP,Q')}d\pi(Q')\le  \pi\pa{\sB}\exp\cro{\Xi_{1}}\label{eq-th01-11}
\end{align}
with 
%Align
\begin{align}
\Xi_{1}&=\log\cro{1+\frac{\exp\cro{-\pa{\tau^{-1}\cc-\gamma}}}{1-\exp\cro{-\pa{\tau^{-1}\cc-2\gamma}}}}\label{def-xi1}\\
&\ge \log\cro{1+\frac{\exp\cro{-\pa{\tau^{-1}\cc_{2}-\gamma}}}{1-\exp\cro{-\pa{\tau^{-1}\cc_{2}-2\gamma}}}}.\nonumber
\end{align}

Putting~\eref{eq-th01-10} and~\eref{eq-th01-11} together leads to
%Align
\begin{align*}
\exp\cro{\gL(P,Q)}&=\int_{\sM}\E\cro{\exp\cro{\beta\pa{\cc_{2}\gT(\bsX,P,Q')-\cc_{1}\gT(\bsX,P,Q)}}}d\pi(Q')\\
&\le  \exp\cro{n\beta\Delta_{2}(P,Q)}\int_{\sM}\exp\cro{-\tau^{-1}\cc_{2}n \beta a_{1}\ell(\PAP,Q')}d\pi(Q')\\
&\le\pi\pa{\sB}\exp\cro{\Xi_{1}+n\beta \Delta_{2}(P,Q)},
\end{align*}
and since, for all $(P,Q)\in\sB^{2}$, by definition~\eref{def-Delta2} of $\Delta_{2}(P,Q)$,  
%Align
\begin{align}
\Delta_{2}(P,Q)&\le \pa{e_{1}+\cc_{2}}a_{1}\ell(\overline P\et,\PAP)+\cro{\tau\cc_{0}'a_{1}+\tau\cc_{1}a_{0}}\ray+\frac{l_{2}^{2}\beta}{8}\nonumber\\
&=\pa{e_{1}+\cc_{2}}a_{1}\ell(\overline P\et,\PAP)+e_{1}a_{1}\ray+\frac{l_{2}^{2}\beta}{8}=\Delta_{2}\label{def-Delta2b}
\end{align}
 we derive that
%Align
\begin{align*}
&\cro{\int_{\sB^{2}}\exp\cro{-\gL(P,Q)}d\pi_{\sB}(P)d\pi_{\sB}(Q)}^{-1}\le \pi\pa{\sB}\exp\cro{\Xi_{1}+n\beta \Delta_{2}}.
\end{align*}

We deduce from~\eref{eq-borne01} that 
%Align
\begin{align*}
\P(\co{A})&\le \frac{z}{\pi\pa{\sB}}\exp\cro{\Xi_{1}+n\beta \Delta_{2}}.
\end{align*}
In particular, $\P(\co{A})\le e^{-\xi}$ for $z$ satisfying 
%Align
\begin{align}
\log\pa{\frac{1}{z}}=\xi+\log\frac{1}{\pi(\sB)}+\Xi_{1}+n\beta \Delta_{2}.\label{def-z}
\end{align}
Putting~\eref{eq-th01-10b} and~\eref{eq-th01-11b} together, we obtain that  
%Align
\begin{align*}
&\exp\cro{\gM(P,Q)}\\
&=\int_{\sM}\E\cro{\exp\cro{\beta\pa{\cc\gT(\bsX,P,Q')-\cc_{1}\gT(\bsX,P,Q)}}}d\pi(Q')\\
&\le  \exp\cro{n\beta\Delta_{1}(P,Q)}\int_{\sM}\exp\cro{-\tau^{-1}\cc n \beta a_{1}\ell(\PAP,Q')}d\pi(Q')\\
&\le\pi\pa{\sB}\exp\cro{\Xi_{1}+n\beta \Delta_{1}(P,Q)}.
\end{align*}
It follows from the definition~\eref{def-Delta1} of $\Delta_{1}(P,Q)$ that for all $P\in\sM$ and for all $Q\in\sB$, 
%Align
\begin{align*}
\Delta_{1}(P,Q)&\le  e_{0}a_{1}\ell(\overline P\et,\PAP)+\tau\cc_{1}a_{0}\ray+\frac{l_{1}^{2} \beta}{8}-\tau^{-1}\cc_{0}a_{1}\ell(\PAP,P),
\end{align*}
and consequently, for all $P\in\sM$ and $Q\in\sB$
\begin{align*}
&\exp\cro{\gM(P,Q)}\\
&\le\pi\pa{\sB}\exp\cro{\Xi_{1}+n\beta \pa{e_{0}a_{1}\ell(\overline P\et,\PAP)+\tau\cc_{1}a_{0}\ray+\frac{l_{1}^{2} \beta}{8}-\tau^{-1}\cc_{0}a_{1}\ell(\PAP,P)}}.
\end{align*}
We derive from Lemma~\ref{LaplaceZ} that 
%Align
\begin{align*}
&\E\cro{\exp\cro{-\beta\gT(\bsX,P)}}\nonumber\\
&\le \frac{1}{\pi(\sB)}\cro{\int_{\sB}\exp\cro{-\gM(P,Q)}d\pi_{\sB}(Q)}^{-1}\\
&\le \exp\cro{\Xi_{1}+n\beta \pa{e_{0}a_{1}\ell(\overline P\et,\PAP)+\tau\cc_{1}a_{0}\ray+\frac{l_{1}^{2} \beta}{8}-\tau^{-1}\cc_{0}a_{1}\ell(\PAP,P)}},
\end{align*}
hence,
%Align
\begin{align}
\int_{\co{\sB}(\PAP,2^{J}\ray)}&\E\cro{\exp\cro{-\beta \gT(\bsX,P)}}d\pi(P)\label{eq-th01-12b}\\
&\le \exp\cro{\Xi_{1}+n\beta \pa{ e_{0}a_{1}\ell(\overline P\et,\PAP)+\tau\cc_{1}a_{0}\ray+\frac{l_{1}^{2} \beta}{8}}}\nonumber\\
&\quad \times \int_{\co{\sB}(\PAP,2^{J}\ray)}\exp\cro{-\tau^{-1}\cc_{0} n\beta a_{1} \ell(\PAP,P)}d\pi(P).\nonumber
\end{align}
Applying Lemma~\ref{lem-som} with $\gamma_{0}=\tau^{-1}\cc_{0}>2\gamma$ and setting $e_{2}=\tau^{-1}\cc_{0}-2\gamma$, we get 
\begin{align*}
\int_{\co{\sB}(\PAP,2^{J}\ray)}\exp\cro{-\tau^{-1}\cc_{0} n\beta a_{1} \ell(\PAP,P)}d\pi(P)\le \pi\pa{\sB}\exp\cro{\Xi_{2}-e_{2}n\beta a_{1}2^{J}\ray}
\end{align*}
with 
%Equation
\begin{equation}\label{def-xi2}
\Xi_{2}=-\gamma+\log\cro{\frac{1}{1-\exp\cro{-e_{2}}}},
\end{equation}
which together with~\eref{eq-th01-12b} leads to 
%Align
\begin{align}
&\log\int_{\co{\sB}(\PAP,2^{J}\ray)}\E\cro{\exp\cro{-\beta \gT(\bsX,P)}}d\pi(P)\nonumber\\
&\le \log\cro{ \pi\pa{\sB}}+\Xi_{1}+\Xi_{2}\nonumber \\
&\quad +n\beta \cro{e_{0}a_{1}\ell(\overline P\et,\PAP)+\tau\cc_{1}a_{0}\ray+\frac{l_{1}^{2} \beta}{8}-e_{2}a_{1}2^{J}\ray}.\label{eq-th01-13b}
\end{align}
Using the definitions~\eref{def-z} of $z$ and~\eref{def-Delta2b} of $\Delta_{2}$ we deduce from~\eref{eq-th01-13b} that
%Align
\begin{align}
&\log\cro{\frac{1}{z}\int_{\co{\sB}(\PAP,2^{J}\ray)}\E\cro{\exp\cro{-\beta \gT(\bsX,P)}}d\pi(P)}\nonumber\\
&\le \log\pa{\frac{1}{z}}+\log\cro{ \pi\pa{\sB}}+\Xi_{1}+\Xi_{2}\nonumber\\
&\quad +n\beta \cro{e_{0}a_{1}\ell(\overline P\et,\PAP)+\tau\cc_{1}a_{0}\ray+\frac{l_{1}^{2} \beta}{8}-e_{2}a_{1}2^{J}\ray}\nonumber\\
&=\xi+\log\frac{1}{\pi(\sB)}+\Xi_{1}+n\beta \Delta_{2}+\log\cro{ \pi\pa{\sB}}+\Xi_{1}+\Xi_{2}\nonumber\\
&\quad+n\beta \cro{e_{0}a_{1}\ell(\overline P\et,\PAP)+\tau\cc_{1}a_{0}\ray+\frac{l_{1}^{2} \beta}{8}-e_{2}a_{1}2^{J}\ray}\nonumber\\
&=n\beta\cro{\pa{e_{1}+\cc_{2}+e_{0}}a_{1}\ell(\overline P\et,\PAP)+e_{1}a_{1}\ray+\frac{l_{2}^{2}\beta}{8}++\tau\cc_{1}a_{0}\ray+\frac{l_{1}^{2} \beta}{8}}\nonumber\\
&\quad +\xi+2\Xi_{1}+\Xi_{2}-e_{2}n \beta a_{1}2^{J}\ray\nonumber\\
&=n\beta\cro{\pa{e_{0}+e_{1}+\cc_{2}}a_{1}\ell(\overline P\et,\PAP)+\cro{e_{1}+\frac{\tau\cc_{1}a_{0}}{a_{1}}}a_{1}\ray+\frac{(l_{1}^{2}+l_{2}^{2}) \beta}{8}}\nonumber\\
&\quad+\xi+2\Xi_{1}+\Xi_{2}-e_{2}n \beta a_{1}2^{J}\ray.\label{eq-thm01-fin}
\end{align}
Setting, 
\[
C_{1}=e_{0}+e_{1}+\cc_{2}\quad \text{and}\quad C_{2}=e_{1}+\frac{\tau\cc_{1}a_{0}}{a_{1}},
\]
we see that the right-hand side of~\eref{eq-thm01-fin} is not larger than $-\xi$, provided that 
%Align
\begin{align*}
e_{2}n\beta a_{1}2^{J}\ray&\ge 2\xi+2\Xi_{1}+\Xi_{2}+n\beta\cro{C_{1}a_{1}\ell(\overline P\et,\PAP)+C_{2}a_{1}\ray+\frac{(l_{1}^{2}+l_{2}^{2}) \beta}{8}}
\end{align*}
or equivalently if 
%Align
\begin{align}
2^{J}&\ge \frac{1}{e_{2}}\cro{\frac{2\xi+2\Xi_{1}+\Xi_{2}}{\beta na_{1}\ray}+\frac{C_{1}\ell(\overline P\et,\PAP)}{\ray}+C_{2}+\frac{\cro{l_{1}^{2}+l_{2}^{2}}\beta}{8a_{1}\ray}}.\label{cond-J}
\end{align}
Choosing $\PAP$ in $\sM(\beta)$ and using the inequalities $a_{1}^{-1}\beta\ge \ray_{n}(\beta,\PAP)\ge 1/(\beta na_{1})$, for 
\[
\ray=\ell(\overline P\et,\PAP)+\frac{1}{a_{1}}\pa{\beta+\frac{2\xi}{n\beta}}\ge \frac{1}{\beta n a_{1}}
\]
we obtain that the right-hand side of~\eref{cond-J} satisfies 
%Align
\begin{align*}
&\frac{1}{e_{2}}\cro{\frac{2\xi+2\Xi_{1}+\Xi_{2}}{\beta na_{1}\ray}+\frac{C_{1}\ell(\overline P\et,\PAP)+C_{2}\ray}{\ray}+\frac{\cro{l_{1}^{2}+l_{2}^{2}}\beta}{8a_{1}\ray}}\\
&\le \frac{1}{e_{2}}\cro{C_{2}+2\Xi_{1}+\Xi_{2}+\frac{C_{3}}{\ray}\pa{\ell(\overline P\et,\PAP)+\frac{1}{a_{1}}\pa{\beta+\frac{2\xi}{n\beta}}}}\\
&= \frac{1}{e_{2}}\cro{C_{2}+2\Xi_{1}+\Xi_{2}+C_{3}}
\end{align*}
with $C_{3}=\max\{1,C_{1},\cro{l_{1}^{2}+l_{2}^{2}}/8\}$. Inequality~\eref{cond-J} is therefore satisfied 
for $J\in\N$ such that 
\[
2^{J}\ge \frac{C_{2}+2\Xi_{1}+\Xi_{2}+C_{3}}{e_{2}}\vee 1> 2^{J-1},
\]
and we may take 
%Align
\begin{align}\label{def-kappa0}
\kappa_{0}=\tau\cro{ \frac{2\pa{C_{2}+2\Xi_{1}+\Xi_{2}+C_{3}}}{e_{2}}\vee 1+1}\ge \tau\pa{2^{J}+1}.
\end{align}
We recall below, the list of constants depending on $a_{0},a_{1},c,\tau$ and $\gamma$ and  we have used along the proof. 
%Align
\begin{align*}
c_{0}&=1+c-\frac{ca_{0}}{a_{1}}, &c_{1}&=1+c, &c_{2}&=2+c,\\
c'_{0}&=\frac{c_{2}a_{0}}{a_{1}}-c_{1},&l_{1}&=1+2c, & l_{2}&=3+2c,\\
e_{0}&=c_{0}+c+\frac{\tau c_{1}a_{0}}{a_{1}},& 
e_{1}&=\tau\cro{c_{0}'+c_{1}\frac{a_{0}}{a_{1}}}, &e_{2}&=\tau^{-1}c_{0}-2\gamma,\\
C_{1}&=e_{0}+e_{1}+c_{2},& C_{2}&=e_{1}+\frac{\tau c_{1}a_{0}}{a_{1}},& 
C_{3}&=\max\ac{1,C_{1},\frac{l_{1}^{2}+l_{2}^{2}}{8}},
\end{align*}
and 
%Align
\begin{align*}
\Xi_{1}&=\log\cro{1+\frac{\exp\cro{-\pa{\tau^{-1}\cc-\gamma}}}{1-\exp\cro{-\pa{\tau^{-1}\cc-2\gamma}}}},\quad \Xi_{2}=-\gamma+\log\cro{\frac{1}{1-\exp\cro{-e_{2}}}}.
\end{align*}
%

%For $\PAP\in\sM(\beta)$ and the choice $\ray=\ell(\overline P\et,\PAP)+a_{1}^{-1}\beta\ge \ray_{n}(\beta,\PAP)\ge 1/(na_{1}\beta)$, \eref{cond-J} is satisfied if
%\[
%2^{J}\ge\frac{1}{e_{2}}\pa{2\xi+2\Xi_{1}+\Xi_{2}+C_{1}+C_{2}+\frac{l_{1}^{2}+l_{2}^{2}}{8}}
%\]
%and the requirement $\kappa_{0}\ge\tau(2^{J}+1)$ if
%%Equation
%\begin{equation}\label{def-Kn0}
%\kappa_{0}=\tau\left[1+\frac{2}{e_{2}}\pa{2\xi+2\Xi_{1}+\Xi_{2}+C_{1}+C_{2}+\frac{l_{1}^{2}+l_{2}^{2}}{8}}\right].
%\end{equation}
%%
%In particular, $\kappa_{0}$ writes as $K_{1}+K_{2}\xi$ with 
%%Align
%\begin{align}\label{def-Kn}
%K_{1}=\tau\cro{1+\frac{2}{e_{2}}\pa{2\Xi_{1}+\Xi_{2}+C_{1}+C_{2}+\frac{l_{1}^{2}+l_{2}^{2}}{8}}}\quad \text{and}\quad K_{2}=\frac{4\tau}{e_{2}}.
%\end{align}
%% 

\subsection{Proof of Theorem~\ref{main2}}

The proof follows the same lines as that of Theorem~\ref{main1}. Under Assumption~\ref{Hypo-1}-\ref{cond-4}, the random variables $U_{i}$ and $V_{i}$ defined by~\eref{def-U} and~\eref{def-V} are not larger than with $b=\cc+\cc_{1}=l_{1}$ and $b=\cc_{2}+\cc_{1}=l_{2}$ respectively. Since under Assumption~\ref{Hypo-2}, more precisely its consequence~\eref{cond-3bis}, that
%Align
\begin{align*}
\frac{1}{n}\sum_{i=1}^{n}\E\cro{U_{i}^{2}}&\le 2\cro{\frac{\cc^{2}}{n}\sum_{i=1}^{n}\Var\cro{t_{(P,Q')}(X_{i})}+\frac{\cc_{1}^{2}}{n}\sum_{i=1}^{n}\Var\cro{t_{(P,Q)}(X_{i})}}\\
&\le 2a_{2}\cro{(\cc^{2}+\cc_{1}^{2})\ell(\overline P\et,P)+\cc^{2}\ell(\overline P\et,Q')+\cc_{1}^{2}\ell(\overline P\et,Q)}
\end{align*}
and  
%Align
\begin{align*}
\frac{1}{n}\sum_{i=1}^{n}\E\cro{V_{i}^{2}}&\le 2a_{2}\cro{(\cc_{2}^{2}+\cc_{1}^{2})\ell(\overline P\et,P)+\cc_{2}^{2}\ell(\overline P\et,Q')+\cc_{1}^{2}\ell(\overline P\et,Q)}
\end{align*}
we may apply Lemma~\ref{lem-Pois} and using the notation $\Lambda_{1}=\tau\phi(\beta l_{1})$, $\Lambda_{2}=\tau\phi(\beta l_{2})$ and Assumption~\ref{Hypo-0}, we get 
%Align
\begin{align}
&\frac{1}{n\beta}\log\cro{\prod_{i=1}^{n}\E\cro{\exp\cro{\beta U_{i}}}}\nonumber\\
&\le {\phi(\beta l_{1})}\beta a_{2}\cro{(\cc^{2}+\cc_{1}^{2})\ell(\overline P\et,P)+\cc^{2}\ell(\overline P\et,Q')+\cc_{1}^{2}\ell(\overline P\et,Q)}\nonumber\\
&\le 2\Lambda_{1} \beta a_{2}\cro{\cc^{2}+\cc_{1}^{2}}\ell(\overline P\et,\PAP)\nonumber \\
&\quad +\Lambda_{1} \beta a_{2}\cro{(\cc^{2}+\cc_{1}^{2})\ell(\PAP,P)+\cc^{2}\ell(\PAP,Q')+\cc_{1}^{2}\ell(\PAP,Q)}\label{eq-Th2-01b}
\end{align}
and similarly
%Align
\begin{align}
&\frac{1}{n\beta}\log\cro{\prod_{i=1}^{n}\E\cro{\exp\cro{\beta V_{i}}}}\nonumber\\
&\le 2\Lambda_{2} \beta a_{2}\cro{\cc_{2}^{2}+\cc_{1}^{2}}\ell(\overline P\et,\PAP)\nonumber \\
&\quad +\Lambda_{2} \beta a_{2}\cro{(\cc_{2}^{2}+\cc_{1}^{2})\ell(\PAP,P)+\cc_{2}^{2}\ell(\PAP,Q')+\cc_{1}^{2}\ell(\PAP,Q)}.\label{eq-Th2-01}
\end{align}

It follows from~\eref{eq-Th1-02b} that 
%Align
\begin{align*}
E_{1}&=n^{-1}\ac{\cc\E\cro{\gT(\bsX,P,Q')}-\cc_{1}\E\cro{\gT(\bsX,P,Q)}}\\
&\quad +2\Lambda_{1} \beta a_{2}\cro{\cc^{2}+\cc_{1}^{2}}\ell(\overline P\et,\PAP)\nonumber \\
&\quad +\Lambda_{1} \beta a_{2}\cro{(\cc^{2}+\cc_{1}^{2})\ell(\PAP,P)+\cc^{2}\ell(\PAP,Q')+\cc_{1}^{2}\ell(\PAP,Q)}\\
&\le\cro{e_{0}a_{1}+2\Lambda_{1} \beta a_{2}\pa{\cc^{2}+\cc_{1}^{2}}}\ell(\overline P\et,\PAP)\\
&\quad -\cro{\tau^{-1}\cc_{0}a_{1}-\Lambda_{1} \beta a_{2}(\cc^{2}+\cc_{1}^{2})}\ell(\PAP,P)\\
&\quad - \cro{\tau^{-1}\cc a_{1}-\Lambda_{1} \beta a_{2}\cc^{2}}\ell(\PAP,Q')\\
&\quad +  \cro{\tau \cc_{1}a_{0}+\Lambda_{1} \beta a_{2}\cc_{1}^{2}}\ell(\PAP,Q).
\end{align*}
Using the definitions~\eref{def-delta1} of $\overline \cc_{1}$ and \eref{def-delta2} of $\overline \cc_{2}$, , that is, 
%Align
\begin{align*}
\overline \cc_{1}=\cc_{0}-\tau \Lambda_{1}\beta a_{2}a_{1}^{-1}(\cc^{2}+\cc_{1}^{2})\quad \text{and}\quad \overline \cc_{2}=\cc-\tau\Lambda_{1} \beta a_{2}a_{1}^{-1}\cc^{2}
\end{align*}
and setting 
%Align
\begin{align*}
e_{3}&=e_{0}+2\Lambda_{1} \beta \frac{a_{2}\pa{\cc^{2}+\cc_{1}^{2}}}{a_{1}}\\
e_{4}&=\frac{1}{a_{1}}\cro{\tau \cc_{1}a_{0}+\Lambda_{1} \beta a_{2}\cc_{1}^{2}}
\end{align*}
and arguing as in the proof of inequality~\eref{eq-th01-10b}, we deduce from~\eref{eq-Th2-01b} that 
\begin{align}
&\log\E\cro{\exp\cro{\beta\pa{\cc\gT(\bsX,P,Q')-\cc_{1}\gT(\bsX,P,Q)}}}\nonumber \\
&\le n\beta E_{1}\nonumber \\
&\le n\beta a_{1}\cro{e_{3}\ell(\overline P\et,\PAP)-\tau^{-1}\cro{\overline \cc_{1}\ell(\PAP,P)+\overline \cc_{2}\ell(\PAP,Q')}+e_{4}\ell(\PAP,Q)}.
\label{eq-th02-10b}
\end{align}

It follows from~\eref{eq-Th1-02} that 
%Align
\begin{align*}
E_{2}&=n^{-1}\ac{\cc_{2}\E\cro{\gT(\bsX,P,Q')}-\cc_{1}\E\cro{\gT(\bsX,P,Q)}}\\
&\quad +2\Lambda_{2} \beta a_{2}\cro{\cc_{2}^{2}+\cc_{1}^{2}}\ell(\overline P\et,\PAP)\nonumber \\
&\quad +\Lambda_{2} \beta a_{2}\cro{(\cc_{2}^{2}+\cc_{1}^{2})\ell(\PAP,P)+\cc_{2}^{2}\ell(\PAP,Q')+\cc_{1}^{2}\ell(\PAP,Q)}\\
&\le  \pa{e_{1}+\cc_{2}}a_{1}\ell(\overline P\et,\PAP)+\tau\cc_{0}'a_{1}\ell(\PAP,P)-\tau^{-1}\cc_{2} a_{1}\ell(\PAP,Q')\\
&\quad +\tau \cc_{1}a_{0}\ell(\PAP,Q)+2\Lambda_{2} \beta a_{2}\cro{\cc_{2}^{2}+\cc_{1}^{2}}\ell(\overline P\et,\PAP)\nonumber \\
&\quad +\Lambda_{2} \beta a_{2}\cro{(\cc_{2}^{2}+\cc_{1}^{2})\ell(\PAP,P)+\cc_{2}^{2}\ell(\PAP,Q')+\cc_{1}^{2}\ell(\PAP,Q)}\\
&=\cro{\pa{e_{1}+\cc_{2}}a_{1}+2\Lambda_{2} \beta a_{2}\pa{\cc_{2}^{2}+\cc_{1}^{2}}}\ell(\overline P\et,\PAP)\nonumber \\
&\quad + \cro{\tau\cc_{0}'a_{1}+\Lambda_{2} \beta a_{2}(\cc_{2}^{2}+\cc_{1}^{2})}\ell(\PAP,P)\\
&\quad -\cro{\tau^{-1}\cc_{2} a_{1}-\Lambda_{2} \beta a_{2}\cc_{2}^{2}}\ell(\PAP,Q')\\
&\quad + \cro{\tau \cc_{1}a_{0}+\Lambda_{2} \beta a_{2}\cc_{1}^{2}}\ell(\PAP,Q).
\end{align*}

Using the definition~\eref{def-delta3} of $\overline\cc_{3}$, , that is,  
\[
\overline \cc_{3}=\cc_{2}-\tau\Lambda_{2} \beta a_{2}a_{1}^{-1}\cc_{2}^{2},
\]
and setting 
%Align
\begin{align*}
e_{5}&=e_{1}+\cc_{2}+2\Lambda_{2} \beta \frac{a_{2}\pa{\cc_{2}^{2}+\cc_{1}^{2}}}{a_{1}}, &e_{6}&=\tau\cc_{0}'+\Lambda_{2} \beta \frac{a_{2}(\cc_{2}^{2}+\cc_{1}^{2})}{a_{1}}\\
e_{7}&=\frac{1}{a_{1}}\cro{\tau \cc_{1}a_{0}+\Lambda_{2} \beta a_{2}\cc_{1}^{2}},
\end{align*}
and arguing as in the proof of~\eref{eq-th01-10}, we deduce from~\eref{eq-Th2-01} that 
\begin{align}
&\log\E\cro{\exp\cro{\beta\pa{\cc_{2}\gT(\bsX,P,Q')-\cc_{1}\gT(\bsX,P,Q)}}}\nonumber \\
&\le n\beta E_{2}\nonumber \\
&=n\beta a_{1}\pa{e_{5}\ell(\overline P\et,\PAP)+e_{6}\ell(\PAP,P)-\tau^{-1}\overline \cc_{3}\ell(\PAP,Q')+e_{7}\ell(\PAP,Q)}.\label{eq-th02-10}
\end{align}

Under our assumption on $\beta$, we know that the quantities $\overline \cc_{2}$ and $\overline\cc_{3}$ are positive and that $2\gamma<\tau^{-1}\pa{\overline \cc_{2}\wedge \overline \cc_{3}}$. We may therefore apply Lemma~\ref{lem-som} with $\gamma_{0}=\tau^{-1} \overline\cc_{2}$ and $\gamma_{0}=\tau^{-1} \overline\cc_{3}$ successively and get
%Align
\begin{align}\label{eq-th02-11b}
\int_{\sM}\exp\cro{-\tau^{-1} \overline\cc_{2}n\beta a_{1}\ell(\PAP,Q')}d\pi(Q')&\le  \pi\pa{\sB}\exp\cro{\overline \Xi_{1}}
\end{align}
and
%Align
\begin{align}\label{eq-th02-11}
\int_{\sM}\exp\cro{-\tau^{-1} \overline\cc_{3}n\beta a_{1}\ell(\PAP,Q')}d\pi(Q')&\le  \pi\pa{\sB}\exp\cro{\overline \Xi_{1}}
\end{align}
with 
%Equation
\begin{equation}\label{def-xibar1}
\overline \Xi_{1}=\log\cro{1+\frac{\exp\cro{-\pa{\tau^{-1}( \overline\cc_{2}\wedge  \overline\cc_{3})-\gamma}}}{1-\exp\cro{-\pa{\tau^{-1}( \overline\cc_{2}\wedge  \overline\cc_{3})-2\gamma}}}}.
\end{equation}

Putting~\eref{eq-th02-10} and~\eref{eq-th02-11} together, we obtain that for all $(P,Q)\in\sB^{2}$
%Align
\begin{align*}
&\exp\cro{\gL(P,Q)}\\
&=\int_{\sM}\E\cro{\exp\cro{\beta\pa{\cc_{2}\gT(\bsX,P,Q')-\cc_{1}\gT(\bsX,P,Q)}}}d\pi(Q')\\
&\le  \exp\cro{n\beta a_{1}\pa{e_{5}\ell(\overline P\et,\PAP)+e_{6}\ell(\PAP,P)+e_{7}\ell(\PAP,Q)}}\\
&\quad \times \int_{\sM}\exp\cro{-\tau^{-1}\overline \cc_{3}n \beta a_{1}\ell(\PAP,Q')}d\pi(Q')\\
&\le\pi\pa{\sB}\exp\cro{\overline \Xi_{1}+n\beta a_{1}\pa{e_{5}\ell(\overline P\et,\PAP)+e_{6}\ell(\PAP,P)+e_{7}\ell(\PAP,Q)}}\\
&\le \pi\pa{\sB}\exp\cro{\overline \Xi_{1}+n\beta a_{1}\pa{e_{5}\ell(\overline P\et,\PAP)+(e_{6}+e_{7})\ray}}.
\end{align*}
Consequently, 
%Align
\begin{align*}
&\cro{\int_{\sB^{2}}\exp\cro{-\gL(P,Q)}d\pi_{\sB}(P)d\pi_{\sB}(Q)}^{-1}\\
&\le \pi\pa{\sB}\exp\cro{\overline \Xi_{1}+n\beta a_{1}\pa{e_{5}\ell(\overline P\et,\PAP)+(e_{6}+e_{7})\ray}}.
\end{align*}
We deduce from~\eref{eq-borne01} that 
%Align
\begin{align*}
\P(\co{A})&\le \frac{z}{\pi\pa{\sB}}\exp\cro{\overline \Xi_{1}+n\beta a_{1}\pa{e_{5}\ell(\overline P\et,\PAP)+(e_{6}+e_{7})\ray}}.
\end{align*}
In particular, $\P(\co{A})\le e^{-\xi}$ for $z$ satisfying 
%Align
\begin{align}
\log\pa{\frac{1}{z}}=\xi+\log\frac{1}{\pi(\sB)}+\overline \Xi_{1}+n\beta a_{1}\cro{e_{5}\ell(\overline P\et,\PAP)+(e_{6}+e_{7})\ray}.\label{def-zb}
\end{align}
Putting~\eref{eq-th02-10b} and~\eref{eq-th02-11b} together, we obtain that  for all $Q\in\sB$
\begin{align*}
&\exp\cro{\gM(P,Q)}\\
&=\int_{\sM}\E\cro{\exp\cro{\beta\pa{\cc\gT(\bsX,P,Q')-\cc_{1}\gT(\bsX,P,Q)}}}d\pi(Q')\\
&\le  \exp\cro{n\beta a_{1}\pa{e_{3}\ell(\overline P\et,\PAP)-\tau^{-1} \overline\cc_{1}\ell(\PAP,P)+e_{4}\ell(\PAP,Q)}} \\
&\quad \times \int_{\sM}\exp\cro{-\tau^{-1} \overline\cc_{2} n \beta a_{1}\ell(\PAP,Q')}d\pi(Q')\\
&\le \pi(\sB)\exp\cro{\overline \Xi_{1}+n\beta a_{1}\pa{e_{3}\ell(\overline P\et,\PAP)+e_{4}\ray-\tau^{-1} \overline\cc_{1}\ell(\PAP,P)}}.
\end{align*}
We derive from Lemma~\ref{LaplaceZ} that 
%Align
\begin{align*}
&\E\cro{\exp\cro{-\beta\gT(\bsX,P)}}\nonumber\\
&\le \frac{1}{\pi(\sB)}\cro{\int_{\sB}\exp\cro{-\gM(P,Q)}d\pi_{\sB}(Q)}^{-1}\\
&\le  \exp\cro{\overline \Xi_{1}+n\beta a_{1}\pa{e_{3}\ell(\overline P\et,\PAP)+e_{4}\ray-\tau^{-1} \overline\cc_{1}\ell(\PAP,P)}},
\end{align*}
and consequently,  
%Align
\begin{align}
&\int_{\co{\sB}(\PAP,2^{J}\ray)}\E\cro{\exp\cro{-\beta \gT(\bsX,P)}}d\pi(P)\nonumber\\
&\le \exp\cro{\overline \Xi_{1}+n\beta a_{1}\pa{e_{3}\ell(\overline P\et,\PAP)+e_{4}\ray}}\nonumber \\
&\quad \times \int_{\co{\sB}(\PAP,2^{J}\ray)}\exp\cro{-\tau^{-1} \overline \cc_{1}n\beta a_{1} \ell(\PAP,P)}d\pi(P).\label{eq-th02-12b}
\end{align}
Since under our assumptions, $ \overline\cc_{1}>0$ and $2\gamma<\tau^{-1} \overline \cc_{1}$ we may apply Lemma~\ref{lem-som} with $\gamma_{0}=\tau^{-1} \overline \cc_{1}$, and setting $e_{8}=\tau^{-1} \overline\cc_{1}-2\gamma$ which leads to
\begin{align*}
& \int_{\co{\sB}(\PAP,2^{J}\ray)}\exp\cro{-\tau^{-1} \overline\cc_{1} n\beta a_{1} \ell(\PAP,P)}d\pi(P)\le \pi\pa{\sB}\exp\cro{\overline \Xi_{2}-e_{8}n\beta a_{1}2^{J}\ray}.
\end{align*}
with 
%Equation
\begin{equation}\label{def-xibar2}
\overline \Xi_{2}=-\gamma +\log\cro{\frac{1}{1-\exp\cro{-e_{8}}}},
\end{equation}
which together with~\eref{eq-th02-12b} leads to 
%Align
\begin{align}
&\int_{\co{\sB}(\PAP,2^{J}\ray)}\E\cro{\exp\cro{-\beta \gT(\bsX,P)}}d\pi(P)\nonumber\\
&\le \pi\pa{\sB}\exp\cro{\overline \Xi_{1}+\overline \Xi_{2}+n\beta a_{1}\pa{e_{3}\ell(\overline P\et,\PAP)+e_{4}\ray-e_{8}2^{J}\ray}}.\label{eq-th02-13b}
\end{align}
Using the definition~\eref{def-zb} of $z$, we deduce that 
%Align
\begin{align*}
&\log\cro{\frac{1}{z}\int_{\co{\sB}(\PAP,2^{J}\ray)}\E\cro{\exp\cro{-\beta \gT(\bsX,P)}}d\pi(P)}\\
&\le \log\pa{\frac{1}{z}}+\log \pi(\sB)+\overline \Xi_{1}+\overline \Xi_{2}+n\beta a_{1}\pa{e_{3}\ell(\overline P\et,\PAP)+e_{4}\ray-e_{8}2^{J}\ray}\\
&=\xi+\log\frac{1}{\pi(\sB)}+\overline \Xi_{1}+n\beta a_{1}\cro{e_{5}\ell(\overline P\et,\PAP)+(e_{6}+e_{7})\ray}\\
&\quad +\log \pi(\sB)+\overline \Xi_{1}+\overline \Xi_{2}+n\beta a_{1}\pa{e_{3}\ell(\overline P\et,\PAP)+e_{4}\ray-e_{8}2^{J}\ray}\\
&= \xi +2\overline \Xi_{1}+\overline \Xi_{2}+n\beta a_{1}\cro{\pa{e_{3}+e_{5}}\ell(\overline P\et,\PAP)+(e_{4}+e_{6}+e_{7})\ray}\\
&\quad - e_{8}n\beta a_{1}2^{J}\ray.
\end{align*}
The right-hand side is not larger than $-\xi$ provided that 
%Align
\begin{align}\label{J-thm2}
2^{J}&\ge \frac{1}{e_{8}}\cro{\frac{2\xi+2\overline \Xi_{1}+\overline \Xi_{2}}{n\beta a_{1}\ray}+\cro{\pa{e_{3}+e_{5}}\frac{\ell(\overline P\et,\PAP)}{\ray}+e_{4}+e_{6}+e_{7}}}.
\end{align}
%hat 
Using the fact that $\ray_{n}(\beta,\PAP)\ge 1/(n \beta a_{1})$, with the choice 
\[
\ray=\ell(\overline P\et,\PAP)+\ray_{n}(\beta,\PAP)+\frac{2\xi}{n\beta a_{1}}\ge \ell(\overline P\et,\PAP)+ \frac{1+2\xi}{n \beta a_{1}}\ge \frac{1}{n\beta a_{1}}, 
\]
the right-hand side of~\eref{J-thm2} satisfies 
%Align
\begin{align*}
 &\frac{1}{e_{8}}\cro{\frac{2\xi+2\overline \Xi_{1}+\overline \Xi_{2}}{n\beta a_{1}\ray}+\cro{\pa{e_{3}+e_{5}}\frac{\ell(\overline P\et,\PAP)}{\ray}+e_{4}+e_{6}+e_{7}}}\\
 &\le \frac{1}{e_{8}}\cro{2\overline \Xi_{1}+\overline \Xi_{2}+e_{4}+e_{6}+e_{7}+\frac{(e_{3}+e_{5})\vee 1}{\ray}\pa{\ell(\overline P\et,\PAP)+\frac{2\xi}{n\beta a_{1}}}}\\
 &\le \frac{2\overline \Xi_{1}+\overline \Xi_{2}+e_{4}+e_{6}+e_{7}+(e_{3}+e_{5})\vee 1}{e_{8}}.
 \end{align*}
Inequality~\eref{J-thm2} holds for $J\in\N$ such that 
\[
2^{J}\ge \frac{2\overline \Xi_{1}+\overline \Xi_{2}+e_{4}+e_{6}+e_{7}+(e_{3}+e_{5})\vee 1}{e_{8}}\vee 1>2^{J-1},
\]
and we may take 
%Equation
\begin{align}
\kappa_{0}&= \tau\cro{\frac{2\cro{2\overline \Xi_{1}+\overline \Xi_{2}+e_{4}+e_{6}+e_{7}+(e_{3}+e_{5})\vee 1}}{e_{8}}\vee 1+1}\label{def-Km0}\\
&\ge \tau\pa{2^{J}+1}.\nonumber
\end{align}
In complements to constants listed at the end of the proof of Theorem~\ref{main1}, we recall that 
\[
\Lambda_{1}=\tau\phi(\beta l_{1}),\quad \Lambda_{2}=\tau\phi(\beta l_{2})
\]
\begin{align*}
\overline \cc_{1}&=\cc_{0}-\tau \Lambda_{1}\beta \frac{a_{2}(\cc^{2}+\cc_{1}^{2})}{a_{1}}, &\overline \cc_{2}&=\cc-\tau\Lambda_{1} \beta \frac{a_{2}\cc^{2}}{a_{1}},&\overline \cc_{3}&=\cc_{2}-\tau\Lambda_{2} \beta \frac{a_{2}\cc_{2}^{2}}{a_{1}},
\end{align*}
%
%Align
\begin{align*}
e_{3}&=e_{0}+2\Lambda_{1} \beta \frac{a_{2}\pa{\cc^{2}+\cc_{1}^{2}}}{a_{1}}, & e_{4}&=\frac{1}{a_{1}}\cro{\tau \cc_{1}a_{0}+\Lambda_{1} \beta a_{2}\cc_{1}^{2}},\\
e_{5}&=e_{1}+\cc_{2}+2\Lambda_{2} \beta \frac{a_{2}\pa{\cc_{2}^{2}+\cc_{1}^{2}}}{a_{1}}, & 
e_{6}&=\tau\cc_{0}'+\Lambda_{2} \beta \frac{a_{2}(\cc_{2}^{2}+\cc_{1}^{2})}{a_{1}},\\
e_{7}&=\frac{1}{a_{1}}\cro{\tau \cc_{1}a_{0}+\Lambda_{2} \beta a_{2}\cc_{1}^{2}},& e_{8}&=\tau^{-1} \overline\cc_{1}-2\gamma,\\
\end{align*}
and 
%Align
\begin{align*}
\overline \Xi_{1}&=\log\cro{1+\frac{\exp\cro{-\pa{\tau^{-1}( \overline\cc_{2}\wedge  \overline\cc_{3})-\gamma}}}{1-\exp\cro{-\pa{\tau^{-1}( \overline\cc_{2}\wedge  \overline\cc_{3})-2\gamma}}}},\\ 
\overline \Xi_{2}&=-\gamma +\log\cro{\frac{1}{1-\exp\cro{-e_{8}}}}.\\
\end{align*}

\subsection{Proof of Theorem~\ref{main1b}}
Let us take $\ray\ge \eps$ and set $\varpi=2\xi+1$ so that 
\[
\pi\pa{\co{\sB}}\le \pi\pa{\co{\sB}(\PAP,\eps)}\le e^{-\varpi}\pi\pa{\sB(\PAP,\eps)}\le e^{-\varpi}\pi\pa{\sB}.
\]
In order to prove the first part, let us go back to the proof of Theorem~\ref{main1}. Clearly, 
%Align
\begin{align*}
\int_{\sM}\exp\cro{-\tau^{-1}\cc n \beta a_{1}\ell(\PAP,Q')}d\pi(Q')\le 1=\pi\pa{\sB}+\pi\pa{\co{\sB}}\le \pi\pa{\sB}(1+e^{-\varpi})
\end{align*}
and similarly, 
%Align
\begin{align*}
\int_{\sM}\exp\cro{-\tau^{-1}\cc_{2}n \beta a_{1}\ell(\PAP,Q')}d\pi(Q')\le \pi\pa{\sB}(1+e^{-\varpi}).
\end{align*}
Inequalities \eref{eq-th01-11b} and \eref{eq-th01-11} are therefore satisfied with $\Xi_{1}=\log(1+e^{-1})$. 
Moreover, 
\begin{align*}
&\int_{\co{\sB}(\PAP,2^{J}\ray)}\exp\cro{-\tau^{-1}\cc_{0} n\beta a_{1} \ell(\PAP,P)}d\pi(P)\\
&\le \exp\cro{-\tau^{-1}\cc_{0} n\beta a_{1}2^{J}\ray}\pi\pa{\co{\sB}}\le \pi\pa{\sB} \exp\cro{-\varpi-\tau^{-1}\cc_{0} n\beta a_{1}2^{J}\ray}.
\end{align*}
We deduce from~\eref{eq-th01-12b} that 
%Align
\begin{align*}
&\int_{\co{\sB}(\PAP,2^{J}\ray)}\E\cro{\exp\cro{-\beta \gT(\bsX,P)}}d\pi(P)\\
&\le \exp\cro{\Xi_{1}+n\beta \pa{ e_{0}a_{1}\ell(\overline P\et,\PAP)+\tau\cc_{1}a_{0}\ray+\frac{l_{1}^{2} \beta}{8}}}\\
&\quad \times \pi\pa{\sB} \exp\cro{-\varpi-\tau^{-1}\cc_{0} n\beta a_{1}2^{J}\ray},
\end{align*}
and consequently, 
%Align
\begin{align*}
&\log\int_{\co{\sB}(\PAP,2^{J}\ray)}\E\cro{\exp\cro{-\beta \gT(\bsX,P)}}d\pi(P)\\
&\le \log \pi\pa{\sB} + \Xi_{1}-\varpi\\
&\quad +n\beta\cro{e_{0}a_{1}\ell(\overline P\et,\PAP)+\tau\cc_{1}a_{0}\ray+\frac{l_{1}^{2} \beta}{8}-\tau^{-1}\cc_{0}  a_{1}2^{J}\ray}.
\end{align*}
Using the definitions~\eref{def-z} of $z$ and~\eref{def-Delta2b} of $\Delta_{2}$, we deduce that 
%Align
\begin{align*}
&\log\cro{\frac{1}{z}\int_{\co{\sB}(\PAP,2^{J}\ray)}\E\cro{\exp\cro{-\beta \gT(\bsX,P)}}d\pi(P)}\\
&\le \xi+\log\frac{1}{\pi(\sB)}+\Xi_{1}+n\beta \Delta_{2}+\log \pi\pa{\sB} + \Xi_{1}-\varpi\\
&\quad +n\beta\cro{e_{0}a_{1}\ell(\overline P\et,\PAP)+\tau\cc_{1}a_{0}\ray+\frac{l_{1}^{2} \beta}{8}-\tau^{-1}\cc_{0}a_{1}2^{J}\ray}\\
&=\xi+2\Xi_{1}+n\beta\cro{\pa{e_{1}+\cc_{2}}a_{1}\ell(\overline P\et,\PAP)+e_{1}a_{1}\ray+\frac{l_{2}^{2}\beta}{8}}-\varpi\\
&\quad +n\beta\cro{e_{0}a_{1}\ell(\overline P\et,\PAP)+\tau\cc_{1}a_{0}\ray+\frac{l_{1}^{2} \beta}{8}-\tau^{-1}\cc_{0}  a_{1}2^{J}\ray}\\
&=\xi+2\Xi_{1}-\varpi\\
&\quad +n\beta a_{1}\cro{C_{1}\ell(\overline P\et,\PAP)+C_{2}\ray+\frac{(l_{1}^{2}+l_{2}^{2}) \beta}{8 a_{1}}-\tau^{-1}\cc_{0} 2^{J}\ray},
\end{align*}
where the constants $C_{1}$ and $C_{2}$ are the same as those defined in the proof of Theorem~\ref{main1}. If we choose $r=\ell(\overline P\et,\PAP)\vee (\beta/a_{1})\vee \eps$ and $J$ such that $\tau^{-1}c_{0}2^{J}\ge C_{1}+C_{2}+(l_{1}^{2}+l_{2}^{2})/8$, we obtain that 
%Align
\begin{align*}
\log\cro{\frac{1}{z}\int_{\co{\sB}(\PAP,2^{J}\ray)}\E\cro{\exp\cro{-\beta \gT(\bsX,P)}}d\pi(P)}\le \xi+2\Xi_{1}-\varpi\le -\xi
\end{align*}
since $\varpi=2\xi+1\ge 2(\xi+\Xi_{1})$. We conclude as in the proof of Theorem~\ref{main1}.

In order to prove the second part of Theorem~\ref{main1b}, we go back to the proof of Theorem~\ref{main2}. The arguments are similar. As before,  
%Align
\begin{align*}
\int_{\sM}\exp\cro{-\tau^{-1} \overline\cc_{2}n\beta a_{1}\ell(\PAP,Q')}d\pi(Q')&\le  \pi\pa{\sB}(1+e^{-\varpi})
\end{align*}
and 
\begin{align*}
\int_{\sM}\exp\cro{-\tau^{-1} \overline\cc_{3}n\beta a_{1}\ell(\PAP,Q')}d\pi(Q')&\le  \pi\pa{\sB}(1+e^{-\varpi}).
\end{align*}
Inequalities \eref{eq-th02-11b} and \eref{eq-th02-11} are therefore both satisfied with $\overline \Xi_{1}=\log(1+e^{-1})$. Moreover 
\begin{align*}
& \int_{\co{\sB}(\PAP,2^{J}\ray)}\exp\cro{-\tau^{-1} \overline\cc_{1} n\beta a_{1} \ell(\PAP,P)}d\pi(P)\\
&\le \pi\pa{\sB}\exp\cro{-\varpi-\tau^{-1} \overline\cc_{1} n\beta a_{1}2^{J}\ray},
\end{align*}
and we deduce from \eref{eq-th02-12b} that 
\begin{align*}
&\int_{\co{\sB}(\PAP,2^{J}\ray)}\E\cro{\exp\cro{-\beta \gT(\bsX,P)}}d\pi(P)\\
&\le \exp\cro{\overline \Xi_{1}+n\beta a_{1}\pa{e_{3}\ell(\overline P\et,\PAP)+e_{4}\ray}}\\
&\quad \times \int_{\co{\sB}(\PAP,2^{J}\ray)}\exp\cro{-\tau^{-1} \overline \cc_{1}n\beta a_{1} \ell(\PAP,P)}d\pi(P)\\
&\le \pi\pa{\sB}\exp\cro{\overline \Xi_{1}+n\beta a_{1}\cro{e_{3}\ell(\overline P\et,\PAP)+e_{4}\ray-\tau^{-1} \overline\cc_{1} 2^{J}\ray}-\varpi}.
\end{align*}
Using the definition~\eref{def-zb} of $z$, we deduce that 
%Align
\begin{align*}
&\log\cro{\frac{1}{z}\int_{\co{\sB}(\PAP,2^{J}\ray)}\E\cro{\exp\cro{-\beta \gT(\bsX,P)}}d\pi(P)}\\
&\le \xi+\log\frac{1}{\pi(\sB)}+\overline \Xi_{1}+n\beta a_{1}\cro{e_{5}\ell(\overline P\et,\PAP)+(e_{6}+e_{7})\ray}\\
&\quad +\log \pi\pa{\sB}+ \overline \Xi_{1}+n\beta a_{1}\cro{e_{3}\ell(\overline P\et,\PAP)+e_{4}\ray-\tau^{-1} \overline\cc_{1} 2^{J}\ray}-\varpi\\
&=\xi+2\overline \Xi_{1}-\varpi\\
&\quad +n\beta a_{1}\cro{(e_{3}+e_{5})\ell(\overline P\et,\PAP)+(e_{4}+e_{6}+e_{7})\ray-\tau^{-1} \overline\cc_{1} 2^{J}\ray}.\end{align*}
Taking $r=\ell(\overline P\et,\PAP)\vee \eps \ge \eps$ and $J\ge 0$ such that 
\[
\tau^{-1} \overline\cc_{1} 2^{J}\ge e_{3}+e_{5}+ e_{4}+e_{6}+e_{7}
\]
we obtain that 
\[
\log\cro{\frac{1}{z}\int_{\co{\sB}(\PAP,2^{J}\ray)}\E\cro{\exp\cro{-\beta \gT(\bsX,P)}}d\pi(P)}\le -\xi
\]
and we conclude as before.

\section{Other proofs}\label{sct-OP}
\subsection{Proof of Lemma~\ref{lem-gamma}}\label{sect-pf-lem-gamma}
Let $Y$ be a random variable with gamma distribution $\gamma(s,1)$. Since $\sigma Y\sim\gamma(s,\sigma)$, it is sufficient to prove the result for $\sigma=1$. Using the inequality $\log(1-x)\ge -x/(1-x)$ which holds for all $x\in [0,1)$, we obtain that 
\[
\log\E\cro{e^{\beta(Y-s)}}=-s\cro{\log(1-\beta)+\beta}\le \frac{s \beta^{2}}{1-\beta}\quad \text{for all $\beta\in [0,1)$.}
\]
Applying Lemma~8.2 in Birg\'e~\citeyearpar{MR1836557} with $a=\sqrt{s}$ and $b=1$, we obtain that 
\[
\P\cro{Y\ge s+2\sqrt{s \xi}+\xi}\le e^{-\xi}\quad \text{for all $\xi\ge 0$}
\]
which proves~\eref{lem-gamma00}.  Let us now turn to the lower bound. For $x\ge 0$, let us set
\[
g(x)=x-\log(1+x)\le \pa{\frac{x^{2}}{2}}\wedge x.
\]
For all $t,u\ge 0$,
%Align
\begin{align*}
\int_{t+u}^{+\infty}x^{t}e^{-x}dx&=\int_{u}^{+\infty}(t+y)^{t}e^{-t-y}dy=t^{t}e^{-t}\int_{u}^{+\infty}e^{-tg(y/t)}dy\\
&\ge t^{t}e^{-t}\pa{\int_{u}^{+\infty}e^{-y^{2}/(2t)}dy\vee \int_{u}^{+\infty}e^{-y}dy}\\
&=t^{t}e^{-t}\cro{\pa{\sqrt{2\pi t}\;\overline F\pa{\frac{u}{\sqrt{t}}}}\vee e^{-u}},
\end{align*}
where $\overline F(z)=\P\cro{\cN(0,1)\ge z}$ for all $z\in\R$. Besides, it follows from Whittaker and Watson~\citeyearpar{MR1424469}[page 253] that 
%Equation
\begin{equation}\label{eq-gamma01}
t^{t-1/2}e^{-t}\sqrt{2\pi}\le \Gamma(t)\le  t^{t-1/2}e^{-t}\sqrt{2\pi}\exp[1/(12t)].
\end{equation}
Taking $t=s-1>0$, we deduce that 
%Align
\begin{align*}
\P\cro{Y\ge t+u}&=\frac{1}{\Gamma(t+1)}\int_{t+u}^{+\infty}x^{t}e^{-x}dx=\frac{1}{t\Gamma(t)}\int_{t+u}^{+\infty}x^{t}e^{-x}dx\\
&\ge\cro{\overline F\pa{\frac{u}{\sqrt{t}}}e^{-1/(12t)}}\vee \cro{\frac{e^{-u-1/(12t)}}{\sqrt{2\pi t}}}.
\end{align*}
Using the fact that $\overline F(\overline \Phi^{-1}(z))=e^{-z}$ for all $z\ge 0$, we obtain that for the choice 
\[
u=\cro{\sqrt{t}\;\overline \Phi^{-1}\pa{\xi-\frac{1}{12t}}}\vee \log\pa{\frac{e^{\xi-1/(12t)}}{\sqrt{2\pi t}}},
\]
which is nonnegative for $\xi\ge \log 2+1/(12t)$, the quantity $\P\cro{Y\ge t+u}$ is at least $e^{-\xi}$, which proves \eref{lem-gamma01}.

\subsection{Proof of Theorem~\ref{prop-credibleset}}\label{sect-pf-prop-credibleset}
Throughout this proof, $a_{0}=2$, $a_{1}=3/16$, $\beta=2\gamma=1/500$ and $\kappa$ denotes a positive numerical constant that may vary from line to line. It follows from Corollary~\ref{cor-odgh} that for $n$ large enough, $r_{n}(\beta,P_{\theta\et})\le r_{n}\et=\kappa k/n$. Applying our Corollary~\ref{cor-hell} with $\ell=h^{2}$ (and $2\xi$ in place of $\xi$), we obtain that for $n$ large enough, with a probability at least $1-2e^{-\xi}$, 
\[
1-e^{-\xi}\le \widehat \nu_{\bsX}^{h}\pa{\ac{\gtheta\in \Theta,\; h^{2}(\gtheta,\gtheta\et)\le r_{n}(\xi)}} \text{ with } r_{n}(\xi)=\frac{\kappa(k+\xi)}{n}.
\]
We know by Proposition~\ref{prop-regulstat} that under the assumptions of Corollary~\ref{cor-odgh}, Assumption~\ref{ass-para-1}-\ref{ass-para-1i} is satisfied with $\ab{\cdot}_{*}$ given by~\eref{eq-normh} and $\eps=1/2$ for $n$ large enough. This implies that for $n$ large
\[
\ac{\gtheta\in \Theta,\; h^{2}(\gtheta,\gtheta\et)\le r_{n}(\xi)}\subset \ac{\gtheta\in \Theta,\; \ab{\gtheta-\gtheta\et}_{*}^{2}\le 2r_{n}(\xi)},
\]
which leads to \eref{eq-credibleset}.

\subsection{Proof of Proposition~\ref{Mbeta-trans}}\label{Pf-Mbeta-trans}
Let us denote by $F_{\sigma}$ the distribution function of $\nu_{\sigma}$. Throughout this proof, we fix some $\theta\et \in [-\sigma t, \sigma t]$. Our aim is to prove that $P_{\theta\et}$ belongs to $\sM(\overline \beta)$. 

Since the total variation distance is translation invariant, $\norm{P_{\theta}-P_{\theta\et}}=\norm{P_{\theta-\theta\et}-P_{0}}=\norm{P_{\theta\et-\theta}-P_{0}}$ and consequently, for all $r\in [0,1)$, 
\[
\ac{\theta\in \Theta,\; \norm{P_{\theta}-P_{\theta\et}}\le r}=\ac{\theta\in \Theta,\; \ab{\theta\et-\theta}\le \varphi(r)}\quad \text{for all $r\in [0,1)$}
\]
while for $r\ge 1$, $\ac{\theta\in \Theta,\; \norm{P_{\theta}-P_{\theta\et}}\le r}=\Theta=\R$.

We set  $r_{0}=\sup\{r>0,\; \varphi(r)\le \sigma t\}$ and distinguish between two cases.

\noindent
{\bf Case 1:} Assume $r_{0}\le 1/4$. For all $r<r_{0}$, $\varphi(r)<\sigma t$, $2r<1$, and since $q$ is symmetric, positive and decreasing on $\R_{+}$, 
%Align
\begin{align*}
\frac{\pi(\sB(P_{\theta\et},2r))}{\pi(\sB(P_{\theta\et},r))}&=\frac{\nu_{\sigma}\pa{\ac{\theta\in\R,\; \norm{P_{\theta}-P_{ \theta\et}}\le 2r}}}{\nu_{\sigma}\pa{\ac{\theta\in\R,\; \norm{P_{\theta}-P_{\theta\et}}\le r}}}\\
&=\frac{\nu_{\sigma}\pa{\ac{\theta\in\R,\;\ab{\theta-\theta\et}\le \varphi(2r)}}}{\nu_{\sigma}\pa{\ac{\theta\in\R,\; \ab{\theta-\theta\et}\le \varphi(r)}}}\le \frac{2q_{\sigma}(0)\varphi(2r)}{2q_{\sigma}(|\theta\et|+\varphi(r))\varphi(r)}\\
&\le   \frac{q_{\sigma}(0)\varphi(2r)}{q_{\sigma}(|\theta\et|+\sigma t)\varphi(r)}\le \frac{q_{\sigma}(0)\varphi(2r)}{q_{\sigma}(2\sigma t)\varphi(r)}\\
&=\frac{q(0)\varphi(2r)}{q(2 t)\varphi(r)}\le  \frac{\overline \Gamma }{q(2 t)}.
\end{align*}
For all $r_{0}<r<1$, $|\theta\et|\le \sigma t< \varphi(r)$, hence $F_{\sigma}(|\theta\et|-\varphi(r))\le F_{\sigma}(0)=1/2$ and $F_{\sigma}(|\theta\et|+\varphi(r))\ge F_{\sigma}(\varphi(r))\ge F_{\sigma}(\sigma t)= F_{1}(t)\ge 3/4$ under our assumption on $t$. 
Consequently, 
%Align
\begin{align*}
\frac{\pi(\sB(P_{\theta\et},2r))}{\pi(\sB(P_{\theta\et},r))}&\le \frac{1}{\nu_{\sigma}\pa{\ac{\theta\in\R,\; \ab{\theta-\theta\et}\le \varphi(r)}}}\\
&=\frac{1}{F_{\sigma}\pa{|\theta\et|+\varphi(r)}-F(|\theta\et|-\varphi(r))}\\
&\le \frac{1}{3/4-1/2}=  4.
\end{align*}
Note that the result also holds for $r=r_{0}$ by letting $r$ decrease to $r_{0}$.

\noindent
{\bf Case 2:} Assume that $r_{0}>1/4$. Then $\varphi(1/4)\le \sigma t$ and arguing as before, we obtain that for all $r\le 1/4<r_{0}$, 
%Align
\begin{align*}
\frac{\pi(\sB(P_{\theta\et},2r))}{\pi(\sB(P_{\theta\et},r))}&\le \frac{2q_{\sigma}(0)\varphi(2r)}{2q_{\sigma}(|\theta\et|+\varphi(r))\varphi(r)}=\frac{q_{\sigma}(0)\varphi(2r)}{q_{\sigma}(|\theta\et|+\varphi(1/4))\varphi(r)}\\
&\le \frac{q_{\sigma}(0)\varphi(2r)}{q_{\sigma}(2\sigma t)\varphi(r)}\le \frac{\overline \Gamma }{q(2 t)}.
\end{align*}

For all $r\in (1/4,1)$, $\varphi(r)\ge \varphi(1/4)$ and
\begin{align*}
\frac{\pi(\sB(P_{\theta\et},2r))}{\pi(\sB(P_{\theta\et},r))}&\le \frac{1}{\nu_{\sigma}\pa{\ac{\theta\in\R,\; \ab{\theta-\theta\et}\le \varphi(r)}}}\\
&\le\frac{1}{\nu_{\sigma}\pa{\ac{\theta\in\R,\; \ab{\theta-\theta\et}\le \varphi(1/4)}}}\\
&\le \frac{1}{2q_{\sigma}(|\theta\et|+\varphi(1/4))\varphi(1/4)}\\
&\le \frac{1}{2q_{\sigma}(2\sigma t)\varphi(1/4)}\le \frac{\overline \Gamma \sigma}{q(2t)}.
\end{align*}

We obtain that in any case, for all $r\in (0,1)$ and $\theta\et\in [-\sigma t,\sigma t]$,
%Align
\begin{align}
\log\pa{\frac{\pi(\sB(P_{\theta\et},2r))}{\pi(\sB(P_{\theta\et},r))}}\le \max\ac{\log\pa{\frac{\overline \Gamma \pa{\sigma \vee 1}}{q(2 t)}},\log 4}.\label{eq-fonda-r}
\end{align}
The inequality is also clearly true for $r\ge 1$ since then $\pi(\sB(P_{\theta\et},2r))=\pi(\sB(P_{\theta\et},r))=1$. Hence,  for all $r\ge a_{1}^{-1}\beta$
%Align
\begin{align*}
\frac{1}{n \gamma a _{1} r}\log\pa{\frac{\pi(\sB(P_{\theta\et},2r))}{\pi(\sB(P_{\theta\et},r))}}&\le \frac{1}{n \gamma \beta}\sup_{r>0}\log\pa{\frac{\pi(\sB(P_{\theta\et},2r))}{\pi(\sB(P_{\theta\et},r))}}\\
&\le \frac{1}{n \gamma \beta}\max\ac{\log\pa{\frac{\overline \Gamma \pa{\sigma \vee 1}}{q(2 t)}},\log 4}.
\end{align*}
The right-hand side is not larger than $\beta$ provided that it satisfies~\eref{eq-beta} and this lower bound is not smaller than $1/\sqrt{n}$ since $\gamma\le 1$. We conclude by using \eref{def-Mbetab}.

\subsection{Proof of Proposition~\ref{prop-priorex3}}\label{Sect-PLemex30}
Under our assumption on $q$, Assumption~\ref{Hypo-Trans} is satisfied and 
\[
\overline \Gamma=2^{1/s}\max\ac{q(0),2^{(1/s)-1}}.
\]
Let $t=(|\theta|/\sigma)\vee t_{0}$. Then, $\theta\in [-\sigma t,\sigma t]$, $\nu_{1}([t,+\infty))\le 1/4$ and inequality \eref{eq-fonda-r} holds true. We deduce from \eref{df-epsn} that 
\[
r_{n}(\beta,P_{\theta})\le  \frac{1}{\gamma n a_{1}\beta }\max\ac{\log\pa{\frac{\overline \Gamma \pa{\sigma \vee 1}}{q(2 t)}},\log 4}
\]
and the result follows from our specific choices of $a_{1},\gamma$ and $\beta$.

\subsection{Proof of Corollary~\ref{cor-Entrop}}\label{Proof-cor-Entrop}
We set for short $\Theta=\Theta[\eta,\delta]$ with the parameters $\eta$ and $\delta$ defined by~\eref{eq-etan} and \eref{eq-deltan} respectively and also define 
%Equation
\begin{equation}\label{def-Jn}
J_{n}=\exp\cro{\frac{(K^{2}-1)\gamma\tau^{4}a_{1}^{2} n\eta_{n}^{2}}{2(k+1)}}
\end{equation}
so that $\sM_{n}(K)$ contains the elements $P=P_{(p,\gm,\sigma)}$ of $\sM$ such that 
\[
|\log\sigma|\vee \ab{\frac{\gm}{\sigma}}_{\infty}\le \log(1+\delta)J_{n}.
\]
Hereafter we fix $P=P_{(p,\gm,\sigma)}\in\sM_{n}(K)$. There exist $\theta=\theta(P)=(\PAP,\overline \gm,\overline \sigma)\in \Theta$ with $\overline \sigma=(1+\delta)^{j_{0}}$, $\overline \gm=\overline \sigma\delta \gj$, $(j_{0},\gj)\in\Z\times \Z^{k}$ such that 
%Equation
\begin{equation}\label{eq-control00}
\frac{\overline \sigma}{(1+\delta)}\le \sigma<\overline \sigma\quad \text{and}\quad \overline m_{i}=j_{i}\overline \sigma\delta\le m_{i}<\overline m_{i}+\overline \sigma\delta,
\end{equation}
for all $i\in\{1,\ldots,k\}$. Consequently, 
%Align
\begin{align}
0\le \pa{1-\frac{\sigma}{\overline \sigma}}\le \frac{\delta}{1+\delta}< \delta\quad \text{and}\quad \ab{\frac{\gm-\overline \gm}{\overline \sigma}}_{\infty}\le \delta,\label{eq-approx00}
\end{align}
and we infer from \eref{eq-trans} and \eref{app00} and the fact that the total variation loss is translation and scale invariant that $P_{\theta}$ satisfies 
%Align
\begin{align*}
\ell\pa{P_{(p,\gm,\sigma)},P_{\theta}}&\le \ell\pa{P_{(p,\gm,\sigma)},P_{(\PAP,\gm,\sigma)}}+\ell\pa{P_{(\PAP,\gm,\sigma)},P_{(\PAP,\overline \gm,\overline \sigma)}}\\
&\le \ell\pa{P_{(p,\bs{0},1)},P_{(\PAP,\bs{0},1)}}+\ell\pa{P_{(\PAP,\bs{0},1)},P_{(\PAP,\frac{\overline \gm-\gm}{\sigma},\frac{\overline \sigma}{\sigma})}}\\
&\le \eta+\cro{A\pa{\ab{\frac{\gm-\overline \gm}{\overline \sigma}}_{\infty}^{s}+\pa{1-\frac{\sigma}{\overline \sigma}}^{s}}}\wedge 1\\
&\le \eta+2A\delta^{s}=2\eta.
\end{align*}

Besides, the parameters $(j_{0},\gj)\in\Z\times \Z^{k}$ can be controlled in the following way. Using that $\sigma\le \overline \sigma$, the inequality $\log(1+\delta)\le \delta$  and \eref{eq-approx00}, we obtain that for all $i\in\{1,\ldots,k\}$, 
%Align
\begin{align*}
\ab{j_{i}}&=\ab{\frac{\overline m_{i}}{\overline \sigma\delta}}=\frac{1}{\overline \sigma\delta}\ab{\overline m_{i}-m_{i}+m_{i}}\le \frac{1}{\overline \sigma\delta}\cro{\overline \sigma \delta+
\sigma\ab{\frac{m_{i}}{\sigma}}}\le 1+\frac{1}{\log(1+\delta)}\ab{\frac{m_{i}}{\sigma}}.
\end{align*}
Besides, 
%Align
\begin{align*}
j_{0}&=\frac{\log\overline \sigma}{\log(1+\delta)}=\frac{1}{\log(1+\delta)}\cro{-\log\pa{1+\frac{\sigma}{\overline \sigma}-1}+\log \sigma}\\
&\le \frac{1}{\log(1+\delta)}\cro{-\log\pa{1-\frac{\delta}{1+\delta}}+|\log \sigma|}\\
&=\frac{1}{\log(1+\delta)}\cro{\log\pa{1+\delta}+|\log \sigma|}\le 1+\frac{|\log \sigma|}{\log(1+\delta)}
\end{align*}
and using the inequality $\log(1+2x)\le 2\log(1+x)$, which holds for all $x\ge 0$, we obtain that
\begin{align*}
j_{0}&\ge \frac{\log \sigma}{\log(1+\delta)}\ge -\frac{|\log \sigma|}{\log(1+\delta)}\ge -\cro{1+\frac{|\log \sigma|}{\log(1+\delta)}}.
\end{align*}
Putting these inequalities together and using the fact that $P\in\sM_{n}(K)$, we get 
%Equation
\begin{equation}\label{eq-b001}
\ab{(j_{0},\gj)}_{\infty}\le 1+\frac{1}{\log(1+\delta)}\cro{|\log \sigma|\vee \ab{\frac{\gm}{\sigma}}_{\infty}}\le 1+J_{n}.
\end{equation}

For all $r>0$, $e^{-L_{\theta}}\le \pi\pa{\sB(P_{\theta},r)}\le 1$ and these two inequalities together with the definition~\eref{eq-etan} of $\eta$ and Assumption~\ref{hypo-entro} imply that for all $r>0$
%Align
\begin{align*}
\frac{\pi\pa{\sB(P_{\theta},2r)}}{\pi\pa{\sB(P_{\theta},r)}}&\le \exp\cro{L_{\theta}}\le \exp\cro{\widetilde D(\eta)+2\sum_{i=0}^{k}\cro{\frac{L}{2}+\log(1+|j_{i}|)}}\\
&\le \exp\cro{\gamma\tau^{4}a_{1}^{2} n\eta^{2}+(k+1)\cro{L+2\log(1+|(j_{0},\gj)|_{\infty})}}.
\end{align*}
Using~\eref{eq-b001}, the definition \eref{def-Jn} of $J_{n}$ and the fact that $\log(2+x)\le \log 3+\log x$ for all $x\ge 1$, we derive that
\begin{align*}
\frac{\pi\pa{\sB(P_{\theta},2r)}}{\pi\pa{\sB(P_{\theta},r)}}&\le \exp\cro{\gamma\tau^{4}a_{1}^{2} n\eta^{2}+(k+1)L+2(k+1)\log(2+J_{n})},\\
&\le \exp\cro{K^{2}\gamma\tau^{4}a_{1}^{2} n\eta^{2}+(k+1)\pa{L+\log 9}}
\end{align*}
and since $\gamma=1/6\le L'=L+\log 9<3.1$, 
%Align
\begin{align*}
\frac{1}{n\beta a_{1}}\le r_{n}(\beta,P_{\theta})&\le \frac{1}{\gamma n\beta a_{1}}\cro{K^{2}\gamma\tau^{4}a_{1}^{2} n\eta^{2}+(k+1)L'}\\
&=\frac{1}{a_{1}\beta}\cro{K^{2}\tau^{4}a_{1}^{2} \eta^{2}+\frac{(k+1)L'}{\gamma n}}.
\end{align*}
For the choice of $\beta=\beta_{n}$ given by~\eref{def-b00}, 
\[
\beta\ge \sqrt{K^{2}\tau^{4}a_{1}^{2} \eta^{2}+\frac{(k+1)L'}{\gamma n}}\ge \sqrt{\frac{k+1}{n}}\vee \frac{K\eta}{2}
\]
hence, $r_{n}(\beta,P_{\theta})\le a_{1}^{-1}\beta$ and $P_{\theta}\in \sM(\beta)$. This implies that 
%Align
\begin{align*}
\inf_{P'\in\sM(\beta)}\ell(\overline P\et,P')+a_{1}^{-1}\beta&\le \ell(\overline P\et,P_{\theta})+a_{1}^{-1}\beta\\
&\le  \ell(\overline P\et,P)+\ell(P,P_{\theta})+a_{1}^{-1}\beta\\
&\le  \ell(\overline P\et,P)+ 2\eta+\cro{K\tau^{2}\eta+\frac{1}{a_{1}}\sqrt{\frac{(k+1)L'}{\gamma n}}},
\end{align*}
and the result follows by applying Corollary~\ref{cor-TV} and by using the fact that $P$ is arbitrary in $\sM_{n}(K)$.

\subsection{Proof of Lemma~\ref{lem-app00}}
For all $p\in\cM_{0}$, $\sigma\ge 1$ and $\gm\in\R^{k}$, the supports of the functions $\gx\mapsto p(\gx/\sigma)$ and $\gx\mapsto p((\gx-\gm)/\sigma)$ are included in the set 
$\cK=[0,\sigma]^{k}\cup\{\gm+\gx,\; \gx\in [0,\sigma]^{k}\}$ the Lebesgue measure of which is not larger than $2\sigma^{k}$. Consequently, using~\eref{app00b}, we deduce that for all $p\in\cM_{0}$, $\sigma\ge 1$ and $\gm\in\R^{k}$,
%Align
\begin{align*}
&\norm{P_{(p,\bs{0},1)}-P_{(p,\gm,\sigma)}}\\
&\quad \le \norm{P_{(p,\bs{0},1)}-P_{(p,\bs{0},\sigma)}}+\norm{P_{(p,\bs{0},\sigma)}-P_{(p,\gm,\sigma)}}\\
&\quad =\frac{1}{2}\int_{\R^{k}}\ab{p(\gx)-\frac{1}{\sigma^{k}}p\pa{\frac{\gx}{\sigma}}}d\gx+\frac{1}{2\sigma^{k}}\int_{\R^{k}}\ab{p\pa{\frac{\gx}{\sigma}}-p\pa{\frac{\gx-\gm}{\sigma}}}d\gx\\
&\quad \le \frac{1}{2}\int_{\R^{k}}\ab{p(\gx)-\frac{1}{\sigma^{k}}p\pa{\gx}}d\gx+\frac{1}{2\sigma^{k}}\int_{\R^{k}}\ab{p(\gx)-p\pa{\frac{\gx}{\sigma}}}d\gx\\
&\quad \quad +\frac{1}{2\sigma^{k}}\int_{\R^{k}}\ab{p\pa{\frac{\gx}{\sigma}}-p\pa{\frac{\gx-\gm}{\sigma}}}d\gx\\
&\quad \le \frac{1}{2}\int_{\R^{k}}\ab{p(\gx)-\frac{1}{\sigma^{k}}p\pa{\gx}}d\gx+\frac{1}{2\sigma^{k}}\int_{[0,1]^{k}}\ab{p(\gx)-p\pa{\frac{\gx}{\sigma}}}d\gx\\
&\quad \quad +\frac{1}{2\sigma^{k}}\int_{[0,\sigma]^{k}\setminus[0,1]^{k}}\ab{p\pa{\frac{\gx}{\sigma}}}d\gx
+\frac{1}{2\sigma^{k}}\int_{\cK}\ab{p\pa{\frac{\gx}{\sigma}}-p\pa{\frac{\gx-\gm}{\sigma}}}d\gx\\
&\quad \le \frac{1}{2}\pa{1-\frac{1}{\sigma^{k}}}+\frac{1}{2\sigma^{k}}\int_{[0,1]^{k}}L_{1}\pa{1-\frac{1}{\sigma}}^{s}\ab{\gx}^{s}d\gx\\
&\quad \quad + \frac{1}{2}\int_{[0,1]^{k}\setminus[0,1/\sigma]^{k}}\ab{p(\gx)}d\gx+\frac{L_{1}}{2\sigma^{k}}\int_{\cK}\ab{\frac{\gm}{\sigma}}^{s}d\gx\\
&\quad \le \frac{1}{2}\pa{1-\frac{1}{\sigma^{k}}}+\frac{L_{1} k^{s/2}}{2\sigma^{k}}\pa{1-\frac{1}{\sigma}}^{s}+\frac{L_{0}}{2}\pa{1-\frac{1}{\sigma^{k}}}+L_{1}\ab{\frac{\gm}{\sigma}}^{s}\\
&\quad \le  \frac{1}{2}\cro{1+L_{1}k^{s/2}+L_{0}}\pa{1-\frac{1}{\sigma}}^{s}+L_{1}\ab{\frac{\gm}{\sigma}}^{s}
\end{align*}
and~\eref{app00} is therefore satisfied with $A=L_{1}\vee [(1+L_{1}k^{s/2}+L_{0})/2]$.

\subsection{Proof of Lemma~\ref{reg00}}\label{Proof-reg00}
By doing the change of variables $u=x-m$ in~\eref{eq-reg001} if ever necessary, we may assume with no loss of generality that $m>0$. Then, since $p$ is nonincreasing in $(0,+\infty)$ and vanishes elsewhere $p(x-m)\ge p(x)$ for all $x\ge m$ and $p(x)\ge p(x-m)=0$ for all $x\in (0,m)$. Consequently, 
%Align
\begin{align*}
\int_{\R}\ab{p(x)-p(x-m)}dx&=\int_{0}^{m}p(x)dx+\int_{m}^{+\infty}\cro{p(x-m)-p(x)}dx\\
&=2\int_{0}^{m}p(x)dx+\int_{m}^{+\infty}p(x-m)dx-\int_{0}^{+\infty}p(x)dx\\
&\le 2mB+1-1,
\end{align*}
and we obtain~\eref{eq-reg001}. 

Since $\sigma\ge 1$, $p(x/\sigma)\ge p(x)$ and $p(x)/\sigma\le p(x)$ for all $x>0$. Hence,
%Align
\begin{align*}
&\int_{\R}\ab{\frac{1}{\sigma}p\pa{\frac{x}{\sigma}}-p(x)}dx\\
&\le \int_{\R}\ab{\frac{1}{\sigma}p\pa{\frac{x}{\sigma}}-\frac{1}{\sigma}p(x)}dx+\int_{\R}\ab{\frac{1}{\sigma}p\pa{x}-p(x)}dx\\
&=\frac{1}{\sigma}\int_{\R}\pa{p\pa{\frac{x}{\sigma}}-p(x)}dx+\int_{\R}\pa{p(x)-\frac{1}{\sigma}p\pa{x}}dx\\
&=2\pa{1-\frac{1}{\sigma}},
\end{align*}
which leads to~\eref{eq-reg002}.

Finally, by combining~\eref{eq-reg001} and~\eref{eq-reg002} we deduce that for all $m\in\R$ and $\sigma\ge 1$
%Align
\begin{align*}
&\frac{1}{2}\int_{\R}\ab{\frac{1}{\sigma}p\pa{\frac{x-m}{\sigma}}-p(x)}dx\\
&=\frac{1}{2}\int_{\R}\ab{\frac{1}{\sigma}p\pa{\frac{x-m}{\sigma}}-\frac{1}{\sigma}p\pa{\frac{x}{\sigma}}}dx+\frac{1}{2}\int_{\R}\ab{\frac{1}{\sigma}p\pa{\frac{x}{\sigma}}-p(x)}dx\\
&=\frac{1}{2}\int_{\R}\ab{p\pa{u-\frac{m}{\sigma}}-p(u)}du+\frac{1}{2}\int_{\R}\ab{\frac{1}{\sigma}p\pa{\frac{x}{\sigma}}-p(x)}dx\\
&\le B\ab{\frac{m}{\sigma}}+\pa{1-\frac{1}{\sigma}}
\end{align*}
which yields to~\eref{eq-reg003}.

\subsection{Proof of Proposition~\ref{cor-sparcity}}\label{Proof-cor-sparcity} 
This proposition is a consequence of Corollary~\ref{cor-hell}. Let us first check that the assumptions of this corollary are satisfied.  
For all $S\in\sP$, the mapping $\gtheta\mapsto h(S,P_{\theta})$ is continuous because of \eref{cond-hell}. It is therefore measurable and it follows from the definition of the algebra $\cA$ that Assumption~\ref{Hypo-0} is satisfied. Since the mapping $(x,\gtheta)\mapsto p(x,\gtheta)$ is measurable, so are the mappings 
\[
\map{\gp}{(\R^{k}\times E\times E}{\R_{+}}{(x,\bst,\bst')}{(p_{\bst}(x),p_{\bst'}(x)).}
\]
and $ (x,\bst,\bst')\mapsto \psi\pa{\sqrt{p_{\bst'}(x)/p_{\bst}(x)}}$, since $\psi$ is measurable. We deduce that $(x,P,P')\mapsto t_{(P,P')}(x)$ is measurable on $(E\times \sM\times \sM,\cE\otimes \cA\otimes \cA)$ which proves that Assumption~\ref{Hypo-1}-\ref{cond-1} holds true. The requirements of Corollary~\ref{cor-hell} are therefore satisfied and we may apply it. In order to evaluate the quantity $r_{n}(\beta,P_{\gtheta})$ for $\gtheta\in \R^{k}$, we use the following lemma the proof of which is postponed to Section~\ref{Proof-lem-uni}.
\begin{lem}\label{lem-uni}
Let $\bst\in [-R,R]^{k}$. For all $m\subset \{1,\ldots,k\}$ and $r>0$  
%Equation
\begin{align*}
&\nu_{m}\pa{\ac{\bst'\in\R^{k},\, \ab{\bst'-\bst}_{\infty}\le r}}\\
&\quad =
\begin{cases}
\dps{\frac{1}{2^{|m|}}\prod_{i\in m}\cro{\pa{1-\frac{|\theta_{i}|}{R}}\wedge  \frac{r}{R}+\pa{1+\frac{|\theta_{i}|}{R}}\wedge  \frac{r}{R}}} & \text{if $\ab{\theta_{i}}\le  r$ for all $i\not\in m$ }\\
0 & \text{otherwise,}
\end{cases}
\end{align*}
with the convention $\prod_{\vide}=1$. In particular, if $\bst\in \Theta_{m}(R)$ and 
%Equation
\begin{equation}\label{eq-lem-uni00}
\nu_{m}\pa{\ac{\bst'\in\R^{k},\, \ab{\bst'-\bst}_{\infty}\le r}}\ge \frac{1}{2^{|m|}}\pa{\frac{r}{R}\wedge 1}^{|m|}
\end{equation}
and for all $K>1$
%Equation
\begin{equation}\label{eq-lem-uni}
\frac{\nu_{m}\pa{\ac{\bst'\in\R^{k},\, \ab{\bst'-\bst}_{\infty}\le Kr}}}{\nu_{m}\pa{\ac{\bst'\in\R^{k},\, \ab{\bst'-\bst}_{\infty}\le r}}}\le K^{|m|}.
\end{equation}
\end{lem}

Let us set $B=B_{k}$ for short and define $m\et$ as the subset of $\{1,\ldots,k\}$ that minimizes over those $m\subset \{1,\ldots,k\}$ the mapping 
\[
m\mapsto \inf_{\bst\in \Theta_{m}(R)}\ell(\overline P\et,P_{\bst})+\frac{|m|\log\pa{2kR(nB)^{1/s}}+1}{\gamma n\beta a_{1}}.
\]
Finally, let $\bst\et$ for some arbitrary element of $\Theta_{m\et}(R)$.
It follows from~\eref{cond-hell} and \eref{eq-lem-uni00} that for all $r>0$, 
%Align
\begin{align}
1&\ge \pi_{m}\pa{\sB(P_{\bst\et},r)}\nonumber\\
&=\nu_{m}\pa{\ac{\bst\in \R^{k},\; h^{2}(P_{\bst\et},P_{\bst})\le r}}\nonumber\\
&\ge \nu_{m}\pa{\ac{\bst\in \R^{k},\; \ab{\bst-\bst\et}_{\infty}\le (r/B)^{1/s}}}\nonumber\\
&\ge \frac{1}{2^{|m|}}\pa{\frac{(r/B)^{1/s}}{R}\wedge 1}^{|m|}\ge  \frac{1}{2^{|m|}}\pa{\frac{(r\wedge 1)^{1/s}}{RB^{1/s}}}^{|m|},\label{eq-cor2p2}
\end{align}
where the last inequality holds true under the assumption that $RB^{1/s}\ge 1$.

%
%For all $m\subset \{1,\ldots,k\}$ and $r\in (0,1]$, it follows from~\eref{cond-hell} that
%%Align
%\begin{align*}
%\pi_{m}\pa{\sB(P_{\bst\et},2r)}&=\nu_{m}\pa{\ac{\bst\in \R^{k},\; h^{2}(P_{\bst\et},P_{\bst})\le 2r}}\nonumber\\
%&\le \nu_{m}\pa{\ac{\bst\in \R^{k},\; 2A\pa{\ab{\bst-\bst\et}_{\infty}^{2s}\wedge 1}\le 2r}}\\
%&= \nu_{m}\pa{\ac{\bst\in \R^{k},\; \ab{\bst-\bst\et}_{\infty}\wedge 1\le (r/A)^{1/(2s)}}}.
%\end{align*}
%%
%Since $K=(B/A)^{1/(2s)}$, we deduce that for $r<A$
%%Align
%\begin{align}
%\pi_{m}\pa{\sB(P_{\bst\et},r)}&\le\nu_{m}\pa{\ac{\bst\in \R^{k},\; \ab{\bst-\bst\et}_{\infty}\le K(r/B)^{1/(2s)}}}.\label{eq-cor2p1}
%\end{align}
%%
%Similarly, for all $r\in (0,1)$
%%Align
%\begin{align}
%\pi_{m}\pa{\sB(P_{\bst\et},r)}
%&\ge \nu_{m}\pa{\ac{\bst\in \R^{k},\; \ab{\bst-\bst\et}_{\infty}\le (r/B)^{1/(2s)}}},\label{eq-cor2p2}
%\end{align}
%%
%while for $r\ge 1$, $\pi_{m}\pa{\sB(P_{\bst\et},r)}=1$. 
%
We deduce from~\eref{eq-cor2p2} that for all $r>0$
%Align
\begin{align}
&\frac{\pi\pa{\sB(P_{\bst\et},2r)}}{\pi\pa{\sB(P_{\bst\et},r)}}\le \frac{1}{\pi\pa{\sB(P_{\bst\et},r)}}\nonumber \\
&\le \frac{1}{\sum_{m\subset \{1,\ldots,k\}}e^{-L_{m}}\nu_{m}\pa{\ac{\bst\in \R^{k},\; \ab{\bst-\bst\et}_{\infty}\le (r/B)^{1/s}}}}\nonumber \\
&\le \frac{e^{L_{m\et}}}{\nu_{m\et}\pa{\ac{\bst\in \Theta_{m\et},\; \ab{\bst-\bst\et}_{\infty}\le (r/B)^{1/s}}}}\nonumber \\
&\le  \exp\cro{L_{m\et}+|m\et|\log\pa{\frac{2RB^{1/s}}{(r\wedge 1)^{1/s}}}}\nonumber \\
&=\exp\cro{|m\et|\log\pa{2kRB^{1/s}}+k\log\pa{1+\frac{1}{k}}+\frac{|m\et|}{s}\log\pa{\frac{1}{r}\vee 1}}.\label{eq-cor2p1}
\end{align}
Provided that 
\[
r\ge \frac{|m\et|\log\pa{2kR(nB)^{1/s}}+1}{\gamma n\beta a_{1}}\ge \frac{1}{n},
\]
we obtain 
%Align
\begin{align*}
&|m\et|\log\pa{2kRB^{1/s}}+k\log\pa{1+\frac{1}{k}}+\frac{|m\et|}{s}\log\pa{\frac{1}{r}\vee 1}\\
&\le  |m\et|\log\pa{2kRB^{1/s}}+k\log\pa{1+\frac{1}{k}}+|m\et|\log\pa{n^{1/s}}\\
&\le |m\et|\log\pa{2kR(nB)^{1/s}}+1\le \gamma n\beta a_{1}r
\end{align*}
and deduce from~\eref{eq-cor2p1} that $ r_{n}(\beta,P_{\bst\et})$ defined by~\eref{df-epsn} satisfies 
%Align
\begin{align*}
\frac{1}{n\beta a_{1}}\le r_{n}(\beta,P_{\bst\et})&\le \frac{|m\et|\log\pa{2kR(nB)^{1/s}}+1}{\gamma n\beta a_{1}}.
\end{align*}
Applying Corollary~\ref{cor-hell}, we obtain that for some numerical constant $\kappa_{0}'>0$, 
\[
\E\cro{\widehat \pi_{\bsX}\pa{\co{\sB}(\overline P\et,\kappa_{0}'\ray(m\et,\bst\et))}}\le 2e^{-\xi}
\]
with 
%Align
\begin{align*}
r(m\et,\bst\et)&=\ell(\overline P\et,P_{\bst\et})+\frac{|m\et|\log\pa{2kR(nB)^{1/s}}+\xi}{\gamma n\beta a_{1}}.
\end{align*}
Finally, the conclusion follows from the definition of $m\et$ and the fact that $\bst\et$ is arbitrary in $\Theta_{m\et}(R)$.

\subsection{Proof of Lemma~\ref{lem-uni}}\label{Proof-lem-uni}
Let $\theta\in\R$ and $\nu$ be the uniform distribution on $[-R,R]$. For all $\theta\in [-R,R]$ and $r>0$, 
%Align
\begin{align*}
\nu\pa{[\theta-r,\theta+r]}&=\frac{1}{2R}\cro{(\theta+r)\wedge R-(\theta-r)\vee (-R)}_{+}\\
&=\frac{1}{2R}\cro{(r+\theta)\wedge R+(r-\theta)\wedge R}_{+}\\
&=\frac{1}{2R}\cro{(r+|\theta|)\wedge R+(r-|\theta|)\wedge R}_{+}\\
&=\frac{1}{2}\cro{\pa{1-\frac{|\theta|}{R}}\wedge \frac{r}{R}+\pa{1+\frac{|\theta|}{R}}\wedge \frac{r}{R}}.
\end{align*}
Let now $\bst\in\R^{k}$ such that $\ab{\bst}_{\infty}\le R$. For all $m\subset \{1,\ldots,k\}$, $m\ne \vide$, 
%Align
\begin{align*}
\nu_{m}\pa{\ac{\bst'\in \Theta_{m},\; \ab{\bst'-\bst}_{\infty}\le r}}&= 0
\end{align*}
if there exists $i\not \in m$ such that $|\theta_{i}|>r$. Otherwise 
%Align
\begin{align*}
\nu_{m}\pa{\ac{\bst'\in \R^{k},\; \ab{\bst'-\bst}_{\infty}\le r}}&= \nu_{m}\pa{\ac{\bst'\in \Theta_{m},\; \max_{i\in m}\ab{\theta_{i}'-\theta_{i}}\le r}}\\
&=\prod_{i\in m}\nu\pa{\cro{\theta_{i}-r,\theta_{i}+r}}\\
&=\frac{1}{2^{|m|}}\prod_{i\in m}\cro{\pa{1-\frac{|\theta_{i}|}{R}}\wedge \frac{r}{R}+\pa{1+\frac{|\theta_{i}|}{R}}\wedge  \frac{r}{R}}.
\end{align*}
If $m=\vide$, 
\[
\nu_{\vide}\pa{\ac{\bst'\in \R^{k},\; \ab{\bst'-\bst}_{\infty}\le r}}=\1_{\ab{\bst}_{\infty}\le r}.
\]

Let us now turn to the proof of  \eref{eq-lem-uni}. Since $\bst\in \Theta_{m}(R)$, for all $K'\in\{1,K\}$
%Align
\begin{align*}
\nu_{m}&\pa{\ac{\bst'\in \R^{k},\; \ab{\bst'-\bst}_{\infty}\le K'r}}\\
&=\nu_{m}\pa{\ac{\bst'\in \Theta_{m},\; \max_{i\in m}\ab{\theta_{i}'-\theta_{i}}\le K'r}}\\
&=\prod_{i\in m}\nu\pa{[\theta_{i}-K'r,\theta_{i}+K'r]},
\end{align*}
It is therefore enough to show that for all $r>0$ and $\theta\in [0,R]$
\[
\Delta(r)=\frac{\nu\pa{\cro{\theta-Kr,\theta+Kr}}}{\nu\pa{\cro{\theta-r,\theta+r}}}\le K.
\]
This is what we do now by distinguishing between several cases.

When $\theta+Kr\le R$, $\theta-Kr\ge 2\theta-R\ge -R$ and consequently, $\Delta(r)=K$. When $\theta+Kr>R$ and $-R\le \theta-Kr$, 
\[
\Delta(r)=\frac{R-(\theta-Kr)}{(\theta+r)\wedge R-(\theta-r)}=
\begin{cases}
\dps{\frac{R-\theta+Kr}{R-\theta+r}}&\text{when $\theta+r>R$}\\
\ \\
\dps{\frac{R-\theta+Kr}{2r}}& \text{when $\theta+r\le R$,}
\end{cases}
\]
and the conclusion follows from the facts that $0\le R-\theta\le Kr$. When $\theta+Kr> R$ and $\theta-Kr< -R$, $r\ge (\theta+R)/K\ge R/K$, hence $R+r-\theta \ge 2R/K$ and $R\le  Kr$. Consequently, 
%Align
\begin{align*}
\Delta(r)&=\frac{2R}{(\theta+r)\wedge R-(\theta-r)\vee(-R)}\\
&=
\begin{cases}
\dps{\frac{2R}{2R}}=1&\text{when $\theta+r>R$ and $\theta-r<-R$}\\
\ \\
\dps{\frac{2R}{R+r-\theta}}\le K&\text{when $\theta+r>R$ and $\theta-r\ge -R$}\\
\ \\
\dps{\frac{2R}{2r}\le K} & \text{when $\theta+r\le R$},
\end{cases}
\end{align*}
which concludes the proof.

\subsection{Proof of Proposition~\ref{prop-odg}}\label{sect-pf-prop-odg}
Let $\eps$ be a small enough positive number. Since $q$ is continuous and positive at $\gtheta\et$ and since $\cK$ has a nonempty interior, 
there exists $z\et>0$ such that $\Theta\et=\cB_{*}(\gtheta\et,z\et)\subset \cK$, 
%Equation
\begin{equation}\label{eq-para-100}
0<\underline b\et\le q(\gtheta)\le \overline b\et \quad \text{with $\overline b\et/\underline b\et\le 1+\eps$},
\end{equation}
for all $\gtheta\in \Theta\et$ and 
%Equation
\begin{equation}\label{eq-para-101}
(1-\eps) |\gtheta-\gtheta\et|_{*}^{s}\le \ell(\gtheta,\gtheta\et)\le (1+\eps) |\gtheta-\gtheta\et|_{*}^{s}.
\end{equation}
In particular, $\nu(\Theta\et)>0$ and we may define the distribution $\nu\et=\nu(\cdot\cap \Theta\et)/\nu(\Theta\et)$ on $\Theta\et$ with density $q\et=q\1_{\Theta\et}/\nu(\Theta\et)$. Let $\sM\et=\{P_{\gtheta},\; \gtheta\in \Theta\et\}$ and $\pi\et$ be the prior on $\sM\et$ associated with $\nu\et$. The parameter space $\Theta\et$ is convex and it follows fom~\eref{eq-para-101} that $(\Theta\et,\gtheta\et,\ell, \nu\et)$ satisfy Assumption~\ref{Ass-param}-\ref{Ass-param-i} with $\overline a=1+\eps, \underline a=1-\eps$. Besides, it follows from \eref{eq-para-100} that the density $q\et$ satisfies condition \eref{eq-densborn} on $\Theta\et$. We may apply Proposition~\ref{casconv} and deduce that for the model $(\sM\et,\pi\et)$,  $r_{n}\et=r_{n}\et(\beta,P_{\gtheta\et})$ is not larger than $\kappa_{0}\et k/(\beta n)$ with 
\[
\kappa_{0}\et=\frac{1}{a_{1}\gamma}\ac{\cro{1+\frac{\log\cro{2(1+\eps)/(1-\eps)}}{s\log 2}}\log\pa{2(1+\eps)}}\vee 1<\frac{(1+s^{-1})}{a_{1}\gamma}
\]
for $\eps$ small enough. Consequently, for all $r\ge r_{n}\et$
%Align
\begin{align}
\pi\et\pa{\sB(P_{\gtheta\et},2r)}&=\frac{1}{\nu(\Theta\et)}\nu\pa{\{\gtheta\in \Theta,\; \ell(\gtheta,\gtheta\et)\le 2r\}\cap \Theta\et}\nonumber\\
&\le  \frac{\exp\pa{\gamma n\beta a_{1}r}}{\nu(\Theta\et)}\nu\pa{\{\gtheta\in \Theta,\; \ell(\gtheta,\gtheta\et)\le r\}\cap \Theta\et}\nonumber\\
&\le  \frac{\exp\pa{\gamma n\beta a_{1}r}}{\nu(\Theta\et)}\nu\pa{\{\gtheta\in \Theta,\; \ell(\gtheta,\gtheta\et)\le r\}}\label{eq-conf00}.
\end{align}
Let $r_{1}=[(z\et)^{s}\underline a_{\cK})\wedge \eta]/2$. If $r\in (0,r_{1})$ and the parameter $\gtheta\in \Theta$ satisfies $\ell(\gtheta,\gtheta\et)\le 2r$, then $\ell(\gtheta,\gtheta\et)<\eta$ and $\gtheta$ necessarily belongs to $\cK$ under Assumption~\ref{ass-para-1}-\ref{ass-para-1ii}. Applying \eref{eq-para-1ii} we deduce that for such a parameter $\gtheta\in \Theta$
\[
\underline a_{\cK}\ab{\gtheta-\gtheta\et}_{*}^{s}\le  \ell\pa{\gtheta,\gtheta\et}\le 2r < 2r_{1}\le \underline a_{\cK}(z\et)^{s},
\]
which implies that $\gtheta\in \Theta\et$. For $n$ large enough, $r_{n}\et=\kappa_{0}\et k/n<r_{1}$ and for $r\in (r_{n},r_{1})$ we may therefore write, using~\eref{eq-conf00}, 
%Align
\begin{align}
\pi\pa{\sB(P_{\gtheta\et},2r)}&=\nu\pa{\{\gtheta\in \Theta,\; \ell(\gtheta,\gtheta\et)\le 2r\}}\nonumber\\
&=\nu\pa{\{\gtheta\in \Theta,\; \ell(\gtheta,\gtheta\et)\le 2r\}\cap \Theta\et}\nonumber\\
&\le \exp\pa{\gamma n\beta a_{1}r}\nu\pa{\{\gtheta\in \Theta,\; \ell(\gtheta,\gtheta\et)\le r\}}\nonumber\\
&=\exp\pa{\gamma n\beta a_{1}r}\pi\pa{\sB(P_{\gtheta\et},r)}.\label{eq-conf02}
\end{align}
Since $q$ is bounded away from 0 in a neighbourhood of $\gtheta\et$, $\pi\pa{\sB(P_{\gtheta\et},r_{1})}>0$ and we may also write that for $n$ large enough and $r\ge r_{1}$
%Align
\begin{align}
\pi\pa{\sB(P_{\gtheta\et},r)}&\ge \pi\pa{\sB(P_{\gtheta\et},r_{1})}\nonumber\\
&=\exp\cro{\log \pi\pa{\sB(P_{\gtheta\et},r_{1})}+\gamma n\beta a_{1}r_{1}-\gamma n\beta a_{1}r_{1}}\nonumber\\
&\ge \exp\cro{-\gamma n\beta a_{1}r_{1}}\ge  \exp\cro{-\gamma n\beta a_{1}r}\nonumber\\
&\ge \exp\cro{-\gamma n\beta a_{1}r}\pi\pa{\sB(P_{\gtheta\et},2r)}.\label{eq-conf03}
\end{align}
Putting \eref{eq-conf02} and \eref{eq-conf03} together we obtain that for $n$ large enough
\[
\pi\pa{\sB(P_{\gtheta\et},2r)}\le \exp\pa{\gamma n\beta a_{1}r}\pi\pa{\sB(P_{\gtheta\et},r)}\quad \text{for all $r\ge r_{n}\et$}
\]
and consequently that $r_{n}(\beta,P_{\gtheta\et})\le r_{n}\et=\kappa_{0}\et k/n$.

\bibliography{biblio}

\end{document}